\documentclass{amsart}

\usepackage{amsmath}
\usepackage[usenames,dvipsnames,svgnames,table]{xcolor}
\usepackage[english]{babel}
\usepackage[utf8]{inputenc}
\usepackage{amsmath}
\usepackage{mathrsfs}
\usepackage{wrapfig}
\usepackage[T1]{fontenc}
\usepackage{amsfonts}
\usepackage{lineno}
\usepackage{amssymb}
\usepackage{amsthm}
\usepackage{xcolor}
\usepackage{graphicx}
\usepackage{appendix}
\usepackage{multicol}
\usepackage{multirow}
\usepackage{comment}
\usepackage{array}
\usepackage{booktabs}
\DeclareFontFamily{U}{BOONDOX-calo}{\skewchar\font=45 }
\DeclareFontShape{U}{BOONDOX-calo}{m}{n}{
  <-> s*[1.05] BOONDOX-r-calo}{}
\DeclareFontShape{U}{BOONDOX-calo}{b}{n}{
  <-> s*[1.05] BOONDOX-b-calo}{}
\DeclareMathAlphabet{\mathcalboondox}{U}{BOONDOX-calo}{m}{n}
\SetMathAlphabet{\mathcalboondox}{bold}{U}{BOONDOX-calo}{b}{n}
\DeclareMathAlphabet{\mathbcalboondox}{U}{BOONDOX-calo}{b}{n}

\usepackage{lineno}
\usepackage{todonotes}
\usepackage{color}

\usepackage{enumitem}
\setlist[itemize]{leftmargin=*}
\setlist[enumerate]{leftmargin=*}

\usepackage{parskip}

\newtheorem{theorem}{Theorem}

\newtheorem{corollary}[theorem]{Corollary}

\newtheorem{definition}[section]{Definition}

\newtheorem{lemma}[theorem]{Lemma}

\newtheorem{proposition}[theorem]{Proposition}

\title{Scaffold for the polyhedral\\embedding of cubic graphs.}
\usepackage[foot]{amsaddr}

\author[F. Aguilar]{Flor Aguilar $^1$}
\address{$^1$ Instituto de Matem\'aticas, UNAM }
\email{ $^1$ flor@matem.unam.mx}

\author[G.Araujo-Pardo]{Gabriela Araujo-Pardo $^2$}
\address{$^2$ Instituto de Matemáticas, UNAM}
\email{$^2$ garaujo@math.unam.mx}

\author[N. Garc\'{i}a-Col\'{i}n]{Natalia Garc\'{i}a-Col\'{i}n $^3$}
\address{$^3$ CONACYT Research Fellow - INFOTEC Centro de
  Investigación en Tecnologías de la Información y Comunicación,
  Mexico. Corresponding Author.} 
   \email{$^3$  natalia.garcia@infotec.mx}

\begin{document}
\maketitle
\begin{abstract}Let $G$ be a cubic graph and $\Pi$ be a polyhedral embedding of this graph. The \textit{extended graph}, $G^{e},$ of $\Pi$ is the graph whose set of vertices is $V(G^{e})=V(G)$ and whose set of edges $E(G^{e})$ is equal to $E(G) \cup \mathcal{S}$,  where $\mathcal{S}$ is constructed as follows: given two vertices $t_0$ and $t_3$ in $V(G^{e})$ we say $[t_0 t_3] \in \mathcal{S},$ if there is a $3$--path, $(t_0 t_1 t_2 t_3) \in G$ that is a $\Pi$-- facial subwalk of the embedding. We prove that there is a one to one correspondence between the set of possible extended graphs of $G$ and polyhedral embeddings of $G$.
\end{abstract}
\section{Introduction}
Our motivation for studying polyhedral embeddings of cubic graphs is twofold. On one hand, it is interesting as the natural alternative point of view on combinatorial characterisations of triangulations of surfaces \cite {Arocha}, given that the dual structure of a triangulation is precisely a polyhedral embedding of a 3-regular graph in a surface. On the other hand, graph embeddings of $3$-regular graphs are interesting in their own right as a plethora of papers on  the subject prove, particularly as related to Gr\"unbaum's  conjectured generalization of the four color theorem:  \emph{If a cubic graph admits a polyhedral embedding in an orientable surface, then it is $3$-edge colorable} \cite{BM2006}.

More precisely, in \cite {Arocha} they prove that the information on the size of the pairwise intersection of triangles in a triangulation suffices in order to determine its whole combinatorial structure. However, it is known that if we only have information on the pairs of triangles that intersect edge to edge (that is, the dual graph of the triangulation) then we cannot uniquely determine the whole incidence structure of the triangulation. For example, in \cite{BM2006} they prove that some connected cubic graphs can be embedded into more than one surface. 

In like manner, \emph{in this paper we prove that some additional combinatorial information suffices to uniquely determine a  polyhedral embedding of a $3$-regular graph in a surface with no boundary.} For a more precise statement of our main result (Theorem \ref{Contribution}) we need to introduce some terminology.

A topological map of a graph into a surface is called a \textit{graph embedding}. If we consider the graph $G$ together with its embedding $\Pi$, we say that $G$ is \textit{$\Pi$--embedded}. Each disjoint region of the complement of the image of an embedded graph is called a \textit{face of the embedding}. The closed walk in the underlying graph $G$ that corresponds to the boundary of a face is called a \textit{$\Pi$--facial walk}. The embedding $\Pi$ determines a \textit{set of $\Pi$--facial walks}. Each edge is either contained in two $\Pi$--facial walks or it appears twice in the same facial walk. If a $\Pi$--facial walk is a cycle, it is also called a \textit{$\Pi$--facial cycle.} Two embeddings of $G$ are \textit{equivalent} if they have the same set of facial walks, up to automorphisms of $G$.

Let $\mathcal{P}_1$ and $\mathcal{P}_2$ be distinct $\Pi-$facial walks. We say that $\mathcal{P}_1$ and $\mathcal{P}_2$ meet properly if the intersection of $\mathcal{P}_1$ and $\mathcal{P}_2$ is either empty, a single vertex, or an edge. In this paper we will consider only embeddings of cubic simple graphs whose facial walks meet properly, this latter characteristic defines a polyhedral embedding: 

\begin{definition}\label{poly}
$\Pi$ is said to be a \textit{polyhedral embbeding} of a graph $G$, if every $\Pi$-facial walk is a cycle and any two $\Pi-$facial cycles meet properly. 
\end{definition}
The set of all the $\Pi$--facial cycles is called \textit{$\Pi$--facial cycle system.}

An indication that there is some non-triviality in determining polyhedral embeddings is that a significant part of Ringel and Youngs' Map Color Theorem \cite{R1974} was to determine which complete graphs have such embeddings, even though for a complete graph $K_n$ (with $n \geq 5$) a polyhedral embedding is necessarily a triangulation. Furthermore, in \cite{BM2001_Existence} it is proven that the decision problem about the existence of polyhedral embeddings of a graph is NP-complete. The problem remains NP-complete even if it is restricted to the case of embeddings in orientable surfaces and it is required that the graph is $6$-connected. 

Concerning the uniqueness of the embedding, recall that the \textit{face-width} is defined as the minimum integer $r$ such that $G$ has $r$ facial walks whose union contains a cycle which is noncontractible on the surface. Whitney \cite{W1933} proved that every $3$-connected planar graph has an essentially unique embedding in the plane. Robertson and Vitray \cite{RV1990} extended the previous result to an arbitrary surface of genus $g$ by assuming that the  \emph{face-width} is at least $2g+3.$ Seymour and Thomas \cite{S1996} and Mohar \cite{B1995} improved the bound on the face width to $O(\frac{\log g}{ \log \log g})$. Moreover, Robertson and Vitray \cite{RV1990} proved the following result:
\begin{proposition}\label{layout}
An embedding of a graph $G$ is polyhedral if and only if $G$ is 3-connected and the embedding has face-width at least 3.
\end{proposition}
It is also known that for each surface $S$, there is a constant $\zeta= \zeta(S)$ such that every $3$-connected graph admits at most $\zeta$ embeddings of face width greater than three \cite{BM2001_Flexibility}, then it can be concluded that a graph may have many different polyhedral embedings in the same or different surfaces.

\subsection{The extended graph of an embedding}
The \textit{extended graph} of the polyhedral embedding $\Pi$, $G^{e}(\Pi),$ is the graph whose set of vertices is $V(G^{e}(\Pi))=V(G)$ and whose set of edges $E(G^{e}(\Pi))$ is equal to $E(G) \cup \mathcal{S}$.  We will call $\mathcal{S}$ the set of \emph{scaffold edges} and construct it as follows:
given two vertices $t_0$ and $t_3$ in $V(G)$, $[t_0 t_3] \in \mathcal{S},$ if there are vertices $t_1$ and $t_2$, different from $t_0$ and $t_3$, such that $(t_0 t_1 t_2 t_3)$ is a  $\Pi$-- facial subwalk. In such case, we say that the \emph{$3$--path of $G$ corresponding to} $[t_0 t_3],$  is $(t_0 t_1 t_2 t_3)$. Notice that the scaffold edges may be double (but not triple or more). That is, if $[t_0 t_3] \in \mathcal{S}$ and its corresponding path is $(t_0 t_1 t_2 t_3),$ there may be a second $3$--path between $t_0$ and  $t_3$, internally disjoint from $(t_0 t_1 t_2 t_3),$ say 
$(t_0 t'_1 t'_2 t_3),$ that \emph{also} corresponds to $[t_0 t_3].$ In this case we say $[t_0 t_3]$ is a \emph{double scaffold edge}, and we denote this by a double bracket $[[t_0 t_3]].$ It will become obvious later, in Proposition \ref{pizzageneralization}, that this only happens when $[[t_0 t_3]]$ appears as a \emph{chord} of a $6$-cycle which is a $\Pi$--facial cycle of the embedding, or when two $4$--cicles intersect in one edges. Notice that $E(G)\subset E(G^{e}(\Pi)).$ As such, we will refer to the edges in $E(G)$ as simply \textit{edges.}

Given two different polyhedral embeddings $\Pi$ and $\Pi'$ is not obvious that their corresponding extended graphs $G^{e}(\Pi)$ and $G^{e}(\Pi')$ are combinatorically different. Figuring this out is precisely the aim of this paper.

\begin{theorem}\label{Contribution}
Let $G$ be a finite cubic graph. Then there is a one to one correspondence between the set of embeddings of $G$, $\mathfrak{P}(G)=\{\Pi | \Pi \text{ is an embedding of } G\}$, and the set of extended graphs $\{G^e(\Pi)| \Pi \in \mathfrak{P}(G)\}$.
\end{theorem}

\section{Preliminaries}
Firstly, note that as $G$ is a cubic graph and $\Pi$ is an embedding of $G$ then every path of length two is in a $\Pi$--facial cycle. This follows as there are three faces incident to every vertex of the embedding, thus any path of length two is contained in one of the faces incident to the vertex in the center of the path.

\begin{proposition}\label{edge}
Let $n\geq 5$ and $\mathcal{C}=\ (t_0 t_1 ... t_{n-1} t_0)$ be a cycle in $G$ corresponding to a $\Pi$--facial cycle, then there is no edge $(t_i t_j)$ in $E(G),$ with $i\neq j$, $|i-j|\geq 2$ and this difference taken $\mod n$.
\end{proposition}

\begin{proof}
We will proceed by contradiction. Suppose that $(t_i t_j) \in E(G)$ where $|i-j|\geq 2$ and let $\mathcal{C}_{ij}$ be a facial cycle that passes through the edge $(t_i t_j)$. Since the graph has degree three, $\mathcal{C}_{ij}$ contains either the edge $(t_{i}t_{i+1})$ or $(t_{i-1}t_i)$, and one of $(t_{j}t_{j+1})$ or $(t_{j-1}t_j)$. This contradicts the definition of polyhedrality, since the $\Pi_{ij}$--facial cycle would intersect $\mathcal{C}$ in two edges. 
\end{proof}

\begin{proposition}\label{pathtwo}
Let $n\geq 5$ and $\mathcal{C}=\ (t_0 t_1 ... t_{n-1} t_0)$ be a cycle in $G$ corresponding to a $\Pi$--facial cycle, then for all $t_i, t_j$, there is no $t_k \in V(G)$ such that $(t_i t_k t_j)$ is a path of $G$,  where $|i-j|\geq 2$ and this difference is taken $\mod n$.
\end{proposition}
\begin{proof}
Notice that every path of length two belongs to a $\Pi$--facial cycle, and then the proof follows by using similar arguments to those in the proof of Proposition \ref{edge}. 
\end{proof}

\begin{proposition}\label{3}
Let $G$ be a cubic graph. If $\mathcal{C}$ is a $3$--cycle of $G$ then $\mathcal{C}$ is a facial cycle in every polyhedral embedding of $G$. 
\end{proposition}
\begin{proof}
Let $(t_0 t_1 t_2 t_0)$ be a cycle in $G$. Since every path of length two belongs to a $\Pi$--facial cycle, say $(t_0 t_1 t_2)$ is in a facial cycle $\mathcal{C}$. If $(t_0  t_2) \not \in \mathcal{C}$ then we would contradict Proposition \ref{edge}, and the statement follows.
\end{proof}
\begin{proposition}\label{4}
Let $G$ be a cubic graph. If $\mathcal{C}$ is a $4$--cycle of $G$ and $G\neq K_4$ then $\mathcal{C}$ is a facial cycle in every polyhedral embedding of $G$.
\end{proposition}
\begin{proof}
If G has a polyhedral embedding and $G \neq K_4$, then every $4$--cycle of $G$ is induced, since G is 3-connected by Proposition \ref{layout}. Let $(t_0 t_1 t_2 t_3 t_0)$ be  a cycle in $G$. Since every path of length two belongs to a $\Pi$--facial cycle, say $(t_0 t_1 t_2)$ is in a facial cycle $\mathcal{C}$. If the path $(t_2 t_3 t_0) \not \in \mathcal{C}$ then we would contradict Proposition \ref{pathtwo}, and the statement follows.
\end{proof} 

\begin{proposition}\label{one} 
Given two paths,  $(t_0 t_1 t_2 t_3), (t_0 t_1 t_2 t_{3'}),$ of $G$ then either $(t_0 t_1 t_2 t_3)$ or $(t_0 t_1 t_2 t_{3'})$, but not both, is a $\Pi$--facial subwalk.
\end{proposition}

\begin{proof} Given that every path of length two is in a $\Pi$--facial cycle then either $(t_0 t_1 t_2 t_3)$ or $(t_0 t_1 t_2 t_{3'})$ is a $\Pi$--facial subwalk. Suppose that both  $(t_0 t_1 t_2 t_3)$ and $(t_0 t_1 t_2 t_{3'})$ are $\Pi$--facial subwalks,  if each of these paths belongs to a different facial cycle, then they would intersect improperly, contradicting Definition \ref{poly}. If they belong to the same facial cycle, then such facial cycle self intersects, contradicting  Definition \ref{poly}.
\end{proof}
\begin{proposition}\label{pizzageneralization}
Let $\mathcal{P}= (q_0=t_0 , t_1, ..., t_{n} =q_{m})$ and $\mathcal{Q}= (q_0=t_0, q_1, ..., q_{m} =t_{n})$ be two internally disjoint $\Pi$--facial subwalks, such that $(t_0 t_{n})\not \in E(G)$, then $\mathcal{P}\cup \mathcal{Q}$ must be a $\Pi -$facial cycle. 
\end{proposition}

\begin{proof}
We will proceed by contradiction. Let  $\mathcal{C}_{\mathcal{P}} $ be the $\Pi$--facial cycle associated to $\mathcal{P}$ and  $\mathcal{C}_{\mathcal{Q}}$ be the $\Pi$--facial cycle associated to $\mathcal{Q}$, where $\mathcal{C}_{\mathcal{P}} \neq \mathcal{C}_{\mathcal{Q}}$. Since $G$ is a cubic graph, let $u$ and $v$ be the remaining vertices adjacent to $t_0 = q_0$ and $t_{n}=q_{m}$, respectively. Since the edge $(t_0 t_{n})\not\in E(G)$, observe that $v\neq t_0$ and $u \neq t_{n}$. Which implies that $(t_0 u)$ and $(t_{n} v)$ are different. 

Notice both  $\mathcal{C}_{\mathcal{P}}$ and $ \mathcal{C}_{\mathcal{Q}}$ have to contain two edges incident to $t_0=q_0$, but $t_0=q_0$ has degree three; thus $\mathcal{C}_{\mathcal{P}}$ intersects $ \mathcal{C}_{\mathcal{Q}}$ in at least one edge incident to $t_0=q_0$. The same holds for $t_{n}=q_{m}$, hence $\mathcal{C}_{\mathcal{P}}$ and $ \mathcal{C}_{\mathcal{Q}}$ would have to intersect in at least two edges, contradicting Definition \ref{poly}.
\end{proof}

\section{Proof of the main theorem} 

In this section we will denote as $G^e$ any extended graph in $\{G^e(\Pi)| \Pi \in \mathfrak{P}(G)\}$, where we emphasize that we do not claim knowledge of what embedding $G^e$ corresponds to.

The proof of the main result of the paper will be split in to two subsections: the simple case and the difficult case. 

\subsection{The simple case}
We present the simple case first as within its proof it becomes obvious why the second part requires much more detail.

\begin{theorem}\label{thm:easy}
Let $G^{e}$ be an extended graph of a finite cubic graph, $G,$ such that for all edges $[t, t'] \in \mathcal{S}$ there is a unique path of length three in $G$ whose ends are $t$ and $t'$, then $G^{e}$ uniquely determines the $\Pi$--facial cycle system of an embedding. 
\end{theorem}

\begin{proof} From this information we will construct the $\Pi$--facial cycle system.
Let $[t_0 t_3] \in \mathcal{S},$ by hypothesis we may say that $(t_0  t_1  t_2  t_3)$ is the unique path of length three between $t_0$ and $t_3,$ and there is a $\Pi$--facial cycle that contains the facial subwalk  $(t_0  t_1  t_2  t_3)$. Let $t_4$ and $t_{4'}$ be the remaining vertices adjacent to $t_3,$ then by Proposition \ref{one} either $( t_1  t_2  t_3 t_4)$ or $( t_1  t_2  t_3 t_{4'})$ is a facial subwalk, but not both. Additionally, by hypothesis, only one of  $[t_1 t_4]$ or $[t_1 t_{4'}]$ is in $\mathcal{S}$. Hence we know with certainty if the facial cycle that contains $(t_0  t_1  t_2  t_3)$ continues on to $t_4$ or $t_{4'}$.  We can continue with this procedure until we obtain the unique $\Pi-$facial cycle that contains $(t_0  t_1  t_2  t_3).$

Now consider the edge $(t_1 t_2) \in E(G)$, this edge must belong to another $\Pi$--facial cycle. Let $t_{1'}$ and $t_{2'}$ be the remaining vertices adjacent to $t_1$ and $t_2$, respectively. Then, as the embedding is polyhedral, the other facial cycle containing $(t_1 t_2)$, necessarily contains $(t_{1'} t_1 t_2 t_{2'})$, thus $[t_{1'} t_{2'}] \in \mathcal{S}.$ Now we may use the same procedure as before to find the rest of the edges in this facial cycle. 

We can find every $\Pi$--facial cycle in this manner, by selecting at every step an edge of the union of the preceding facial cycles that hasn't appeared in two facial cycles yet.
This procedure is finite, since the graph $G$ is finite.
\end{proof}

As a consequence of Theorem \ref{thm:easy} we have the following:

\begin{corollary} Let $G^{e}$ be an extended graph of a finite cubic graph, $G,$ with no $6$--cycles, then $G^{e}$ uniquely determines the $\Pi$--facial cycle system of an embedding of $G$.
\end{corollary}

%%%%%%%%%%%% LA PRUEBA DIFICIL %%%%%%%%%%%%%
\subsection{The difficult case} Clearly, the difficulty arises when we have the possibility that the hypothesis of Theorem \ref{thm:easy} does not hold for an extended graph, $G^{e}$. Namely, NOT for all scaffold edges  $[t, t'] \in \mathcal{S}$  there is a unique path of length three in $G$ whose ends are $t$ and $t'$. This motivates the following definitions:

\begin{definition}\label{ye}(Fork)
Let  $Y  \subset G^e$ be a subgraph with set of vertices $V(Y)=\{t_1, t_2, t_3, t_4, t_{4'}\}$ and set of edges $E(Y)=\{(t_1 t_2), (t_2 t_3), (t_3 t_4), (t_3 t_{4'})\}\cup \{[t_1 t_4], [t_1 t_{4'}]\}$. We will call such graph a \emph{fork.} 
\begin{figure}[h]
\begin{center}
\includegraphics[scale=1.5]{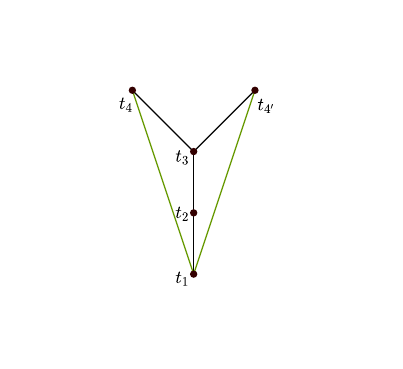}
\caption{Fork}
\end{center}
\end{figure}
\end{definition}

Here, if we tried to reproduce the reconstruction procedure presented in the proof of Theorem \ref{thm:easy}, when arriving at a fork we would have a \emph{disjunction} consisting of whether $(t_1 t_2 t_3 t_4)$ is the $3$--path corresponding to $[t_1 t_4]$ or $(t_1 t_2 t_3 t_{4'})$ is the $3$--path corresponding to $[t_1 t_{4'}]$. A great part of this section consists of proving that, in most cases, this \emph{disjuntion can be solved} by looking at the rest of the structure of the graph. We will say that the \emph{disjuntion cannot be solved} whenever the rest of the structure of the extended graph does not force either $(t_1 t_2 t_3 t_4)$ to be the $3$--path corresponding to $[t_1t_4]$ or $(t_1 t_2 t_3 t_{4'})$ to be the $3$--path corresponding to $[t_1 t_{4'}]$.

We will now prove some more technical lemmas which will help us discover the very particular structure of graphs for which disjunctions cannot be solved.

\begin{proposition}\label{3path}
Let $\mathcal{P}_{1}= (u t_{0} t_{1} v)$, $\mathcal{P}_2 =(u t_2 t_3 v)$ and $\mathcal{P}_3 =(u t_4 t_5 v)$ be three internally disjoint $3$--paths in $G$ such that $[uv]\in G^e$ then $[[uv]]$ is a double scaffold edge. Furthermore, $\mathcal{P}_i \cup \mathcal{P}_j$  is a $\Pi$-facial cycle for a pair $i, j\in\{1, 2, 3\}$.
\end{proposition}
\begin{proof}
With out loss of generality, suppose that $[uv]$ corresponds to the path $\mathcal{P}_1$, which means that $(u t_0 t_1 v)$ is $\Pi$-facial subwalk. By Proposition \ref{one}, it must satisfied that either $(u t_0 t_1 v t_3)$ or $(u t_0 t_1 v t_5)$ is a $\Pi$-facial subwalk. Suppose, with out loss of generality, that $(u t_0 t_1 v t_3)$ is a $\Pi$-facial subwalk, then by Proposition \ref{pathtwo} $(u t_0 t_1 v  t_3 t_2 u)$ is $\Pi$-facial cycle, hence $[[uv]]$ is a double scaffold edge. See Figure 2. 
\end{proof}
\begin{figure}[h]
\vspace{-1.64cm}
    \centering
    \includegraphics[scale=1.5]{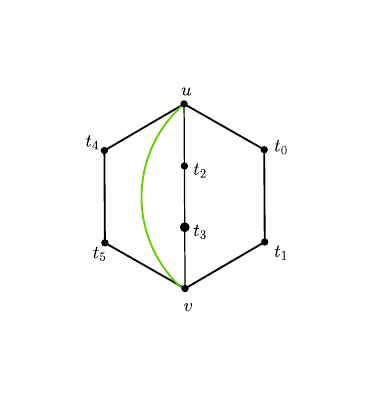}
    \vspace{-1.5cm}
    \label{fig:13-0lemma10(aristadoble).pdf}
    \caption{}
\end{figure}
\begin{corollary}\label{pelito}
Let $\mathcal{P}_{1}= (u t_{0} t_{1} v)$, $\mathcal{P}_2 =(u t_2 t_3 v)$ and $\mathcal{P}_3 =(u t_4 t_5 v)$ be three internally disjoint $3$--paths in $G$ such that $[uv]\in G^e$ and $(t_1 t_0 u t_4)$ is a $\Pi$-facial subwalk. Then $(v t_5 t_4 u)$ is a $\Pi$-facial subwalk. 
\end{corollary}
\begin{proof}
Since $[uv]\in G^e$, Proposition \ref{3path} implies $[[uv]]$ is a double scaffold edge and either $(u t_0 t_1 v t_5 t_4 u)$, $(u t_0 t_1 v t_3 t_2 u)$ or $(u t_2 t_3 v t_5 t_4 u)$ is a $\Pi$-facial cycle. Since $(t_1 t_0 u t_4)$ is a $\Pi$-facial subwalk, by Proposition \ref{one} $(t_1 t_0 u t_2)$ is not a $\Pi$-facial subwalk. This implies that $(u t_0 t_1 v t_3 t_2 u)$ can not be a $\Pi$-facial cycle and, necessarily, either $(u t_0 t_1 v t_5 t_4 u)$ or $(u t_2 t_3 v t_5 t_4 u)$ is a $\Pi$-facial cycle. Notice that the path $(v t_5 y_4 u)$ appears in both cases, hence, this path is a $\Pi$-facial subwalk. 
\end{proof}

\begin{corollary}\label{otro}
Let $(u t_{0} t_{1} v)$ and $(u t_2 t_3 v)$ be two internally disjoint $3$--paths in $G$ such that $[uv]\in G^e$ and $(u t_{0} t_{1} v)$ is not a $\Pi$-facial subwalk. Then $(u t_2 t_3 v)$ is a $\Pi$-facial subwalk. 
\end{corollary}
\begin{proof}
We proceed by contradiction. If $(u t_2 t_3 v)$ is not a $\Pi$-facial subwalk, then there is a third $3$--path between $u$ and $v$, $(u t_4 t_5 v)$, internally disjoint to $(u t_0 t_1 v)$ and$(u t_2 t_3 v)$ which corresponds to the scaffold edge $[uv]$. But, by Proposition \ref{3path} the edge $[[uv]]$ is double. Since $(u t_0 t_1 v)$ is not a $\Pi$-facial subwalk, necessarily $(u t_2 t_3 v)$ and $(u t_4 t_5 v)$ are $\Pi$-facial subwalks, which is a contradiction. 
\end{proof}

\begin{definition}\label{butterfly}(Butterfly one and two)
Let $B  \subset G^e$ be a subgraph of $G$ with set of vertices $V(B)=\{t_0, t_{0'}, t_1, t_2, t_3, t_4, t_{4'}, t_5, t_{5'}\}$, and set of edges equal to the union of the two cycles  $E(Y)=(t_0 t_1 t_2 t_3 t_4 t_5 t_0) \cup (t_{0'}t_1 t_2 t_3 t_{4'} t_{5'} t_{0'})$, we will call this graph a \emph{butterfly.}
A \emph{butterfly one}, $B_1$, is a butterfly with the additional scaffold edges $[t_1 t_4], [t_1 t_{4'}],  [t_0 t_{3}], [t_{0'} t_{3}].$ A \emph{butterfly two}, $B_2$, is a butterfly with the additional scaffold edges $[t_1 t_4], [t_1 t_{4'}],  [t_0 t_{3}].$ 
\begin{figure}[h]
\begin{center}
\label{fig:butterfly}
\includegraphics[scale=1.3]{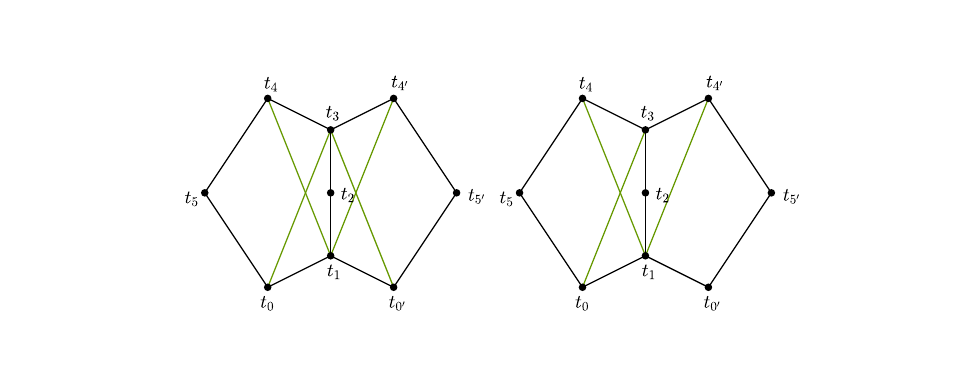}
\caption{The butterfly one is shown on the left and the butterfly two is shown on the right}
\end{center}
\end{figure}
\end{definition}
We are ready to state and proof the main theorems of this section. The spirit is the following: if a fork $Y$  appears as a subgraph of $G^e$ in such a way that the disjunction cannot be solved, then $G^e$ will contain either butterfly one or butterfly two as subgraphs and, in either case, we will be able to reconstruct the full extended graph uniquely.

\begin{lemma}\label{sinduda} Let $Y$ be a fork labelled as in Definition \ref{ye} then if either $(t_1 t_2 t_3 t_4)$ is the only $3$--path between $t_1$ and $t_4$ or $(t_1 t_2 t_3 t_{4'})$ is the only $3$--path between $t_1$ and  $t_{4'}$ then the disjunction can be solved.
\end{lemma}

\begin{proof}
By Definition, every scaffold edge corresponds to a $3$-path. Thus, when there is only one $3$-path between the end vertices of a scaffold edge necessarily such path is its corresponding $3$-path. 
\end{proof}
Next we will prove different theorems. A long of the proofs we state different claims and the procedure that we use is given a claim followed immediately by its proof. The figures that shows the following results appears in the Apendix at the end of the paper. 
\begin{lemma}\label{containedbutterfly}
Let $Y$ be a fork labelled as in Definition \ref{ye}  such that the disjunction cannot be solved then $G^e$ contains a subgraph isomorphic to butterfly one, $B_1,$ or butterfly two, $B_2$.
\end{lemma}

\begin{proof}
Since $G$ is a cubic graph, let $t_0$ and $t_{0'}$ be the remaining vertices adjacent to $t_1,$ different from $t_2.$ 

\emph{Claim 1. The vertices $t_0$ and $t_{0'}$ are different to each $
 t_3, t_4, t_{4'}.$}\\
Note that the proofs of $t_0$ and $t_{0'}$ must be analogous. We only detail the proof for $t_0.$ As $G$ is a $3$--regular graph, $t_0 \neq t_3$ (otherwise $t_3$ would have degree four).
If $t_0=t_4$, then by Proposition \ref{4}, $(t_1 t_2 t_3 t_4 t_1)$ is $\Pi$-facial cycle, and $(t_1 t_2 t_3 t_4)$ is forced to be the  $3$-path corresponding to $[t_1 t_4]$ thus the disjunction can be solved. Hence $t_0\neq t_4$. The case $t_0\neq t_{4'} $ is analogous to the case $t_0\neq t_4$. $\boxminus$\\
Applying Lemma \ref{sinduda} we obtaing the following:\\
\emph{Claim 2. For both edges $[t_1 t_4]$ and $[t_1 t_{4'}]$ there exist at least two internally disjoint $3$-paths between its end vertices.}\\

Without loss of generality we may say that $(t_1 t_0 t_5 t_4)$ is the additional $3$--path with ends $[t_1 t_4],$ internally disjoint from $(t_1 t_2 t_3 t_4)$. Note that then $t_5 \neq t_0, t_1, t_2, t_3, t_4.$

\emph{Claim 3. The vertex $t_5$ is different to the vertices $t_{0'}, t_{4'}$.}\\
If $t_5 = t_{4'}$, Proposition \ref{3} implies that $(t_3 t_4 t_{4'} t_3)$ is a $\Pi$-facial cycle; by Proposition \ref{one} $(t_2 t_3 t_{4'} t_0)$ is a $\Pi$-facial subwalk; Proposition \ref{pathtwo} implies that $(t_0 t_1 t_2 t_3 t_{4'} t_0)$ is a $\Pi$-facial cycle, so $(t_1 t_2 t_3 t_{4'})$ is a $3$-path corresponding to $[t_1 t_{4'}]$ ; hence, the disjunction could be solved, contradicting the hypothesis. Therefore $t_5 \neq t_{4'}.$ \\
If $t_5 =t_{0'}$, using a similar arguments we would arrive to the conclusion that $(t_1 t_2 t_3 t_4)$ is a $3$--path corresponding to $[t_1 t_4]$ contradicting our hypothesis.  Therefore $t_5 \neq t_{0'}.$
$\boxminus$

Now, we will give the additional $3$-path between $t_1$ and $t_{4'}$ internally disjoint to $(t_1 t_2 t_3 t_{4'})$. We have three cases: (Case A) when the $3-$path is $(t_1 t_{0'} t_{5'} t_{4'})$, (Case B) when the $3-$path is $(t_1 t_0 x_0 t_{4'})$ and (Case C) when the $3-$path is $(t_1 t_0 t_5 t_{4'})$. We will deal with each case separately.

\emph{\textbf{Case A.}} Assume the $3-$path is $(t_1 t_{0'} t_{5'} t_{4'})$. Since $G$ is a cubic graph and all the previous vertices already have degree two,  $t_{5'}$ must be a new vertex of $G$, otherwise we would have a vertex with degree four. By Proposition \ref{one}, given that the $3$--paths, $(t_0 t_1 t_2 t_3) $ and $(t_{0'} t_1 t_2 t_3)$ are in $ G$, one of these paths is a $\Pi$--facial subwalk. So we have two cases: either $[t_0 t_3]$ and $[t_{0'}t_3]$ are in $\mathcal{S}$ which implies that $B_1 \subset G^e$, or only of them ($[t_{0}t_3]$ or $[t_{0'}t_3]$) is in $\mathcal{S}$, which implies that $B_2 \subset G^e$. Either way, the assertion of the lemma follows. \emph{This ends the proof for Case A.}$\boxminus$

\emph{\textbf{Case B.}} Assume the $3-$path is $(t_1 t_0 x_0 t_{4'})$. First, observe that Claim 3 implies that $x_0$ is a new vertex. Moreover, the $3-$path $(t_0 t_1 t_2 t_3)$ is in $G$ and it is possible that $[t_0 t_3]\in \mathcal{S}$ or $[t_0 t_3]\not \in \mathcal{S}$. We will analyse both cases:

\emph{Case B1.} $[t_0 t_3]\in \mathcal{S}$.

\setlength{\leftskip}{0.5cm}
\emph{Claim B1.1. $[[t_0 t_3]]$ is a double scaffold edge.} \\
Since there are three internally disjoint $3-$paths between $t_0$ and $t_3$: $(t_0 t_1 t_2 t_3)$, $(t_0 t_5 t_4 t_3)$ and $(t_0 x_0 t_{4'} t_3)$, then $[[t_0 t_3]]$ is a double scaffold edge,  by Corollary \ref{3path}. $\boxminus$\\
\newline
\emph{Claim B1.2. $[t_1 t_4]$ or $[t_1 t_{4'}]$ is a double scaffold edge.} \\
Note that if either $[t_1 t_4]$ or $[t_1 t_{4'}]$ is a double scaffold edge then we necessarily have three 3-paths internally disjoint between the pair $t_1,t_4$ in the first instance and $t_1,t_{4'}$ in the second instance, which would confirm our claim.\\
Thus, we may now suppose that neither $[t_1 t_4]$ nor $[t_1 t_{4'}]$ are double scaffold edges, then $(t_0 t_5 t_4 t_3 t_{4'} x_0 t_0)$ is a $\Pi$--facial cycle, by Corollary \ref{3path}. The latter and Corollary \ref{otro} imply that $(t_1 t_2 t_3 t_4)$ and $(t_1 t_2 t_3 t_{4'})$ are the $\Pi$--facial subwalks corresponding to $[t_1 t_4]$ and $[t_1 t_{4'}]$ respectively, which contradicts Definition \ref{poly}. Hence, at least one of $[t_1 t_4]$ or $[t_1 t_{4'}]$ has to be a double scaffold edge.$\boxminus$

\emph{Claim B1.3. $B_1$ or $B_2$ is contained in $G^e$.}\\
Without loss of generality, suppose that $[[t_1 t_{4'}]]$ is a double scaffold edge. Observe that $(t_1 t_{4'})\not \in E(G)$, otherwise it would contradict the $3-$regularity of $G$. Since there are already two internally disjoint $3$--paths between $t_1$ and $t_{4'}$: $(t_1 t_2 t_3 t_{4'})$ and $(t_1 t_0 x_0 t_{4'})$, we have to have a third one, else $(t_1 t_2 t_3 t_{4'} x_0 t_0 t_1)$ would be a $\Pi$--facial cycle, by Proposition \ref{pizzageneralization}, which implies that $(t_1 t_2 t_3 t_{4'})$ is a $\Pi$--facial subwalk. That is, the disjunction can be solved, contradicting our hypothesis. \\
Let $(t_1 t_{0'} t_{5'} t_{4'})$ be the third path. Notice that $t_{5'}$ is a new vertex, otherwise it would contradict the hypothesis that $G$ is a cubic graph. Then, if $[t_{0'}t_3]\in \mathcal{S}$ then $B_1 \subset G^e$ or if $[t_{0'}t_3]\not\in \mathcal{S}$ then $B_2 \subset G^e$. \\
If we assume instead that $[[t_1 t_4]]$ is a double scaffold edge, the proof follows in a similar way. $\boxminus$

\setlength{\leftskip}{0cm}
This completes the proof for \emph{Case B1.}

\emph{Case B2.} $[t_0 t_3]\not\in \mathcal{S}$. This implies that $(t_0 t_1 t_2 t_3)$ is not a $\Pi$--facial subwalk and, by Proposition \ref{one}, $(t_{0'} t_1 t_2 t_3)$ is a $\Pi$--facial subwalk, hence $[t_{0'} t_3]\in \mathcal{S}$.

\setlength{\leftskip}{0.5cm}
\emph{Claim B2.1 If $(t_{0'}t_1 t_2 t_3 t_4)$ is a $\Pi$--facial subwalk then $(t_5 t_0 t_1 t_2)$ is a $\Pi$--facial subwalk. If $(t_{0'}t_1 t_2 t_3 t_{4'})$ then $(t_{0'} t_1 t_0 t_5 t_4)$ is a $\Pi$--facial subwalk.}\\
Suppose that $(t_{0'}t_1 t_2 t_3 t_4)$ is a $\Pi$--facial subwalk. The later implies that $(t_1 t_2 t_3 t_{4'})$ is not a $\Pi$--facial subwalk. Since $[t_1 t_{4'}]\in \mathcal{S}$ and using Corollary \ref{otro}, $(t_1 t_0 x_0 t_{4'})$ is a $\Pi$--facial subwalk. This $\Pi$--facial subwalk cannot contain the edge $(t_1 t_2)$ otherwise, by Proposition \ref{pathtwo}, $(t_1 t_0 x_0 t_{4'} t_3 t_2 t_1)$ would be a $\Pi$--facial cycle intersecting the $\Pi$--facial subwalk $(t_{0'} t_1 t_2 t_3)$ in two edges, which contradicts the definition of polyhedral embedding. Then, $(t_{0'} t_1 t_0 x_0 t_{4'})$ is a $\Pi$--facial subwalk. We know that two different $\Pi$--facial cycles pass through each edge of $G$; this fact together with Definition \ref{poly} and Proposition \ref{one} imply that $(t_5 t_0 t_1 t_2)$ is a $\Pi$--facial subwalk. \\
Now, suppose that $(t_{0'} t_1 t_2 t_3 t_{4'})$ is a $\Pi$--facial subwalk. This implies that $(t_1 t_2 t_3 t_4)$ is not the $\Pi$--facial subwalk corresponding to the scaffold edge $[t_1 t_4]$. Hence, using Corollary \ref{otro} we know that $(t_1 t_0 t_5 t_4)$ is the $\Pi$--facial subwalk corresponding to the scaffold edge $[t_1 t_4]$. This $\Pi$--facial subwalk cannot contain the edge $(t_1 t_2)$, otherwise $(t_2 t_1 t_0 t_5 t_{4} t_3 t_2)$ would be a $\Pi$--facial cycle, by Proposition \ref{pathtwo}, and it would intersect the $\Pi$--facial subwalk $(t_{0'} t_1 t_2 t_3)$ in two edges, contradicting the definition of polyhedral embedding. Then, $(t_{0'} t_1 t_0 t_5 t_4)$ is a $\Pi$--facial subwalk.$\boxminus$

\emph{Claim B2.2 The fork $Y'$ given by the set of vertices $\{t_5, t_0, t_1, t_{0'}, t_2 \}$, and the edges, $\{(t_5 t_0), (t_0 t_1), (t_1 t_2), (t_1 t_{0'}), [t_5 t_2], [t_5 t_{0'}]\}$ is contained in $G^e$.} \\
If the statement were false, i.e., if one of the scaffold edges $[t_5 t_2]$ or $[t_5 t_{0'}]$ is not in $\mathcal{S}$, then the disjunction can be solved, contradicting our hypothesis. Then, the statement follows. $\boxminus$\\
    
In a way similar to our treatment of the cases for fork $Y$, we will analyse the fork $Y'.$ We need to have two internally disjoint $3$--paths for each pair of vertices $\{t_5, t_2\}$ and $\{t_5, t_{0'}\}$. Notice that for the pair $\{t_5, t_2\}$ we already have two internally disjoint $3$--paths: $(t_5 t_4 t_3 t_2)$ and $(t_5 t_0 t_1 t_2)$. Therefore, we only need to find a second $3$--path between $t_5$ and $t_{0'}$. We have three cases:\\

\emph{Case B2.A $(t_5 t_4 t_3 t_{0'})$ cannot be the second $3$--path between $t_5$ and $t_{0'}$}. \\
Suppose the second $3$--path is $(t_5 t_4 t_3 t_{0'})$. This would imply that $t_3$ has degree four, contradicting that $G$ is a cubic graph. $\boxminus$.

\emph{Case B2.B If $(t_5 t_4 x_4 t_{0'})$ is the second $3$--path  between $t_5$ and $t_{0'}$ then $B_2$ is contained in $G^e$. }\\
Suppose the second $3$--path is $(t_5 t_4 x_4 t_{0'})$, where $x_4$ is the remaining vertex adjacent to $t_4$. Observe that $x_4$ is a new vertex, since $x_4$ is different to $t_3, t_4, t_5$, and if $x_4$ is equal to one of $t_0,\ t_1,\ t_2,\ t_{4'}, x_0$, that would contradict the fact that $G$ is a cubic graph. Finally, $x_4 \neq t_{0'}$, otherwise, by Proposition \ref{pathtwo}, $(t_{0'} t_1 t_2 t_3 t_4 t_{0'})$ would be a $\Pi$--facial cycle, implying that $(t_1 t_2 t_3 t_4)$ is a $\Pi$--facial subwalk, contradicting our hypothesis. \\
Hence, we would have to $B_2$ contained in $G^e$, given by the set of vertices: $\{t_0, t_{0'}, t_1, t_2, t_3, t_4, t_{4'}, x_0, x_4\}$ and by the set of edges of the two $6$--cycles: $(t_1 t_2 t_3 $ $t_{4'} x_0 t_0 t_1)\cup (t_1 t_2 t_3 t_4 x_4 t_{0'} t_1)$ and the scaffold edges $\{[t_1 t_4], [t_1 t_{4'}], [t_{0'} t_3]\}$. $\boxminus$

\emph{Case B2.C $(t_5 x_5 x_{0'} t_{0'})$ cannot be the second $3$--path.}\\
Suppose the second $3$--path is $(t_5 x_5 x_{0'} t_{0'})$, where $x_5$ and $x_{0'}$ are the remaining vertices adjacent to $t_5$ and $t_{0'}$ respectively.\\ 
Observe that $x_{0'}$ and $x_5$ are different from the vertices $t_0, t_{0'}, t_1, t_2, t_3, t_4, t_5$ by a similar proof to the case A. Also $x_{0'}$ and $x_5$ are different from the vertices $t_{4'}$ and $x_0$, otherwise contradicts the fact that $G$ is a cubic graph. So there is a butterfly contained in $G$ given by the union of the two $6$--cycles $(t_0,  t_1, t_2, t_3, t_4, t_5 t_0)$$\cup$$(t_0 t_1 t_{0'} x_{0'} x_5 t_5 t_0)$ and the set of scaffold edges $\{[t_5 t_2], [t_5 t_{0'}], [t_1 t_4], [t_1 x_5]\}$ if is a butterfly one, or the set of scaffold edges $\{[t_5 t_2], [t_5 t_{0'}], [t_1 t_4]\}$ if is a butterfly two. 

This completes the proof for \emph{Case B2} and Case B. $\boxminus$\\

%%%%%%%%%Case C%%%%%%%%%%%%%%%

\emph{\textbf{Case C.}} Assume that the $3$--path is $(t_1 t_0 t_5 t_{4'})$. Notice that $(t_3 t_4 t_5 t_{4'} t_3)$ is a $\Pi$--facial cycle of $G$. \\
Suppose that $[t_0 t_3]\in \mathcal{S}$. Since $(t_3 t_4 t_5 t_{4'} t_3)$ is a $\Pi$--facial cycle, $(t_0 t_5 t_4 t_3)$ cannot be a $\Pi$--facial subwalk, by Propositon \ref{one}. Using Corollary \ref{otro}, $(t_0 t_1 t_2 t_3)$ is the $\Pi$--facial subwalk corresponding to $[t_0 t_3]$. The $\Pi$--facial subwalk $(t_0 t_1 t_2 t_3)$ cannot continue trough the vertex $t_4$, otherwise by Proposition \ref{pathtwo} implies that $(t_0 t_1 t_2 t_3 t_4 t_5 t_0)$ is a $\Pi$--facial cycle, and it would intersect the $\Pi$--facial cycle $(t_3 t_4 t_5 t_{4'} t_3)$ in two edges, contradicting the definition of polyhedral embedding. So $(t_0 t_1 t_2 t_3 t_{4'})$ is a $\Pi$--facial subwalk, i.e., the disjunction can be solved, contradicting out hypothesis. \\
Then, suppose that $[t_0 t_3]\not\in \mathcal{S}$. Since $(t_0 t_1 t_2 t_3)$ is not a $\Pi$--facial subwalk, using Proposition \ref{one} implies that $(t_{0'} t_1 t_2 t_3)$ is a $\Pi$--facial subwalk. Thus either $(t_{0'} t_1 t_2 t_3 t_4)$ or $(t_{0'} t_1 t_2 t_3 t_{4'})$ is a $\Pi$--facial subwalk.

\emph{Claim C1. The remaining vertex adjacent to $t_2$, $x_2,$ is a new vertex of $G$.} \\
We are going to proceed by contradiction:
\begin{enumerate}
    \item $(x_2 \neq t_5)$. If $x_2 = t_5$ it would contradict the $3$--regularity of $G$. 
    \item $(x_2\neq t_4)$. If $x_2 = t_4$ then $(t_2 t_3 t_4 t_2)$ is a $\Pi$--facial cycle, by Proposition \ref{3}. The latter implies that $(t_{0'} t_1 t_2 t_3 t_4)$ cannot be a $\Pi$--facial subwalk, since it would intersect the $\Pi$--facial cycle $(t_2 t_3 t_4 t_2)$ in two edges, contradicting the polyhedrality of the embedding. Then, by Proposition \ref{one}, $(t_{0'} t_1 t_2 t_3 t_{4'})$ is a $\Pi$--facial subwalk, i.e., $(t_1 t_2 t_3 t_{4'})$ is the $\Pi$--facial subwalk corresponding to the edge $[t_1 t_{4'}]$, and the disjunction can be solved, contradicting out hypothesis. 
    \item $(x_2\neq t_{4'})$ This proof follows in a manner similar to the previous one.
    \item $(x_2 \neq t_{0'})$. If $x_2= t_{0'}$ it would contradict the definition of polyhedral embedding, since $(t_{0'} t_1 t_2 t_3)$ is a $\Pi$--facial subwalk. 
    \item $(x_2 \neq t_0)$. If $x_2 = t_0$ then $(t_0 t_1 t_2 t_0)$ would be a $\Pi$--facial cycle, by Proposition \ref{3}. Since there are two different $\Pi$--facial cycles passing through every edge, then Proposition \ref{one} implies that $(t_3 t_2 t_0 t_5)$ is a $\Pi$--facial subwalk. Thus, $(t_3 t_2 t_0 t_5 t_4 t_3)$ is a $\Pi$--facial cycle, by Proposition \ref{pathtwo}. But notice that this would contradict the polyhedrality of the embedding.  
\end{enumerate}
Hence, $x_2$ is a new vertex. $\boxminus$

\emph{Claim C2. If $(t_{0'} t_1 t_2 t_3 t_4)$ is a $\Pi$--facial subwalk then $(x_2 t_2 t_3 t_{4'})$ is a $\Pi$--facial subwalk.}\\
Notice that $(t_1 t_2 t_3 t_4)$ is the $\Pi$--facial subwalk corresponding to $[t_1 t_4]$. Given that two different $\Pi$--facial cycles pass through every edge and using Proposition \ref{one} we can deduce that $(x_2 t_2 t_3 t_{4'})$ is a $\Pi$--facial subwalk. $\boxminus $

\emph{Claim C3. If $(t_{0'} t_1 t_2 t_3 t_{4'})$ is a $\Pi$--facial subwalk then $(t_1 t_2 t_3 t_{4'})$ is a $\Pi$--facial subwalk.}\\
This statement follows immediately from the hypothesis, since $(t_1 t_2 t_3 t_{4'})$ is contained in the $\Pi$--facial subwalk $(t_{0'} t_1 t_2 t_3 t_{4'})$.
$\boxminus$

\emph{Claim C4. The fork $Y'$ given by the set of vertices $\{t_{4'}, t_3, t_2, t_1, x_2\}$ and the set of edges $\{(t_3 t_{4'}), (t_2 t_3), (t_1 t_2), (t_2 x_2), [t_1 t_{4'}], [t_{4'} x_2]\} $. is contained in $G^e$.}\\
Notice that only remains to prove that $[t_{4'} x_2]$ is in $\mathcal{S}$. If this scaffold edge is not in $\mathcal{S}$, then $(x_2 t_2 t_3 t_{4'})$ is not a $\Pi$--facial subwalk. The latter and using Proposition \ref{one} imply that $(t_{0'} t_1 t_2 t_3 t_{4'})$ is a $\Pi$--facial subwalk. Thus there is no disjunction contradicting our hypothesis.
$\boxminus$\\
Let $x_{4'}$ be the remaining vertex adjacent to $t_{4'}$. 

\emph{Claim C5. $(t_2 t_3 t_{4'} x_{4'})$ is a $\Pi$--facial subwalk.} \\
Since every path of length two is a $\Pi$--facial subwalk, then $(t_2 t_3 t_{4'})$ is a $\Pi$--facial subwalk. But this $\Pi$--facial subwalk cannot continue to $t_5$, otherwise it would intersect the $\Pi$--facial cycle $(t_3 t_4 t_5 t_{4'} t_3)$ in two edges, contradicting Definition \ref{poly}. Hence, we have to continue towards the remaining neighbour of $t_{4'}$, $x_{4'}$, and the statement follows. $\boxminus$

\emph{Claim C6. There are two internally disjoint $3$--paths for each pair of vertices $\{t_1, t_{4'}\}$ and $\{t_{4'}, x_2\}$. }\\
It follows immediately by Lemma \ref{sinduda}. 
$\boxminus$

The two internally disjoint $3$--paths between $t_1$ and $t_{4'}$ are: $(t_1 t_2 t_3 t_{4'})$ and $(t_1 t_0 t_5 t_{4'})$. 
The first of the two internally disjoint $3$--paths between $t_{4'}$ and $x_2$ is $(t_{4'} t_3 t_2 x_2)$; as for the second there are three options: (Case C1) $(t_{4'} t_5 t_0 x_2)$, (Case C2) $(t_{4'} t_5 t_4 x_2)$ and (Case C3) $(t_{4'} x_{4'} x_{2'} x_2)$. We will deal with each instance separately.

\emph{Case C1.} Assume that the second internally disjoint $3$--path between $t_{4'}$ and $x_2$ is $(t_{4'} t_5 t_0 x_2)$. 
\\
\emph{Claim C1.1. The remaining vertex adjacent to $t_{4'}$, $x_{4'}$, is a new vertex.}\\
We are going to proceed by contradiction:
\begin{enumerate}
    \item ($x_{4'}\neq t_0$, $x_{4'}\neq t_1$, $x_{4'}\neq t_2$, $x_{4'}\neq t_3$, $x_{4'}\neq t_5$). If any of the inequalities was an equality instead we would have a contradiction on the fact that $G$ is a cubic graph. 
    \item ($x_{4'}\neq t_{0'}$). If $x_{4'}=t_{0'}$ then by Proposition \ref{pathtwo}, $(t_{0'} t_1 t_2 t_3 t_{4'} t_{0'})$ is a $\Pi$--facial cycle because $(t_{0'} t_1 t_2 t_3)$ is a $\Pi$--facial subwalk. This implies that $(t_1 t_2 t_3 t_{4'})$ is the $\Pi$--facial subwalk corresponding to the edge $[t_1 t_{4'}]$, i.e., the disjunction can be solved, contradicting our hypothesis. Hence $x_{4'}\neq t_{0'}$. 
    \item ($x_{4'}\neq t_4$). If $x_{4'} = t_4$, since $(t_3 t_4 t_5 t_{4'} t_3 )$ is a $\Pi$--facial cycle this contradicts Proposition \ref{edge}. Hence $x_{4'} \neq x_4$.
    \item ($x_{4'}\neq x_2$). If $x_{4'}=x_2$ then the $\Pi$--facial subwalk $(t_{0'} t_1 t_2 t_3 )$ cannot continue through $t_{4'}$, since there is a $2$--path $(t_2 x_2 t_{4'})$, which would contradict Proposition \ref{pathtwo}. Then $(t_{0'} t_1 t_2 t_3 )$ continues through $t_4$, implying that $(t_1 t_2 t_3 t_4)$ is the $\Pi$--facial subwalk corresponding to $[t_1 t_4]$, i.e., the disjunction can be solved, contradicting our hypothesis. Hence $x_{4'}\neq x_2$.
\end{enumerate}
Thus, $x_{4'}$ is a new vertex. $\boxminus$\\
\\
\emph{Claim C1.2. The remaining vertex adjacent to $t_{4}$, $x_{4}$, is a new vertex.}\\
We are going to proceed by contradiction:
\begin{enumerate}
    \item ($x_{4}\neq t_0$, $x_{4}\neq t_1$, $x_{4}\neq t_2$, $x_{4}\neq t_3$, $x_{4}\neq t_5$, $x_4 \neq t_{4'}$). If any of the inequalities was an equality instead we would have a contradiction on the fact that $G$ is a cubic graph. 
    \item ($x_4 \neq x_2$). If $x_{4}=x_2$ then by Proposition \ref{4}, $(t_2 t_3 t_4 x_2 t_2)$ is a $\Pi$--facial cycle. Thus the $\Pi$--facial subwalk $(t_{0'} t_1 t_2 t_3)$ cannot continue through $t_4$, otherwise it would contradict the polyhedrality of $G$. Then $(t_{0'} t_1 t_2 t_3 t_{4'} x_{4'})$ is a $\Pi$--facial subwalk. This implies that $(t_1 t_2 t_3 t_{4'})$ is the $\Pi$--facial subwalk corresponding to $[t_1 t_{4'}]$, therefore the disjunction can be solved, contradicting our hypothesis. Hence $x_4 \neq x_2$. 
    \item ($x_4 \neq t_{0'}$). If $x_4 = t_{0'}$ then using Proposition \ref{pathtwo} it can be said that $(t_{0'} t_1 t_2 t_3 t_4 t_{0'})$ is a $\Pi$--facial cycle. The latter implies that $(t_1 t_2 t_3 t_4)$ is the $\Pi$--facial subwalk corresponding to $[t_1 t_4]$, which contradicts our hypothesis. Hence $x_4 \neq t_{0'}$.  
    \item ($x_4 \neq x_{4'}$). If $x_4 =x_{4'}$, since $(t_3 t_4 t_5 t_{4'} t_3)$ is a $\Pi$--facial cycle, it would contradict Proposition \ref{pathtwo}.
\end{enumerate}
Hence, $x_4$ is a new vertex. $\boxminus$\\
\\
\emph{Claim C1.3. The remaining vertex adjacent to $x_{2}$, $x_{2'}$, is a new vertex.}\\
We are going to proceed by contradiction:
\begin{enumerate}
    \item ($x_{2'}\neq t_0, x_{2'}\neq t_1 ,x_{2'}\neq t_2, x_{2'}\neq t_3, x_{2'}\neq t_4, x_{2'}\neq t_{4'}, x_{2'}\neq t_5$). If any of the inequalities was an equality instead we would have a contradiction on the fact that $G$ is a cubic graph. 
    \item ($x_{2'}\neq t_{0'}$). If $x_{2'}=t_{0'}$ since $(t_0 t_1 t_2 x_2 t_0)$ is a $\Pi$--facial cycle, it would contradict Proposition \ref{pathtwo}. Hence $x_{2'}\neq t_{0'}$. 
    \item ($x_{2'} \neq x_4$). If $x_{2'}=x_4$ then $(x_4 x_2 t_2 t_3)$ is a $\Pi$--facial subwalk, since $(t_0 t_1 t_2 x_2 t_0)$ is a $\Pi$--facial cycle. Using Proposition \ref{pathtwo} and the latter imply that $(x_4 x_2 t_2 t_3 t_4 x_4)$ is a $\Pi$--facial cycle. Therefore the $\Pi$--facial subwalk $(t_{0'} t_1 t_2 t_3)$ cannot continue though $t_4$, otherwise it would contradict the polyhedrality of the embedding. Then $(t_{0'} t_1 t_2 t_3 t_{4'})$ is a $\Pi$--facial subwalk, which implies that $(t_1 t_2 t_3 t_{4'})$ is the $\Pi$--facial subwalk corresponding to $[t_1 t_{4'}]$ and the disjunction can be solved, contradicting our hypothesis. Hence $x_{2'} \neq x_4$. 
    \item ($x_{2'} \neq x_{4'}$). This proof follows in a similar way that the previuos one. 
\end{enumerate}
Thus $x_{2'}$ is a new vertex of $G$. $\boxminus$ \\
\\
\emph{Claim C1.4. Either $(t_1 t_2 t_3 t_{4'})$ or $(t_1 t_0 t_5 t_{4'})$ is a $\Pi$--facial subwalk.}\\
We know that $[t_1 t_4]$ and $[t_1 t_{4'}]$ are in $\mathcal{S}$, also Proposition \ref{one} implies that either $(t_1 t_2 t_3 t_4)$ or $(t_1 t_2 t_3 t_{4'})$ is a $\Pi$--facial subwalk. \\
If $(t_1 t_2 t_3 t_{4'})$ is a $\Pi$--facial subwalk the statement follows. If, on the other hand, $(t_1 t_2 t_3 t_4)$ is not a $\Pi$--facial subwalk, Corollary \ref{otro} implies that $(t_1 t_0 t_5 t_{4'})$ is the $\Pi$--facial subwalk corresponding to $[t_1 t_{4'}]$. 
$\boxminus$\\
\\
\emph{Claim C1.5. For either choice between $(t_1 t_2 t_3 t_{4'})$ or $(t_1 t_0 t_5 t_{4'})$ being a $\Pi$--facial subwalk,  there is a $\Pi$--facial subwalk $\mathcal{P}=(t_1 t_{0'} \ldots x_{4'} t_{4'})$, such that either in $\mathcal{P}\cup (t_1 t_2 t_3 t_{4'})$ is a $\Pi$--facial cycle in the first instance or $\mathcal{P}\cup (t_1 t_0 t_5 t_{4'})$ is a $\Pi$--facial cycle in the second instance.}\\
First, suppose that $(t_1 t_2 t_3 t_{4'})$ is a $\Pi$--facial subwalk. This walk cannot be extended as $(t_0 t_1 t_2 t_3 t_{4'} t_5)$, otherwise it would intersect the $\Pi$--facial cycles $(t_0 t_1 t_2 x_2 t_0)$ and $(t_3 t_4 t_5 t_{4'} t_3)$ in two edges, contradicting the definition of polyhedral embedding. Hence the only possible way to extend this walk is $(t_{0'} t_1 t_2 t_3 t_{4'} x_{4'})$ and the claim follows. \\
Suppose $(t_1 t_0 t_5 t_{4'})$ is a $\Pi$--facial subwalk. Similarly, this walk cannot be extended as $(t_2 t_1 t_0 t_5 t_{4'} t_3)$, otherwise it would intersect the $\Pi$--facial cycles $(t_0 t_1 t_2 x_2 t_0)$ and $(t_3 t_4 t_5 t_{4'} t_3)$ in two edges, contradicting the definition of polyhedral embedding. Hence, the only possible way to extend this walk is $(t_{0'} t_1 t_0 t_5 t_{4'} x_{4'})$ and the claim follows. \\
In both cases, a path $\mathcal{P}=(t_1 t_{0'}\ldots x_{4'} t_{4'})$ completes either $\Pi$--facial subwalk in to a $\Pi$--facial cycle. 
$\boxminus$\\
\emph{Claim C1.6. The $\Pi$--facial subwalk $\mathcal{P}$ in the previous claim is $(t_1 t_{0'} x_{4'} t_{4'})$.}\\
We proceed by contradiction. Suppose that there is at least one vertex $t_p\in V(G)$ such that $\mathcal{P}=(t_1 t_{0'}t_p\ldots x_{4'} t_{4'})$. We will now prove that we can deduce further structure of $G^{e}$ from this assumption and arrive to a contradiction for either $(t_1 t_2 t_3 t_{4'})$ or $(t_1 t_0 t_5 t_{4'})$ being a $\Pi$--facial subwalk.$\boxminus$\\
\emph{Claim C1.6.1. The fork given by the set of vertices $\{t_p, t_{0'}, t_1, t_0, t_2\}$ and the set of edges $\{(t_p t_{0'}), (t_{0'} t_1), (t_1 t_0), (t_1 t_2), [t_p t_0], [t_p t_2]\}$ has to be contained in $G^e$.}\\
In this instance, by Proposition \ref{one}, either $(t_p t_{0'}t_1t_0)$ or $(t_p t_{0'} t_1 t_2)$ is a $\Pi$--facial subwalk. Therefore, if one of the scaffold edges $[t_p t_0]$ or $[t_p t_2]$ is not in $\mathcal{S}$ we can deduce whether $\mathcal{P}\cup (t_1 t_2 t_3 t_{4'})$ or $\mathcal{P}\cup (t_1 t_0 t_5 t_{4'})$ is a $\Pi$--facial cycle, by Proposition \ref{pathtwo}, and the disjunction can be solved, contradicting the hypothesis and the claim follows.$\boxminus$\\
\emph{Claim C1.6.2. The paths $(t_p x_{2'} x_2 t_0)$ and $(t_p x_{2'} x_2 t_2)$ are in $G$.}\\
As we have argued in other claims,  if there is to be a disjunction, by Lemma \ref{sinduda}, there are a pair of $3$--paths: one joining the vertices $\{t_p, t_0\}$ and another one joining the vertices $\{t_p, t_2\}$, internally disjoint to $(t_p t_{0'} t_1 t_0)$  and  $(t_p t_{0'} t_1 t_2)$, respectively. Looking at the degree of the vertices involved, it is easy to see that the only possibility is that $t_p$ is adjacent to $x_{2'}$ and the additional $3$--paths are $(t_p x_{2'} x_2 t_0)$ and $(t_p x_{2'} x_2 t_2)$. $\boxminus$\\
\emph{Claim C1.6.3. $\mathcal{P}$ does not continue through $x_{2'}$ after $t_p$.}\\
We will proceed by contradiction. Suppose $\mathcal{P}= (t_1 t_{0'}t_p x_{2'}\ldots x_{4'} t_{4'})$, then $\mathcal{P}$ cannot continue to the vertex $t_2$ after $t_1$, otherwise there is a $2$--path $(t_2 x_2 x_{2'})$ contradicting Proposition \ref{pathtwo}. Then, $\mathcal{P}$ continue to $t_0$ after $t_1$ and using Proposition \ref{pathtwo} we conclude that $\mathcal{P}\cup (t_1 t_0 t_5 t_{4'})=(t_1 t_{0'} t_p x_{2'}\ldots x_{4'}t_{4'}t_5 t_0 t_1)$ is a $\Pi$--facial cycle. This implies that undoubtedly $(t_1 t_2 t_3 t_4)$ is the $\Pi$--facial subwalk corresponding to the scaffold edge $[t_1 t_4]$. Ergo, the disjunction can be solved, contradicting our hypothesis. Hence, $x_{2'}$ is not in $\mathcal{P}$.$\boxminus$\\
Finally, we will now see what happens in each of the cases: either $\mathcal{P}\cup (t_1 t_2 t_3 t_{4'})$ or $\mathcal{P}\cup (t_1 t_0 t_5 t_{4'})$ is a $\Pi$--facial cycle. Inasmuch as both cases are symmetric, we will only prove the first one in detail. 
Suppose that $\mathcal{P}\cup (t_1 t_2 t_3 t_{4'})$ is a $\Pi$--facial cycle. This implies that the $3$--path $(t_1 t_2 t_3 t_{4'})$ corresponds to the edge $[t_1 t_{4'}]\in \mathcal{S} $ then the $3$--path that corresponds to $[t_1 t_4]$ is $(t_1 t_0 t_5 t_4)$, given that the embedding is polyhedral, we can conclude that $(t_{0'} t_1 t_0 t_5 t_4 x_4)$ is a $\Pi$--facial subwalk. \\
Analyzing the existing facial subwalks and using extensive use of the fact that the embedding is polyhedral is easy to see that $(x_4 t_4 t_3 t_2 x_2 x_{2'})$ is a $\Pi$--facial subwalk. Notice that this walk cannot continue through the vertex $t_p$ after $x_{2'}$, otherwise it would intersect the $\Pi$--facial cycle $\mathcal{P}\cup (t_1 t_2 t_3 t_{4'})$ in two edges. \\
Similarly, we can conclude that $(x_{4'} t_{4'} t_5 t_0 x_2 x_{2'})$ is a $\Pi$--facial subwalk, because it cannot continue through the vertex $t_p$ after $x_{2'}$, otherwise it would intersect the $\Pi$--facial cycle $\mathcal{P}\cup (t_1 t_2 t_3 t_{4'})$ in two edges. \\
    We have concluded that both $\Pi$--facial subwalks $(x_4 t_4 t_3 t_2 x_2 x_{2'})$ and $(x_{4'} t_{4'} t_5 t_0 x_2 x_{2'})$ continue after  $x_{2'}$ through the same vertex (not equal to $t_p$). This is, they will intersect in two edges, contradicting the definition of polyhedral embedding. 
As we said the proof of the second case is analagous, then in both cases we get to a contradiction, and we can conclude that there is no additional vertex in the path $\mathcal{P}= (t_1 t_{0'} x_{4'} t_{4'})$, and obviously the edge $(t_{0'}x_{4'})$ is in $G$. Thus, there is a butterfly, $B_2$, contained in $G$: it is given by the union of the two $6$--cycles $(t_0 t_1 t_2 t_3 t_4 t_5 t_0)\cup (t_{0'}t_1 t_2 t_3 t_{4'} x_{4'} t_{0'})$ and the set of scaffold edges $\{[t_1 t_4], [t_1 t_{4'}], [t_{0'}t_3]\}$. 
\\
\emph{This ends the proof of Case C1}$\boxminus$
\\
\emph{Case C2.} Assume that the second internally disjoint $3$--path between $t_{4'}$ and $x_2$ is $(t_{4'} t_5 t_4 x_2)$.
Notice that $(t_2 t_3 t_4 x_2 t_2)$ is a $\Pi$--facial cycle by Proposition \ref{4}. Using Claim 1 and Definition \ref{poly} we can deduce that $(t_{0'} t_1 t_2 t_3 t_{4'})$ is a $\Pi$--facial subwalk, i.e., the disjunction can be solved, contradicting our hypothesis. Hence, this case can never occur. 
\emph{This ends the proof of Case C2.}$\boxminus$\\
\emph{Case C3.} Assume that the second internally disjoint $3$--path between $t_{4'}$ and $x_2$ is $(t_{4'} x_{4'} x_{2'} x_2).$ Remember that Claim C5 asserts that $(t_2 t_3 t_{4'} x_{4'})$ is a $\Pi$--facial subwalk. So, there is a butterfly two contained in $G^e$, $B_2$, whose set of vertices is $\{t_0, t_1, t_2, t_3, t_{4'}, t_5, x_{4'}, x_2, x_{2'}\}$ and whose set of edges is given by the union of edges in the two $6$--cycles, $(t_0 t_1 t_2 t_3 t_{4'} t_5 t_0)\cup (t_2 t_3 t_{4'} x_{4'} x_{2'} x_2 t_2)$, and whose set of scaffold edges is $\{[t_{4'} t_1], [t_{4'} x_2], [t_2 x_{4'}]\}$. \emph{This ends the proof for Case C3.}$\boxminus$

\end{proof}
 
\begin{theorem}
Let  $G^e$ be such that $B_1 \subset G^e$ (labelled as in Definition \ref{butterfly}, Figure \ref{fig:butterfly}) and such that the disjunctions that arise from  the two fork subgraphs of $B_1$ induced by the vertices $\{t_1, t_2, t_3, t_4, t_{4'}\}$ and $\{t_0, t_{0'}, t_1, t_2, t_3\}$  cannot be solved, then $G$ is the Petersen's graph, $P$, and $G^{e}$ is the union of $P$ and the set of (single) scaffold edges $\mathcal{S}=E(K_{10})\setminus E(P)$.   
\end{theorem}

\begin{proof}
As $B_1 \subset G^e$ (labelled as in Definition \ref{butterfly}) then $G$ contains the two $6$-cycles $(t_0 t_1 t_2 t_3 t_4 t_5 t_0)$ and $(t_{0'} t_1 t_2 t_3 t_{4'} t_{5'} t_{0'})$. Also, $[t_0 t_3]$, $[t_{0'} t_3]$, $[t_1 t_4]$, $[t_1 t_{4'}] \in \mathcal{S}$. 

\emph{Claim 1.}
\emph{Either $(t_0 t_1 t_2 t_3 t_{4'})$ or $(t_{0'} t_1 t_2 t_3 t_4)$ (but not both) is a $\Pi$-facial subwalk.}

Suppose, without loss of generality, that $(t_0 t_1 t_2 t_3 )$ is a $\Pi$-facial subwalk. If $(t_1 t_2 t_3 t_4)$ is the $3$--path corresponding to $[t_1 t_4]$, by Definition \ref{poly}, $(t_0 t_1 t_2 t_3 t_4)$ is a $\Pi$-facial subwalk and by Proposition \ref{pathtwo}, $(t_0 t_1 t_2 t_3 t_4 t_5 t_0)$ is a $\Pi$-facial cycle. Furthermore, by Proposition \ref{one} and Corollary \ref{otro} imply that $(t_1 t_{0'} t_{5'} t_{4'})$ is the $\Pi$-facial subwalk corresponding to $[t_1 t_{4'}]$ and $(t_{0'} t_{5'} t_{4'} t_3)$ is the $\Pi$-facial subwalk corresponding to $[t_{0'}t_3]$. It follows that, by Definition \ref{poly}, $(t_1 t_{0'} t_{5'} t_{4'} t_3)$ is a $\Pi$-facial subwalk; and by Proposition \ref{pathtwo}, $(t_{0'} t_1 t_2 t_3 t_{4'} t_{5'} t_{0'})$ is a $\Pi$-facial cycle. But this last statement contradicts Definition \ref{poly}, as two different facial walks would intersect in two edges. Hence $(t_1 t_2 t_3 t_4)$ is not the $3$--path corresponding to $[t_1 t_4],$ so $(t_0 t_1 t_2 t_3 t_{4'})$ is a $\Pi$-facial subwalk and then by Proposition \ref{one}, the statement follows.

The case when we suppose $(t_{0'} t_1 t_2 t_3)$ as a $\Pi$-facial subwalk is resolved similarly.
$\boxminus$

\emph{Claim 2. Neither $(t_5 t_0 t_1 t_2 t_3 t_{4'})$ nor $(t_0 t_1 t_2 t_3 t_{4'} t_{5'})$ can be a $\Pi$-facial subwalk.}

Suppose $(t_5 t_0 t_1 t_2 t_3 t_{4'})$ is a $\Pi$-facial subwalk, by Proposition \ref{pathtwo}, $(t_0 t_1 t_2 t_3 t_{4}  t_5 t_0)$ is a $\Pi$-facial cycle, and has to be the same $\Pi$-facial walk that contains $(t_5 t_0 t_1 t_2 t_3 t_{4'})$, this contradicts the definition of polyhedral embedding (Definition \ref{poly}). 

The case for $(t_0 t_1 t_2 t_3 t_{4'} t_{5'})$ follows similarly.
$\boxminus$

\emph{Claim 3. Neither $(t_{5'} t_{0'} t_1 t_2 t_3 t_4)$ nor $(t_{0'} t_1 t_2 t_3 t_4 t_5)$ can be a $\Pi$-facial subwalk.}

The proof of this statement follows in the same way as that of Claim 2.
$\boxminus$\\

\emph{Claim 4. Either $(t_0 t_1 t_2 t_3 t_{4'})$ is a $\Pi$--facial subwalk, implying that $(t_{0'} t_1 t_0 t_5 t_4)$ and $(t_{0'} t_{5'} t_{4'} t_3 t_4)$ are $\Pi$--facial subwalks, or $(t_{0'} t_1 t_2 t_3 t_4)$ is a $\Pi$--facial subwalk, implying that $(t_0 t_1 t_{0'} t_{5'} t_{4'})$ and $(t_0 t_5 t_4 t_3 t_{4'})$ are $\Pi$--facial subwalks. }

By Claim 1, we know that $(t_0 t_1 t_2 t_3 t_{4'})$ and $(t_{0'} t_1 t_2 t_3 t_4)$ cannot be $\Pi$--facial subwalks simultaneously.

Assume that  $(t_0 t_1 t_2 t_3 t_{4'})$ is a $\Pi$-facial subwalk, then by Corollary \ref{otro}, $(t_1 t_0 t_5 t_4)$ is the $\Pi$-facial subwalk of $[t_1 t_4]$. Since $t_1$ has degree three, necessarily $(t_2 t_1 t_0 t_5 t_4)$ or $(t_{0'} t_1 t_0 t_5 t_4)$ is a $\Pi$-facial subwalk. However, if $(t_2 t_1 t_0 t_5 t_4)$ is a $\Pi$-facial subwalk then it intersects the $\Pi$-facial subwalk $(t_0 t_1 t_2 t_3 t_{4'})$ in two edges, contradicting Definition \ref{poly}. Hence $(t_{0'} t_1 t_0 t_5 t_4)$ is a $\Pi$-facial subwalk. 
Since $(t_0 t_1 t_2 t_3)$ is a $\Pi$--facial subwalk, by Proposition \ref{one}, $(t_{0'}t_1 t_2 t_3)$ is not a $\Pi$--facial subwalk. Thus, by Corollary \ref{otro}, the $3-$path corresponding to $[t_{0'} t_3]$ is $(t_{0'} t_{5'} t_{4'} t_3)$. Since $t_3$ has degree three, necessarily $(t_{0'} t_{5'} t_{4'} t_3 t_2)$ or $(t_{0'} t_{5'} t_{4'} t_3 t_4)$ is a $\Pi$-facial subwalk. However, if $(t_{0'} t_{5'} t_{4'} t_3 t_2)$ was a $\Pi$-facial subwalk then it would intersect the $\Pi$-facial subwalk $(t_0 t_1 t_2 t_3 t_{4'})$ in two edges, contradicting Definition \ref{poly}. Hence $(t_{0'} t_{5'} t_{4'} t_3 t_4)$ is a $\Pi$-facial subwalk.

The other case follows similarly.
$\boxminus$\\

\emph{Claim 5. $(t_0 t_1 t_2 t_3 t_{4'}t_0)$ and $(t_{0'} t_1 t_2 t_3 t_4 t_{0'})$ are $5$--cycles of $G$. (Not necessarily facial cycles of the embedding.)}

By Claim 1, we know that either $(t_0 t_1 t_2 t_3 t_{4'})$ or $(t_{0'} t_1 t_2 t_3 t_4)$ is a $\Pi$--facial subwalk.
Assume that $(t_0 t_1 t_2 t_3 t_{4'})$ is a $\Pi$--facial subwalk. Then, by Claim 4,  $(t_{0'} t_1 t_0 t_5 t_4)$ and $(t_{0'} t_{5'} t_{4'} t_3 t_4)$ are $\Pi$-facial subwalks and observe that they intersect in their start and end vertices. By Proposition \ref{pizzageneralization}, either $(t_{0'} t_{5'} t_{4'} t_3 t_4 t_5 t_0 t_1 t_{0'})$ is a $\Pi$-facial cycle, or the edge $(t_{0'}t_4)$ is in $E(G)$. However, $(t_{0'} t_{5'} t_{4'} t_3 t_4 t_5 t_0 t_1 t_{0'})$ can not be a $\Pi$-facial cycle, since it would intersect the $\Pi$--facial subwalk $(t_0 t_1 t_2 t_3 t_{4'})$ in two edges and this would contradict the assumption that the embedding is polyhedral. Then, if $(t_0 t_1 t_2 t_3 t_{4'})$ is a $\Pi$-facial subwalk, the edge $(t_{0'} t_4)\in E(G)$, and  $(t_{0'}t_1 t_2 t_3 t_{4} t_{0'})$ is a $5$ --cycle of $G$.

If we assume that $(t_{0'} t_1 t_2 t_3 t_4)$ is a $\Pi$--facial subwalk, the proof follows in a similar way and we can deduce that $(t_0 t_{4'})\in E(G)$ and $(t_0 t_1 t_2 t_3 t_{4'})$ is a $5$--cycle of $G$.

Notice that, as a consequence of the previous two arguments, it can not happen that both of $(t_{0'} t_4)$ and  $(t_0 t_{4'})$ are not in $E(G)$ simultaneously. Otherwise, this would contradict Claim 1. We will now argue that both $(t_{0'} t_4)$ and $(t_0 t_{4'})$ are in $E(G)$.

Suppose that $(t_{0} t_{4'})\not \in E(G)$, then $(t_{0'} t_{4})\in E(G)$. If $(t_{0'} t_1 t_2 t_3 t_4)$ is a $\Pi$-facial subwalk then, by Claim 4, $(t_0 t_1 t_{0'} t_{5'} t_{4'})$ and $(t_0 t_5 t_4 t_3 t_{4'})$ are $\Pi$--facial subwalks. It follows that, by Proposition \ref{pizzageneralization}, $(t_0 t_1 t_{0'} t_{5'} t_{4'} t_3 t_4 t_5 t_0)$ is a $\Pi$--facial cycle, but  it intersects the $\Pi$--facial subwalk $(t_{0'} t_1 t_2 t_3 t_4)$ in two edges, which contradicts the definition of polyhedral embedding. Hence $(t_{0'} t_1 t_2 t_3 t_4)$ is not a $\Pi$--facial subwalk, and, by Claim 1, $(t_{0} t_1 t_2 t_3 t_{4'})$ has to be a $\Pi$--facial subwalk, implying that the disjunction can be solved. Therefore $(t_{0} t_{4'}) \in E(G)$.

The case where we begin by assuming that  $(t_{0'} t_4)\not \in E(G)$ follows analogously, thus $(t_{0'} t_4)\in E(G)$.

These arguments prove that $(t_0 t_{4'})$ and $(t_{0'} t_4)$ are in $E(G)$, and $(t_0 t_1 t_2 t_3 t_{4'} t_0)$ and $(t_{0'} t_1 t_2 t_3 t_4 t_{0'})$ are $5$--cycles of $G$.
$\boxminus$\\

\emph{Claim 6. The scaffold edges $[t_0 t_2]$, $[t_1 t_3]$, $[t_2 t_{0'}]$, $[t_2 t_{4'}]$, $[t_2 t_4]$ are in $\mathcal{S}$.}

By Claim 1 and 5, and Proposition \ref{edge} either $(t_0 t_1 t_2 t_3 t_{4'} t_0)$ or $(t_{0'} t_1 t_2 t_3 t_4 t_{0'})$ is a $\Pi$-facial cycle. So, if any of edges $[t_0 t_2]$, $[t_1 t_3]$, $[t_2 t_{0'}]$, $[t_2 t_{4'}]$, $[t_2 t_4]$ were not in $\mathcal{S}$, we could easily discard the possibility that one of the two cycles is a $\Pi$-facial cycle, thus, we could solve the disjunction and the claim follows.
$\boxminus$\\

\emph{Claim 7. There is a new vertex $t_x$ adjacent to $t_2, t_5$ and $t_{5'}$.}

By Claim 6, the edges $[t_0 t_2]$, $[t_2 t_{0'}]$, $[t_2 t_{4'}]$, $[t_2 t_4]$ are in $\mathcal{S}$. If for any of this edges there was only one $3-$path that could correspond to it, then, by Proposition \ref{pathtwo}, we could know whether $(t_0 t_1 t_2 t_3 t_{4'} t_0)$ or $(t_{0'} t_1 t_2 t_3 t_4 t_{0'})$ is a $\Pi$-facial cycle, thus we could solve the disjunction. Hence there have to be additional $3$--paths that can correspond to $[t_0 t_2]$, $[t_2 t_{0'}]$, $[t_2 t_{4'}]$, $[t_2 t_4].$

In order to find a second path between $t_2$ and $t_{4'}$ of length three, which is internally disjoint to $(t_2 t_3 t_{4'})$ and $(t_2 t_1 t_0 t_{4'})$, let $t_x$ be the remaining vertex adjacent to $t_2$. Observe that $t_x \neq t_5$, otherwise by Proposition \ref{4} $(t_0 t_1 t_2 t_5 t_0)$ and $(t_2 t_3 t_4 t_5 t_2)$ are $\Pi$--facial cycles and using Claim 1, we can deduce easily if $(t_0 t_1 t_2 t_3 t_{4'})$ or $(t_0 t_1 t_2 t_3 t_{4'})$ is a $\Pi$--facial subwalk, contradicting our hypothesis. Analogously $t_x \neq t_{5'}$. Thus $t_x$ is a new vertex. Notice that $[t_x t_0], [t_x t_4], [t_x t_{0'}], [t_x t_{4'}]$ are in $\mathcal{S}$, otherwise, by Proposition \ref{one}, we can deduce the sub-paths in the $\Pi$-facial subwalks that correspond to $[t_1 t_4]$ and $[t_1 t_{4'}]$, which implies that we can solve the disjunction. Then, the second $3$--path between $t_2$ and $t_{4'}$ contains $t_x$ and $t_{5'}$. This implies that $(t_x t_{5'})$ is an edge of $G$.

Similarly, for $[t_0 t_2]$ there are already two disjoint paths $(t_2 t_1 t_0)$ and $(t_2 t_3 t_{4'} t_0)$ of length two and three respectively. Then, the third path necessarily contains $t_x$ and $t_5$, which implies (as before) that $(t_x t_5)$ is an edge of $G$.
$\boxminus$

Note that we now know the complete underlying cubic graph $G$, namely $G$ is the Petersen's graph.

\emph{Claim 8. $[t_{0} t_{0'}]$, $[t_{0} t_{5'}]$, $[t_{0} t_{4}]$, $[t_{1} t_{5}]$, $[t_{1} x]$, $[t_{1} t_{5'}]$, $[t_{2} t_{5'}]$, $[t_{2} t_{5}]$, $[t_{3} x]$, $[t_{3} t_{5'}]$, $[t_{3} t_{5}]$, $[t_{4'} t_{5}]$, $[t_{4} t_{4'}]$, $[t_{0} t_{4'}]$, $[t_{5} t_{0'}]$, $[t_{4} t_{5'}]$ and $[t_{5} t_{5'}]$ are in $\mathcal{S}$. }

We will use identical arguments to those in Claim 6. For example if $[t_{0} t_{0'}]$ is not in $\mathcal{S}$, then any $3$--path whith start $t_0$ and end $t_{0'}$ can't be a facial subwalk. This implies that by process of elimination and using Proposition \ref{one} and Proposition \ref{pathtwo}, we can deduce the $\Pi$--facial cycles, and the disjunction can be solved. $\boxminus$ 

This finalizes the proof. Therefore, $G$ is Petersen Graph and $\mathcal{S}= E(K_{10}\setminus E(P))$. 
\end{proof}

%%%%%%%%%%% Teorema Mariposa 2 %%%%%%%%%%%%%%%
\begin{theorem}
Let $G^e$ be such that $B_2\subset G^e$ (labelled as Definition \ref{butterfly}, Figure \ref{fig:butterfly}) and such that the disjunction that arises from the fork subgraph of $B_2$ induced by the vertices $\{t_1, t_2, t_3, t_4, t_{4'}\}$ cannot be solved, then $G$ is the Franklin graph, $F$, and $G^e$ is the union of $F$ and the set $\mathcal{S}$ given by \emph{simple} scaffold edges
$[t_0 t_5], [t_5 x_5], [x_5 x_0], [x_0 t_0], [t_1 t_2], [t_2 x_2], [x_2 t_{0'}], [t_{0'} t_1],[t_3 t_{4'}], [t_{4'} t_{5'}], [t_{5'} t_4], [t_4 t_3]$, and the \emph{double} scaffold edges
$[[t_0 t_3]], [[x_0 t_2]], [[t_1 t_{4'}]], [[t_{0'}t_5]], [[t_4 x_2]], [[x_5 t_{5'}]], [[t_0 t_{5'}]], [[t_1 t_4]],$ $[[t_2 t_5]], [[t_3 x_5]], [[t_{4'} x_2]], [[x_0 t_{0'}]]$.
\end{theorem}

\begin{proof}
Observe that in $B_2$, $(t_0  t_1 t_2 t_3 t_4 t_5 t_0)$ and $(t_{0'} t_1 t_2 t_3 t_{4'} t_{5'} t_{0'})$ are two $6$-cycles such that $[t_1 t_4]$, $[t_1 t_{4'}]$, $[t_0 t_3]\in G^e$ and $[t_{0'} t_3]\not\in G^e$, which implies by Proposition \ref{one}, that $(t_0 t_1 t_2 t_3)$ is $\Pi$-facial subwalk of $[t_0 t_3]$. \\

\emph{Claim 1. $[[t_0 t_3]]$, $[[t_1 t_4]]$ and $[[t_2 t_5]]$ are double scaffold edges.}

If one of $[[t_0 t_3]]$, $[[t_1 t_4]]$ or $[[t_2 t_5]]$ is not a double scaffold edge, then $(t_1 t_2 t_3 t_4)$ cannot be a $\Pi$-facial subwalk, otherwise by Definition \ref{poly} and Proposition \ref{pathtwo}, this would imply that $(t_0 t_1 t_2 t_3 t_4 t_5 t_0)$ is a $\Pi$-facial cycle, but this is impossible as we started by assuming that at least one of  $[[t_0 t_3]]$, $[[t_1 t_4]]$ or $[[t_2 t_5]]$ is not a double scaffold edge.

It follows by Proposition \ref{one} that $(t_1 t_2 t_3 t_{4'})$ is the $\Pi$-facial subwalk that corresponds to the scaffold edge $[t_1 t_{4'}]$, and then, by Proposition \ref{one} and Corollary \ref{otro}, the $\Pi$-facial subwalk corresponding to the edge $[t_1 t_4]$ should be $(t_1 t_0 t_{5} t_4)$, which implies that the disjunction could be solved, contradicting our hypothesis. Hence the claim holds. $\boxminus$
\\
\\
\emph{Claim 2. For each pair of vertices $\{t_0,t_3\}$, $\{t_1,t_4\}$ and $\{t_2,t_5\}$, there is a third $3$-path between them whose vertices are disjoint to the cycle $(t_0 t_1 t_2 t_3 t_4 t_5 t_0)$.}

Suppose that for at least one of the pairs of vertices mentioned above, there are only two $3$--paths, by Proposition \ref{pizzageneralization}, $(t_0 t_1 t_2 t_3 t_4 t_5 t_0)$ is $\Pi$-facial cycle and then $(t_1 t_2 t_3 t_4)$ is the $\Pi$-facial subwalk corresponding to $[t_1 t_4]$, i.e., we can solve the disjunction, contradicting our hypothesis. Hence the claim holds. $\boxminus$
\\
\\
Since $G$ is a cubic graph, let $x_i$ be the remaining vertex adjacent to $t_i$ for $i\in\{0, 2, 5\}$, then:\\

\emph{Claim 3. The vertex $x_0$ is different to the vertices $t_i$ for all $i\in \{0, 1, 2, 3, 4, 5, 0', 4', 5' \}$ and adjacent to $t_{4'}$.}
\\

First, in order to prove that all vertices are different we proceed by contradiction:
\begin{enumerate}
\item ($x_0 \neq t_{0'}$.) If $x_0 = t_{0'}$, then by Proposition \ref{3}, $(t_0 t_{0'} t_1 t_0)$ is a $\Pi$-facial cycle, this implies, by Proposition \ref{one}, that $(t_5 t_0 t_1 t_2)$ is a $\Pi$-facial subwalk. Since $(t_0 t_1 t_2 t_3)$ is $\Pi$-facial subwalk, by Definition \ref{poly}, necessarily $(t_5 t_0 t_1 t_2 t_3 )$ is a $\Pi$-facial subwalk, and using Proposition \ref{pathtwo}, $(t_0 t_1 t_2 t_3 t_4 t_5 t_0)$ is a $\Pi$-facial cycle. The latter implies that $(t_1 t_2 t_3 t_4)$ is the $\Pi$-facial subwalk corresponding to $[t_1 t_4]$ and thus we can solve the disjunction, contradicting our hypothesis. Hence  $x_0 \neq t_{0'}$. 

\item ($x_0 \neq t_2$.) If $x_0 = t_2$, by Proposition \ref{3}, $(t_0 t_1 t_2 t_0)$ is a $\Pi$--facial cycle, but this contradicts, by Definition \ref{poly}, the fact that $(t_0 t_1 t_2 t_3)$ is a $\Pi$--facial subwalk. Hence $x_0 \neq t_2$.

\item ($x_0 \neq t_{4}$.) If $x_0 = t_4$, since $(t_0 t_1 t_2 t_3)$ is a $\Pi$--facial subwalk, Proposition \ref{pathtwo} implies that $(t_0 t_1 t_2 t_3 t_4 t_0)$ is a $\Pi$--facial cycle, and $(t_1 t_2 t_3 t_4)$ is the $\Pi$--facial subwalk corresponding to $[t_1 t_4]$. Thus, the disjunction can be solved, contradicting our hypothesis. Hence,  $x_0 \neq t_{4}$. 

\item ($x_0 \neq t_3$.) If $x_0 = t_3$, then $t_3$ would have degree four and $G$ would not be a cubic graph.

\item ($x_0 \neq t_{4'}$.) If $x_0 = t_{4'}$, since $(t_0 t_1 t_2 t_3)$ is a $\Pi$-facial subwalk, Definition \ref{poly} and Proposition \ref{pathtwo} imply that $(t_0 t_1 t_2 t_3 t_{4'} t_0)$ is a $\Pi$-facial cycle and that $(t_1 t_2 t_3 t_{4'})$ is the $\Pi$-facial subwalk corresponding to $[t_1 t_{4'}]$. Thus we can solve the disjunction, contradicting our hypothesis. Hence  $x_0 \neq t_{4'}$. 

\item ($x_0 \neq t_{5'}$.) If $x_0 = t_{5'}$, then by Proposition \ref{4}, $(t_0 t_1 t_{0'} t_{5'} t_0)$ is a $\Pi$-facial cycle; the latter and Proposition \ref{one} imply that $(t_2 t_1 t_0 t_5)$ is a $\Pi$-facial subwalk, and by Definition \ref{poly}, since $(t_0 t_1 t_2 t_3 )$ is a $\Pi$-facial subwalk, $(t_5 t_0 t_1 t_2 t_3)$ is a $\Pi$-facial subwalk. By Proposition \ref{pathtwo}, it follows that $(t_0 t_1 t_2 t_3 t_4 t_5 t_0)$ is a $\Pi$-facial cycle and $(t_1 t_2 t_3 t_4)$ is the $\Pi$-facial subwalk that corresponds to $[t_1 t_4]$. The previous statements imply that we can solve the disjunction, contradicting our hypothesis. Hence  $x_0 \neq t_{5'}$. 
\end{enumerate}

Finally, Claim 2 implies $x_0$ and $t_{4'}$ are adjacent.$\boxminus$

\emph{Claim 4. The vertex $x_2$ is different to the vertices $t_i$ for all $i\in\{0, 1, 2, 3, 4, 5, 0', 4', 5'\}$ and $x_0$.}

We proceed by contradiction:
\begin{enumerate}
    \item ($x_2 \neq t_0.$) If $x_2 = t_0$ then $t_0$ would have degree four and contradict the fact that $G$ is a cubic graph. Hence  $x_2 \neq t_{0}$.
    
    \item ($x_2 \neq t_{0'}.$) If $x_2 = t_{0'}$, by Proposition \ref{3}, $(t_{0'} t_1 t_2 t_{0'})$ is a $\Pi$--facial cycle. Thus, by the previous statement and Definition \ref{poly}, $(t_3 t_2 t_{0'} t_{5'})$ is a $\Pi$--facial subwalk. By the previous statement and Proposition \ref{pathtwo}, $(t_3 t_2 t_{0'} t_{5'} t_{4'} t_3)$ is a $\Pi$--facial cycle. The latter and Definition \ref{poly}, imply that $(t_1 t_2 t_3 t_{4'})$ is not a $\Pi$--facial subwalk. This, in turn, implies that $(t_1 t_2 t_3 t_4)$ is a $\Pi$--facial subwalk, and the disjunction can be solved, contradicting our hypothesis. Hence  $x_2 \neq t_{0'}$. 
    
    \item ($x_2 \neq t_{4}.$) If $x_2 = t_4$, by Proposition \ref{3}, $(t_2 t_3 t_4 t_2)$ is a $\Pi$--facial cycle. The latter together with Definition \ref{poly} imply $(t_1 t_2 t_3 t_4)$ is not a $\Pi$--facial subwalk, thus $(t_1 t_2 t_3 t_{4'})$ is a $\Pi$--facial subwalk and the disjunction can be solved, contradicting our hypothesis. Hence  $x_2 \neq t_{4}$. 
    
    \item ($x_2 \neq t_{4'}.$) The proof follows identical arguments to those used in the previous case. 
    
    \item ($x_2 \neq t_{5}.$) If $x_2 = t_5$, by Proposition \ref{4}, $(t_2 t_3 t_4 t_5 t_2)$ is a $\Pi$--facial cycle. The latter and Definition \ref{poly}, imply that $(t_1 t_2 t_3 t_4)$ is not a $\Pi$--facial subwalk, thus $(t_1 t_2 t_3 t_{4'})$ is a $\Pi$--facial subwalk and the disjunction can be solved, contradicting our hypothesis. Hence  $x_2 \neq t_{5}$. 
    
    \item ($x_2 \neq t_{5'}.$) The proof is similar to the previous case. 

    \item ($x_2 \neq x_{0}.$) If $x_2 = x_0$, then by Proposition \ref{4}, $(t_0 t_1 t_{2} x_0 t_0)$ is a $\Pi$-facial cycle. Thus the $\Pi$--facial cycle $(t_0 t_1 t_{2} x_0)$ intersects to the $\Pi$--facial subwalk $(t_0 t_1 t_2 t_3)$ in three vertices contradicting that we have a polyhedral embedding.
\end{enumerate}
$\boxminus$

\emph{Claim 5. The vertex $x_5$ is different to the vertices $t_i$ for all $i\in\{0, 1, 2, 3, 4, 5, 0', 4', 5'\}$ and $x_0, x_2$, and adjacent to $x_2$.}

Once more, we proceed by contradiction to prove that the vertices are different:
\begin{enumerate}
    \item ($x_5 \neq t_{0'}$.) If $x_5 = t_{0'}$, then Proposition \ref{4} implies that $(t_5 t_0 t_1 t_{0'} t_5)$ is a $\Pi$-facial cycle. The latter and Definition \ref{poly} imply that $(t_1 t_0 t_5 t_4)$ is not a $\Pi$-facial subwalk. Then, by Corollary \ref{otro}, it follows that $(t_1 t_2 t_3 t_4)$ is the $\Pi$-facial subwalk corresponding to the edge $[t_1 t_4]$. However, using Proposition \ref{pathtwo} we can deduce that $(t_0 t_1 t_2 t_3 t_4 t_5 t_0)$ is a $\Pi$--facial cycle which intersects the $\Pi$--facial cycle $(t_5 t_0 t_1 t_{0'} t_5)$ in two edges, contradicting the polyhedrality of the embedding. In conclusion, this case can never occur. 
    
    \item ($x_5 \neq t_{2},$ $x_5 \neq t_{1},$ $x_5 \neq t_{3},$ $x_5 \neq t_{4'},$.) If any of the inequalities was an equality instead we would have a contradiction on the fact that $G$ is a cubic graph.

    \item ($x_5 \neq x_{0}$.) If $x_5 = x_0$, then it is only possible to find two internally disjoint $3$-paths between $t_2$ and $t_5$. Since $[[t_2 t_5]]$ is a double scaffold edge and $(t_2 t_5)\not\in E(G)$ (otherwise contradicts that $G$ is a cubic graph), Proposition \ref{pizzageneralization} implies that $(t_0 t_1 t_2 t_3 t_4 t_5 t_0)$ is a $\Pi$-facial cycle. In turn, the latter implies that $(t_1 t_2 t_3 t_4)$ is the $\Pi$-facial subwalk corresponding to $[t_1 t_4]$ and that $(t_1 t_{0'} t_{5'} t_{4'})$ is a $\Pi$-facial subwalk, thus we can solve the disjunction, contradicting our hypothesis. Hence  $x_5 \neq x_{0}$. 
    
    \item ($x_5 \neq x_{2}$.) If $x_5 = x_2$, Proposition \ref{pathtwo} implies that $(t_0 t_1 t_2 t_3 t_4 t_5 t_0)$ is not a $\Pi$-facial cycle. Since $(t_0 t_1 t_2 t_3)$ is a $\Pi-$facial subwalk, using Proposition \ref{3path} we can deduce that $(t_0 t_1 t_2 t_3 t_{4'} x_0 t_0)$ is a $\Pi$-facial cycle; i.e., $(t_1 t_2 t_3 t_{4'})$ is the $\Pi-$facial subwalk that corresponds to $[t_1 t_{4'}]$. The latter and Corollary \ref{otro} imply that $(t_1 t_0 t_5 t_4)$ is the $\Pi$-facial subwalk corresponding to $[t_1 t_4]$, and then we can solve the disjunction, contradicting our hypothesis. Hence  $x_5 \neq x_{2}$.
    
    \item ($x_5 \neq t_{5'}$.) If $x_5 = t_{5'}$, it is only possible to find two $3$-paths internally disjoint between them $t_2$ and $t_5$. Since $[[t_2 t_5]]$ is double and $(t_2 t_5)\not\in E(G)$ (otherwise contradicts that $G$ is a cubic graph), Proposition \ref{pizzageneralization} implies that $(t_0 t_1 t_2 t_3 t_4 t_5 t_0)$ is a $\Pi$-facial cycle. Hence, $(t_1 t_2 t_3 t_4)$ is the $\Pi$-facial subwalk corresponding to $[t_1 t_4]$. Using the latter and Corollary \ref{otro} is easy to see  that $(t_1 t_{0'} t_{5'} t_{4'})$ is the $\Pi$-facial subwalk corresponding to $[t_1 t_{4'}]$. That is  we can solve the disjunction, contradicting our hypothesis. Hence  $x_5 \neq x_{5'}$.
\end{enumerate}
Finally,  $x_2$ and $x_5$ are adjacent by Claim 2. $\boxminus$

\emph{Claim 6. The vertex $x_4$ is different to the vertices $t_i$, for all $i\in\{0, 1, 2, 3, 4, 5, 0', 4'\}$ and $x_0, x_2, x_5$.}

We proceed by contradiction:
\begin{enumerate}
    \item ($x_4 \neq t_0$, $x_4 \neq t_1$, $x_4 \neq t_2$, $x_4 \neq t_{4'}$). If any of the inequalities was an equality instead we would have a contradiction on the fact that $G$ is a cubic graph.

    \item ($x_4 \neq t_{0'}$.) If $x_4 = t_{0'}$ then there is a $2$-path $(t_1 t_{0'} t_4)$ and two $3$--paths, $(t_1 t_2 t_3 t_4)$ and $(t_1 t_0 t_5 t_4)$ between $t_1$ and $t_4$. Since $[[t_1 t_4]]$ is a double scaffold edge, and there is no edge $(t_1 t_4)$, by Proposition \ref{pizzageneralization} $(t_0 t_1 t_2 t_3 t_4 t_5 t_0)$is a $\Pi$--facial cycle. But this contradicts Proposition \ref{pathtwo} since there is a $2$--path $(t_1 t_{0'} t_4)$ between $t_1$ and $t_4$. Hence  $x_4 \neq x_{0'}$.
    
    \item ($x_4 \neq x_0$.) If $x_4 = x_0$, Proposition \ref{4} implies that $(t_0 t_5 t_4 x_0 t_0)$ is a $\Pi$-facial cycle. The latter, Proposition \ref{one} and Corollary \ref{otro} imply that $(t_1 t_2 t_3 t_4)$ is the $\Pi$-facial subwalk corresponding to $[t_1 t_4]$ and we can solve the disjunction, contradicting our hypothesis. Hence  $x_4 \neq x_{0}$.
    
    \item ($x_4\neq x_2 $.) If $x_4=x_2 $, then Proposition \ref{4} implies that $(t_2 t_3 t_4 x_2 t_2)$ is a $\Pi$-facial cycle. The latter and Proposition \ref{one} imply that $(t_1 t_2 t_3 t_{4'})$ is the $\Pi$-facial subwalk corresponding to $[t_1 t_{4'}]$. Subsequently, using Corollary \ref{otro} it is possible to deduce that $(t_1 t_0 t_5 t_4)$ is the $\Pi$-facial subwalk corresponding to $[t_1 t_4]$, that is,  we can solve the disjunction, contradicting our hypothesis. Hence  $x_4 \neq x_{2}$.
    
    \item ($x_4\neq x_5 $.) If $x_4 = x_5$, using  Proposition \ref{3} we deduce that $(t_4 t_5 x_5 t_4)$ is a $\Pi$--facial cycle. In addition to this, using Proposition \ref{one} we can prove that $(t_0 t_5 t_4 t_3)$ is a $\Pi$-facial subwalk, ussing that $(t_0 t_3)\not \in E(G)$ and applying Proposition \ref{pizzageneralization}, we have $(t_0 t_1 t_2 t_3 t_4 t_5 t_0)$ is a $\Pi$-facial cycle. The latter implies that $(t_1 t_2 t_3 t_4)$ is the $\Pi$-facial subwalk corresponding to $[t_1 t_4]$. Using that $(t_1 t_2 t_3 t_4)$ is the $\Pi$-facial subwalk corresponding to $[t_1 t_4]$ and Corollary \ref{otro} we arrive to the conclusion that $(t_1 t_{0'} t_{5'} t_{4'})$ is the $\Pi$-facial subwalk corresponding to $[t_1 t_{4'}]$. Thus, we can solve the disjuntion, contradicting our hypothesis. Hence $x_4 \neq x_5$. $\boxminus$
\end{enumerate}

\emph{Claim 7. $[[t_1 t_{4'}]]$ is a double scaffold edge.}

Since  $[t_1 t_{4'}] \in \mathcal{S}$ and between $t_1$ and $t_{4'}$ there are three internally disjoint $3$-paths, $(t_1 t_2 t_3 t_{4'})$, $(t_1 t_{0'} t_{5'} t_{4'})$ and $(t_1 t_0 x_0 t_{4'})$, by Proposition \ref{3path}, necessarily $[[t_1 t_{4'}]]$ is double. $\boxminus$

\emph{Claim 8. $(t_1 t_0 x_0 t_{4'})$ is a $\Pi$-facial subwalk.}

Since $(t_0 t_1 t_2 t_3)$ is a $\Pi$-facial subwalk and $[[t_1 t_{4'}]]$ is a double scaffold edge (by Claim 7), then Corollary \ref{pelito}, implies that $(t_1 t_0 x_0 t_{4'})$ is a $\Pi$-facial subwalk.$\boxminus$
\\

\emph{Claim 9. $[[t_0 t_{5'}]]$, $[[x_0 t_{0'}]]$, $[[t_0 t_3]]$ and $[[x_0 t_2]]$ are double scaffold edges and there exist three internally disjoint paths between each pair of end vertices.}

If some edge, $[t_0 t_{5'}]$, $[x_0 t_{0'}]$, $[t_0 t_3]$ or $[x_0 t_2]$ does not belong to $\mathcal{S}$, or it is not double or they are all double but for one of them there exist only two $3$-paths between its end vertices, then by Proposition \ref{pizzageneralization} we know if $(t_1 t_2 t_3 t_{4'})$ is $\Pi$-facial subwalk or not. This would implies that the disjunction can be solved, contradicting the hypothesis. $\boxminus$
\\
Let $x_{5'}$ be the remaining vertex adjacent to $t_{5'}$.

\emph{Claim 10. $(t_0 t_1 t_{0'} t_{5'})$ is a $\Pi$-facial subwalk.}

Notice that for the double scaffold edge $[[t_0 t_{5'}]]$ there exist three internally disjoint $3-$paths, $(t_0 x_0 t_{4'} t_{5'})$, $(t_0 t_5 x_{5'} t_{5'})$ and $(t_0 t_1 t_{0'} t_{5'})$. Since $(t_1 t_0 x_0 t_{4'})$ is a $\Pi$-facial subwalk, then by Corollary \ref{pelito}, necessarily $(t_0 x_0 t_{4'} t_{5'} t_{0'} t_1 t_0)$ or $(t_0 t_5 x_5 t_{5'} t_{0'} t_1 t_0)$ is a $\Pi$-facial cycle, i.e., $(t_0 t_1 t_{0'} t_{5'})$ is a $\Pi$-facial subwalk. $\boxminus$
\\
\\
\emph{Claim 11. $(t_1 t_0 t_5 t_4)$ is a $\Pi$-facial subwalk.}

Since $[[t_1 t_4]]$ is a double scaffold edge, there exist three internally disjoint $3-$paths : $(t_1 t_0 t_5 t_4)$, $(t_1 t_{0'} x_4 t_4)$ and $(t_1 t_2 t_3 t_4 )$. Since $(t_0 t_1 t_2 t_3)$ is a $\Pi$-facial subwalk, then by Corollary \ref{pelito}, $(t_1 t_0 t_5 t_4)$ is a $\Pi$-facial subwalk. $\boxminus$\\

\emph{Claim 12. $x_4 = t_{5'}$}.

If $x_4 \neq t_{5'}$, then $(t_1 t_0 t_5 t_4 x_4 t_{0'} t_1)$ is not a $\Pi$-facial cycle because $(t_0 t_1 t_{0'} t_{5'})$ is a $\Pi$-facial subwalk, and this would contradict Definition \ref{poly}. In addition, Claim 11 implies that $(t_1 t_0 t_5 t_4 t_3 t_2 t_1)$ is a $\Pi$-facial cycle and $(t_1 t_{0'} t_{5'} t_{4'})$ is the $\Pi$-facial subwalk of $[t_1 t_{4'}]$. Thus the disjunction could be solved, contradicting our hypothesis. $\boxminus$ \\

\emph{Claim 13. $(t_2 t_3 t_4 t_5)$ is a $\Pi$-facial subwalk.}

Notice that for the double scaffold edge $[[t_2 t_5]]$ there are three internally disjoint $3$-paths  $(t_2 t_1 t_0 t_5)$, $(t_2 x_2 x_5 t_5)$ and $(t_2 t_3 t_4 t_5)$. Since $(t_0 t_1 t_2 t_3)$ is a $\Pi$-facial subwalk, by Corollary \ref{pelito}, $(t_2 t_3 t_4 t_5)$ is a $\Pi$-facial subwalk.$\boxminus$

\emph{Claim 14. $[[t_3 x_5]]$, $[[t_4 x_2]]$, $[[t_0 t_3]]$ and $[[t_1 t_4]]$ are double scaffold edges in $\mathcal{S}$ and there exist three internally disjoint paths between its end vertices. }

If one of $[t_3 x_5]$, $[t_4 x_2]$, $[t_0 t_3]$ or $[t_1 t_4]$ does not belong in $\mathcal{S}$, or it is not a double scaffold edge, or they are all double but for one of them there exist only two $3$-paths between its end vertices, then by Proposition \ref{pizzageneralization}, we know if $(t_1 t_2 t_3 t_4)$ is a $\Pi$--facial subwalk or not. Consequently, using Proposition \ref{3path}, we could conclude if the $\Pi$--facial subwalk corresponding to the edge $[t_1 t_4]$ is either $(t_1 t_2 t_3 t_4)$ or $(t_1 t_{0'} t_{5'} t_4)$ thus we could solve the disjunction, contradicting our hypothesis. $\boxminus$ \\

Observe that for the edges $[[t_0 t_3]]$ and $[[t_1 t_4]]$ there already exist three internally disjoint paths between their end vertices. For $[[t_3 x_5]]$ and $[[t_4 x_2]]$ we need to find additional edges of $G$ to complete their corresponding $3$-paths.
 
 \emph{Claim 15. $(x_0 x_5)\in E(G)$}

Note that for $[[t_3 x_5]]$ there are already two $3$-paths, $(t_3 t_2 x_2 x_5)$ and $(t_3 t_4 t_5 x_5)$. Since the only adjacent vertex  to $t_3$ that is not in either of the paths already mentioned is $t_{4'}$, then $t_{4'}$ must belong to the third $3$-path. The third $3$-path can continue through either $t_{5'}$ or $x_0$. However, the vertex $t_{5'}$ is adjacent to $t_4, t_{4'}$ and $t_{0'}$, then the third $3$-path necessarily has to continue to $x_0$. Since $x_0$ has degree two thus far, the only option left now is that it is adjacent to $x_5$. $\boxminus$
 
 \emph{Claim 16. $(x_2 x_{0'})\in E(G)$}
 
 The proof of this statement follows in a manner similar to the proof of Claim 15. $\boxminus$
 
\emph{Claim 17. $[[x_2 t_{4'}]]$, $[[t_{5'} x_5]]$ and $[[t_5 t_{0'}]]$ are double scaffold edges.}

Notice that if $[x_2 t_{4'}]\not \in \mathcal{S}$, by Proposition \ref{one}, $(t_1 t_2 t_3 t_{4'})$ is the $\Pi$-facial subwalk corresponding to $[t_1 t_{4'}]$, and then the disjunction can be solved, contradicting the hypothesis. Hence $[x_2 t_{4'}]\in \mathcal{S}$. Since there are three internally disjoint $3$-paths between its end vertices, by Proposition \ref{3path}, $[[x_2 t_{4'}]]$has to be a double scaffold edge. 
The proof follows similarly for $[[t_{5'} x_5]]$ and $[[t_5 t_{0'}]]$. $\boxminus$\\

\emph{Claim 18. The edges $[t_1 t_{0'}]$, $[t_{0'} x_2]$, $[x_2 t_2]$, $[t_1 t_2]$, $[t_3 t_{4'}]$, $[t_{4'} t_{5'}]$, $[x_4 t_4]$, $[t_3 t_4]$, $[x_0 t_0]$, $[t_0 t_5]$, $[t_5 x_5]$ and $[x_5 x_0]$ are in $\mathcal{S}$.}

Since $(t_1 t_2 x_2 t_{0'} t_1)$, $(t_0 t_5 x_5 x_0 t_0)$ and $(t_3 t_4 t_{5'} t_{4'} t_3)$ are $\Pi$-facial cycles (by Proposition \ref{4}), the edges $[t_1 t_{0'}]$, $[t_{0'} x_2]$, $[x_2 t_2]$, $[t_1 t_2]$, $[t_3 t_{4'}]$, $[t_{4'} t_{5'}]$, $[x_4 t_4]$, $[t_3 t_4]$, $[x_0 t_0]$, $[t_0 t_5]$, $[t_5 x_5]$ and $[x_5 x_0]$ are in $\mathcal{S}$. \\

Claim 18, completes the proof of the Theorem.
\end{proof}

\section{Conclusion}

The main aim of the project, that this paper initiates, is to develop an algorithmic procedure for constructing all the extended graphs of a given cubic graph.

Thus, the next step is characterizing when a set of scaffold edges build on top of a cubic graph is indeed an extended graph.

\newpage
\begin{appendices}
%%%%%%%%%%%%%%%%LEMA12%%%%%%%%%%%%%%%%%%%%
\newpage
\section*{Theorem 15. CASE A}
%%%%%%%%%%%%%%A%%%%%%%%%%%%%%
\begin{table}[h!]
\centering
\begin{tabular}{ |c|c|}
\hline
Claim 1. &
Claim 2. \\
\includegraphics[height=80mm]{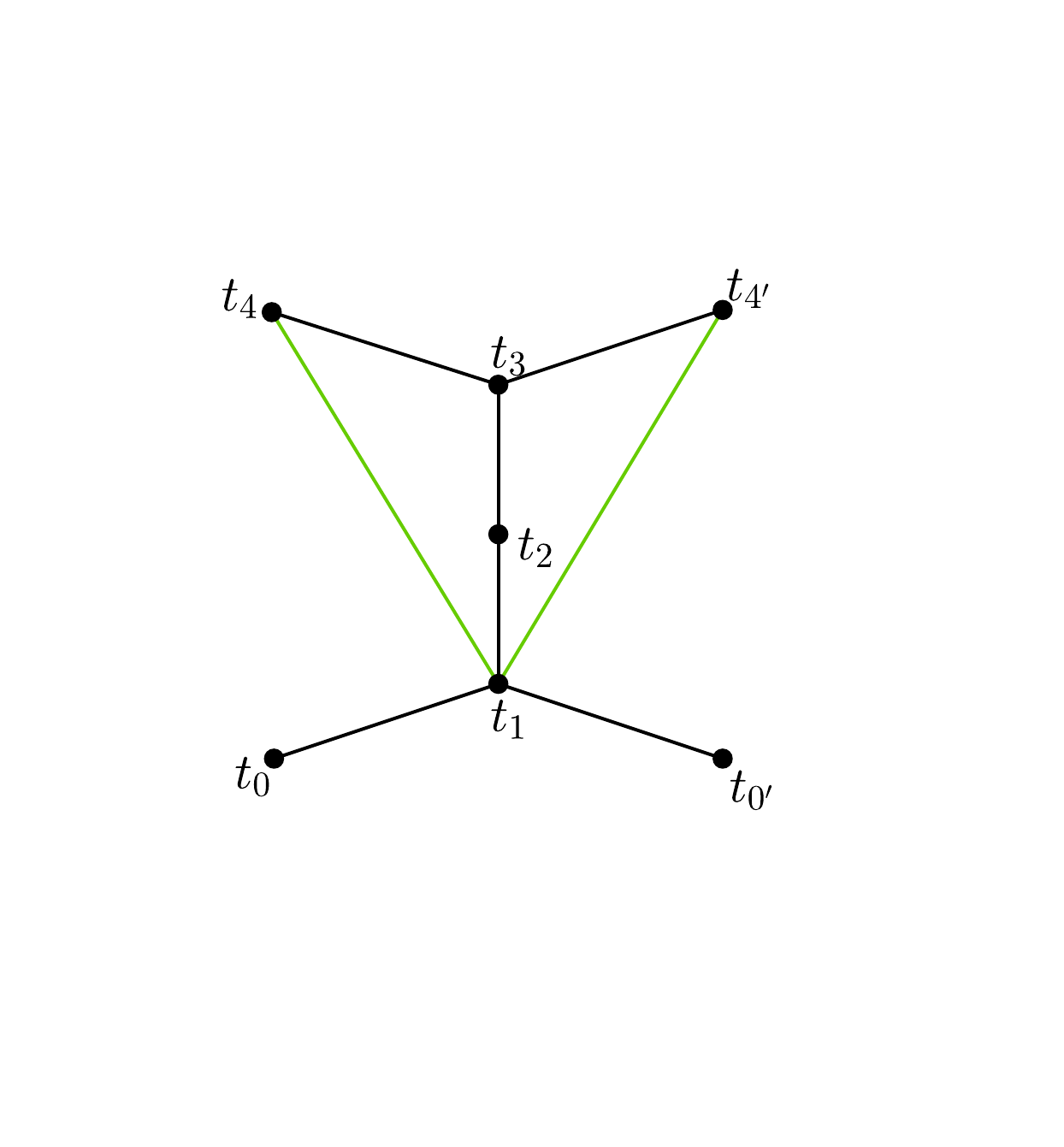}& 
\includegraphics[height=80mm]{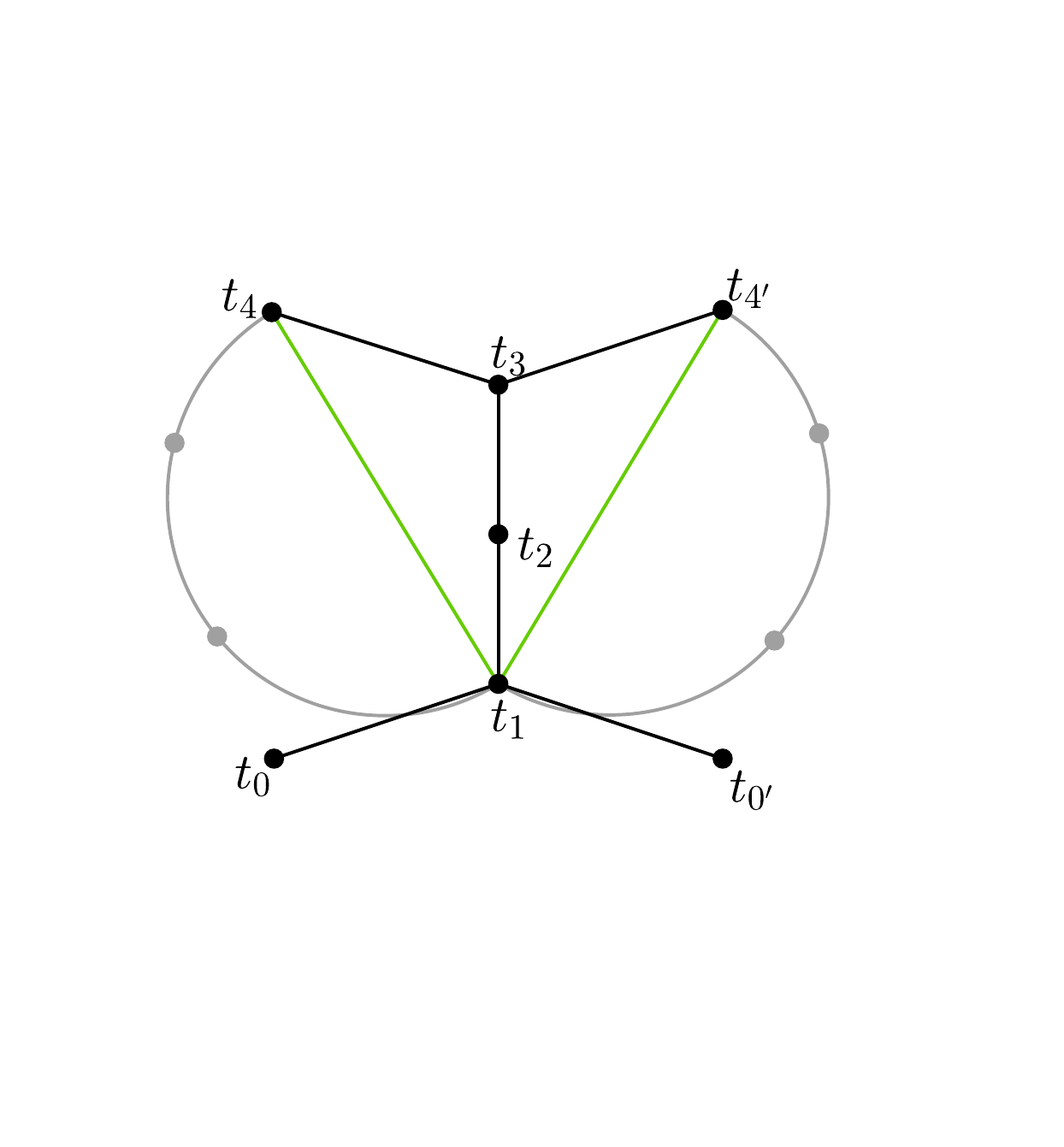}
\\
\hline
Claim 3. &
CASE A. \\
\includegraphics[height=80mm]{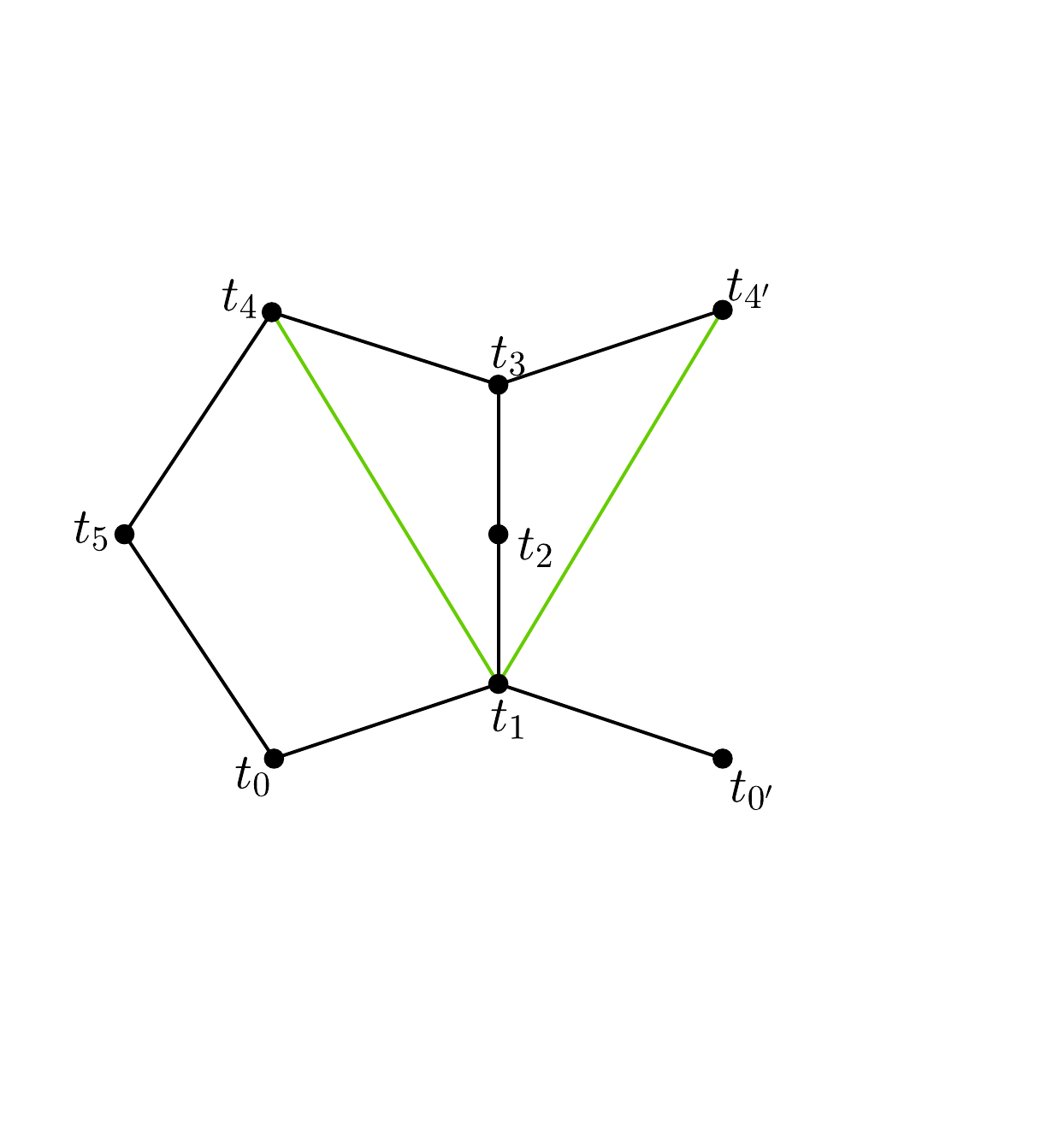}&
\includegraphics[height=80mm]{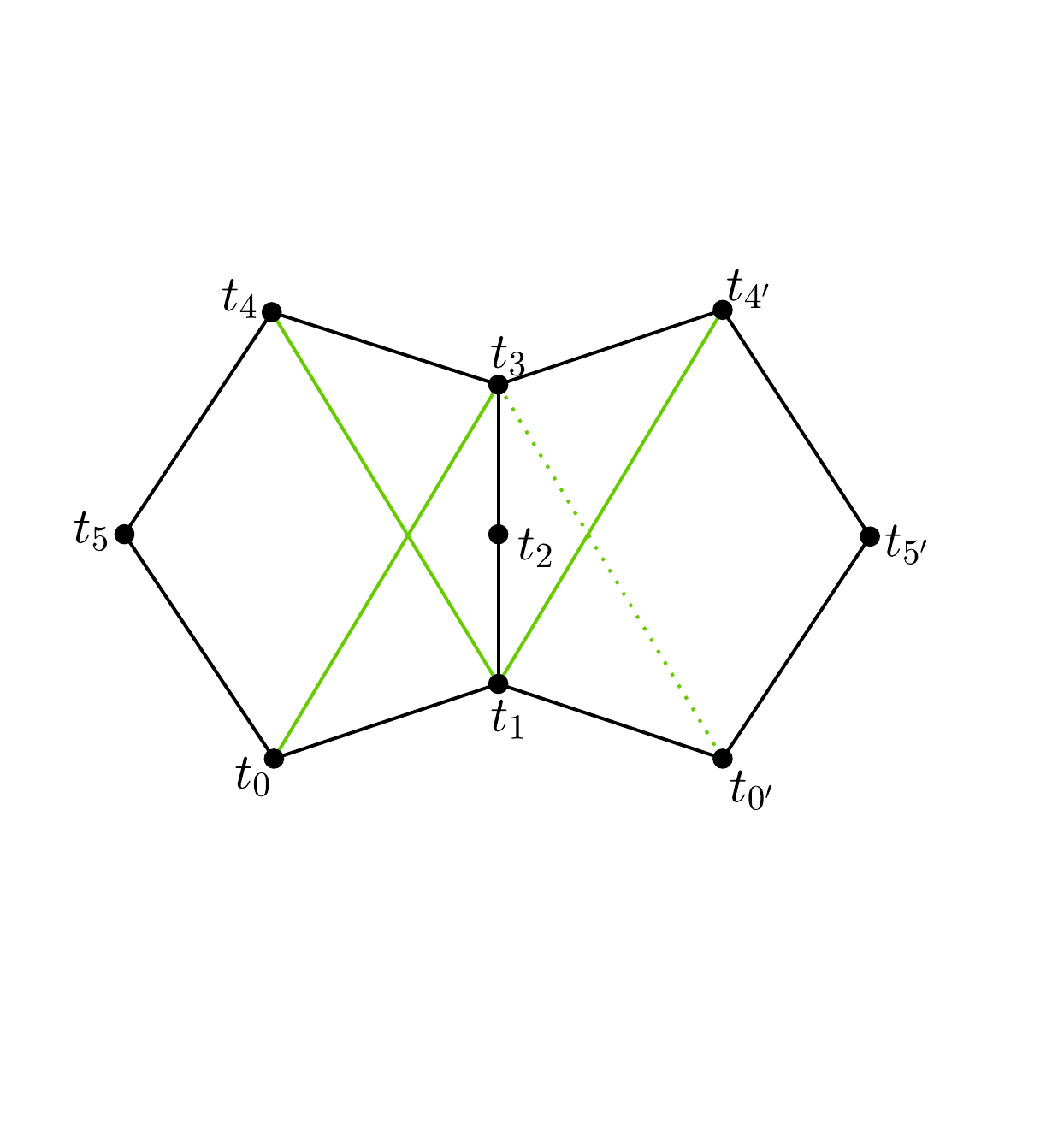}
\\
\hline
\end{tabular}
\end{table}
%%%%%%%%%%%%%%%%%%%%%%%%%%%%%%%%%%%%
\newpage
\section*{Theorem 15. CASE B}
\begin{table}[h!]
\centering
\begin{tabular}{| c | c | c |}
\hline
Claim 1. & 
Claim 2. & 
Claim 3. \\
\includegraphics[height=33mm]{claim1.pdf}&
\includegraphics[height=33mm]{claim2.pdf}&
\includegraphics[height=33mm]{claim3.pdf}\\
\hline
CASE B. &
Case B1.&
Claim B1.1.
\\
\includegraphics[height=33mm]{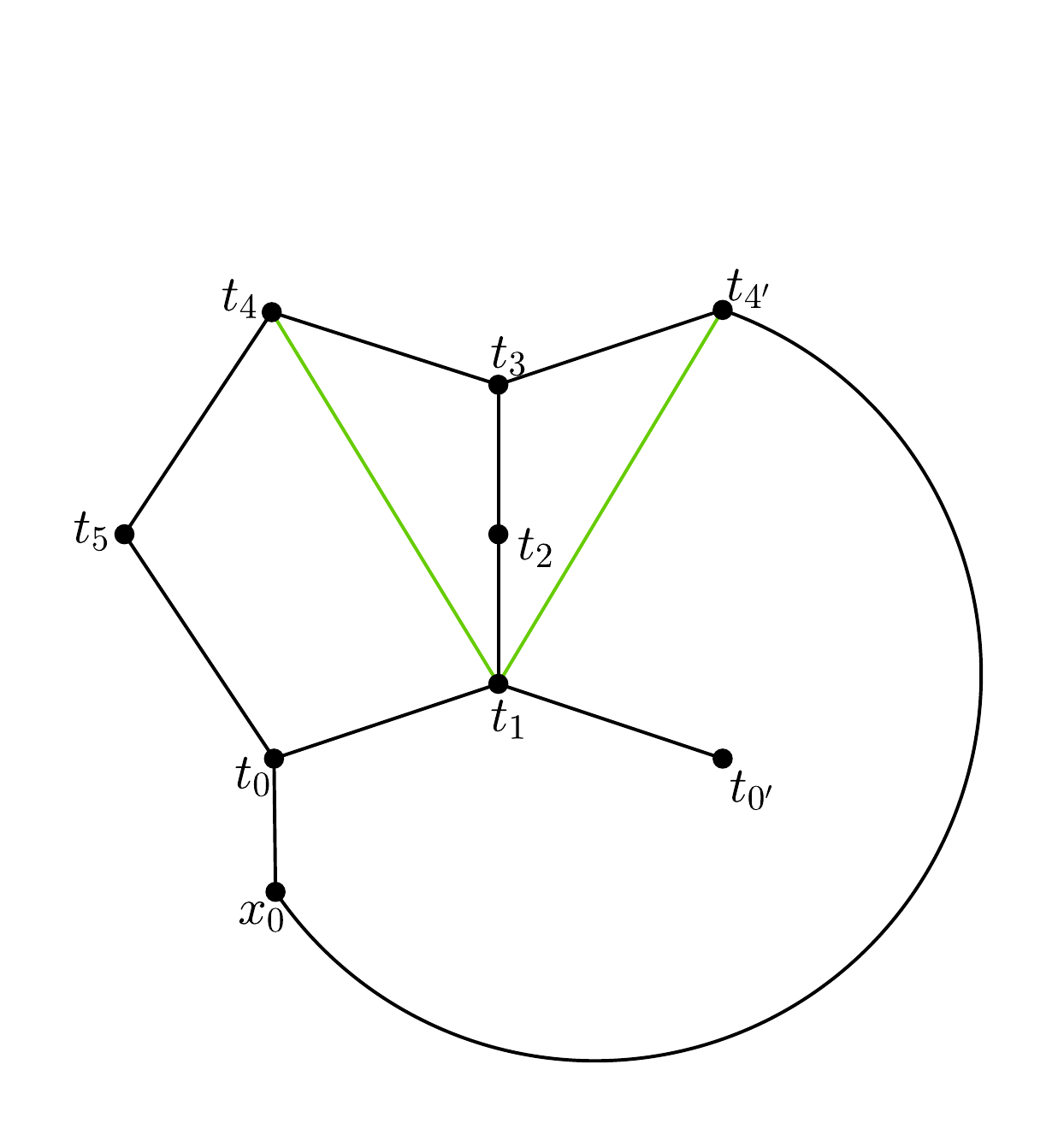} &
\includegraphics[height=33mm]{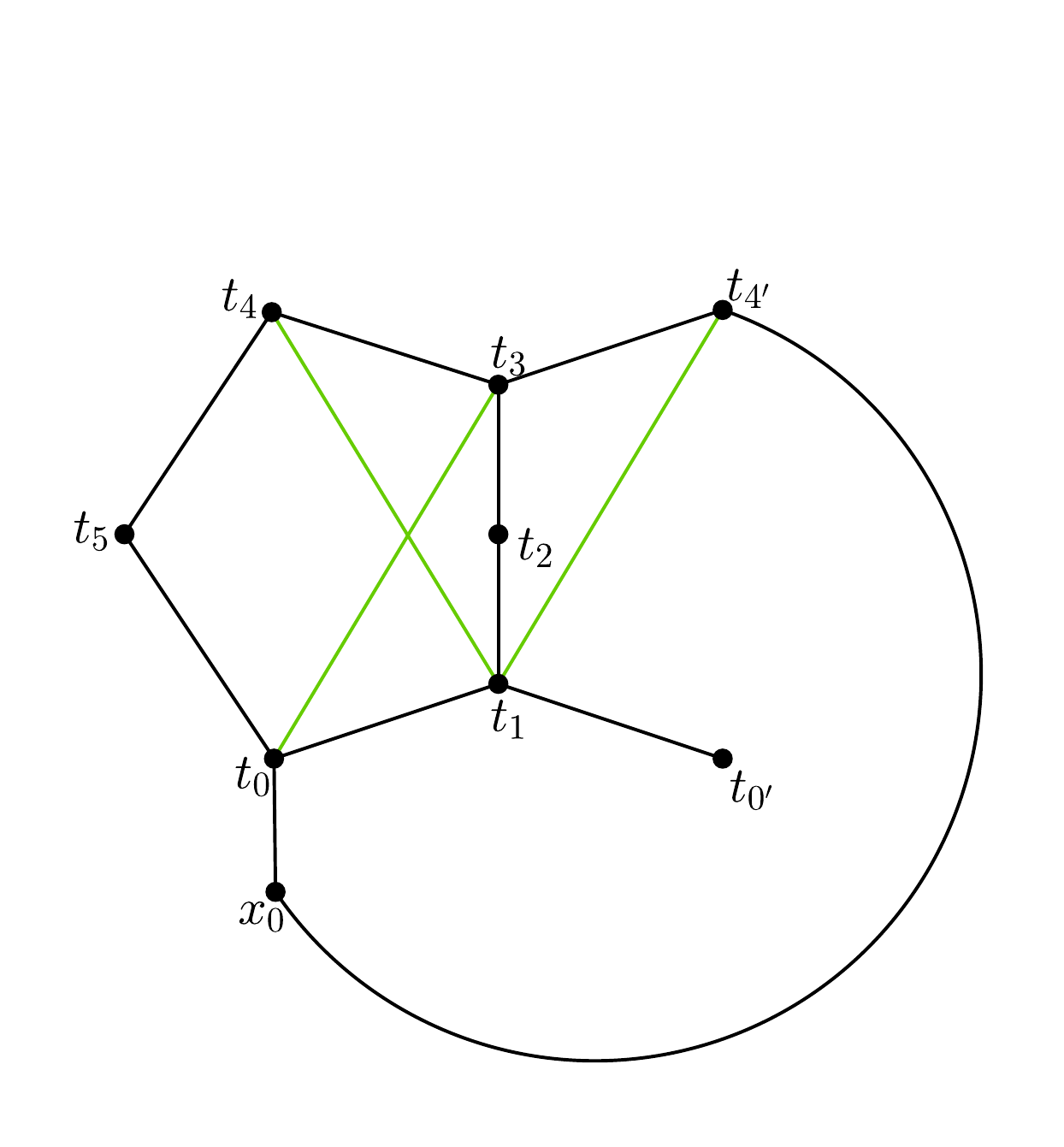} & 
\includegraphics[height=33mm]{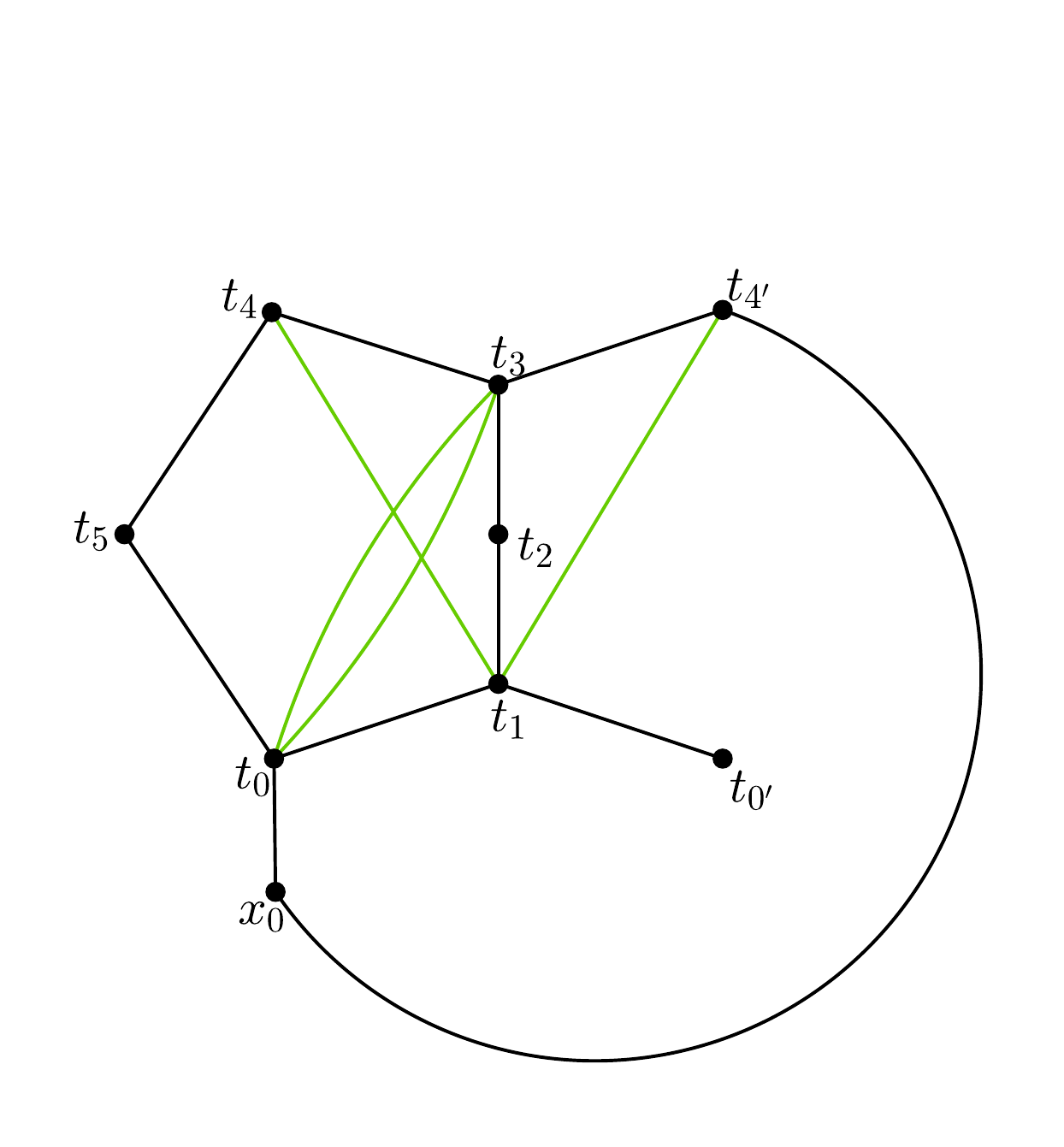} \\
\hline
Claim B1.2. &
Claim B1.3.&
Case B2. \\
\includegraphics[height=33mm]{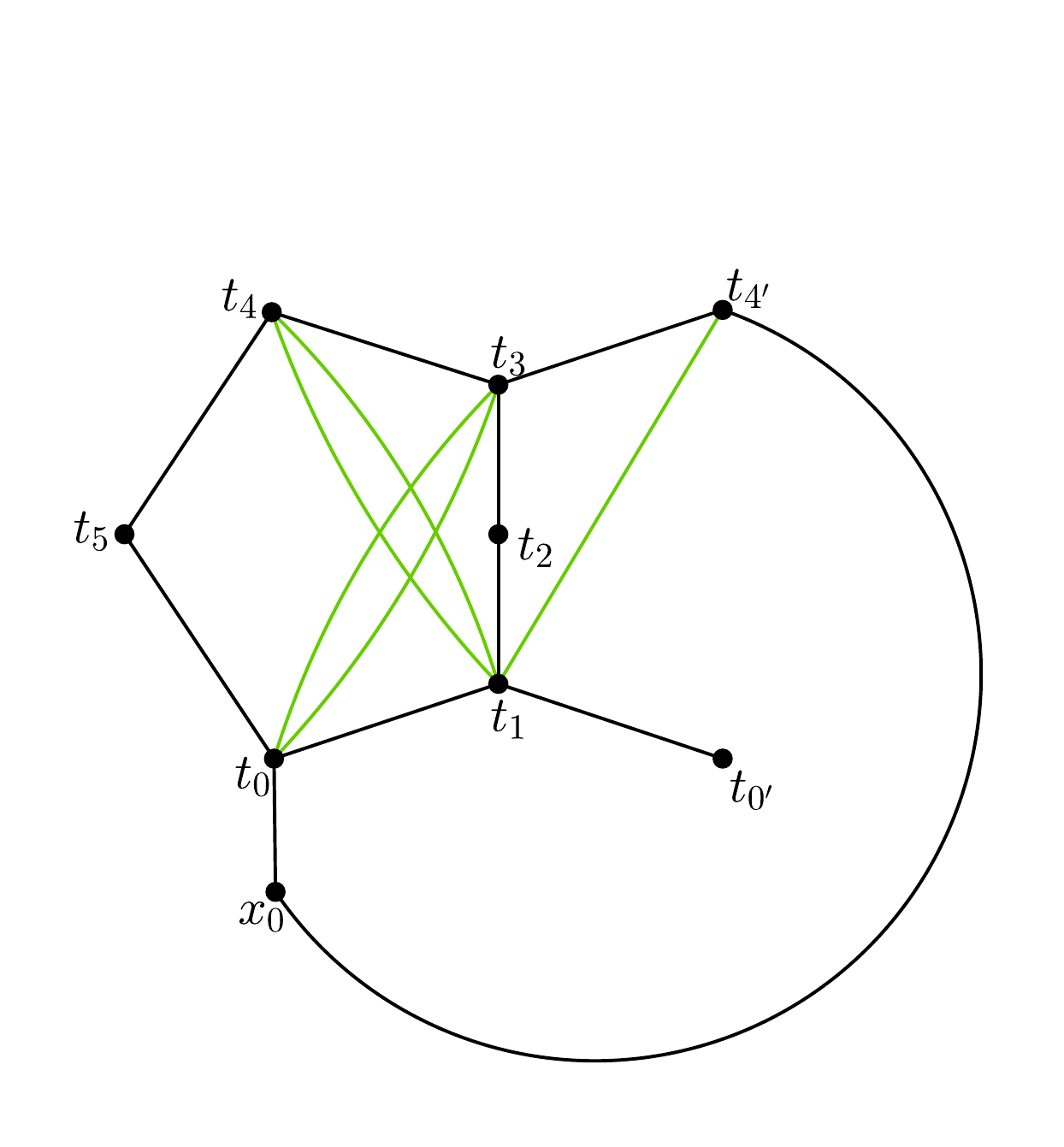}& 
\includegraphics[height=33mm]{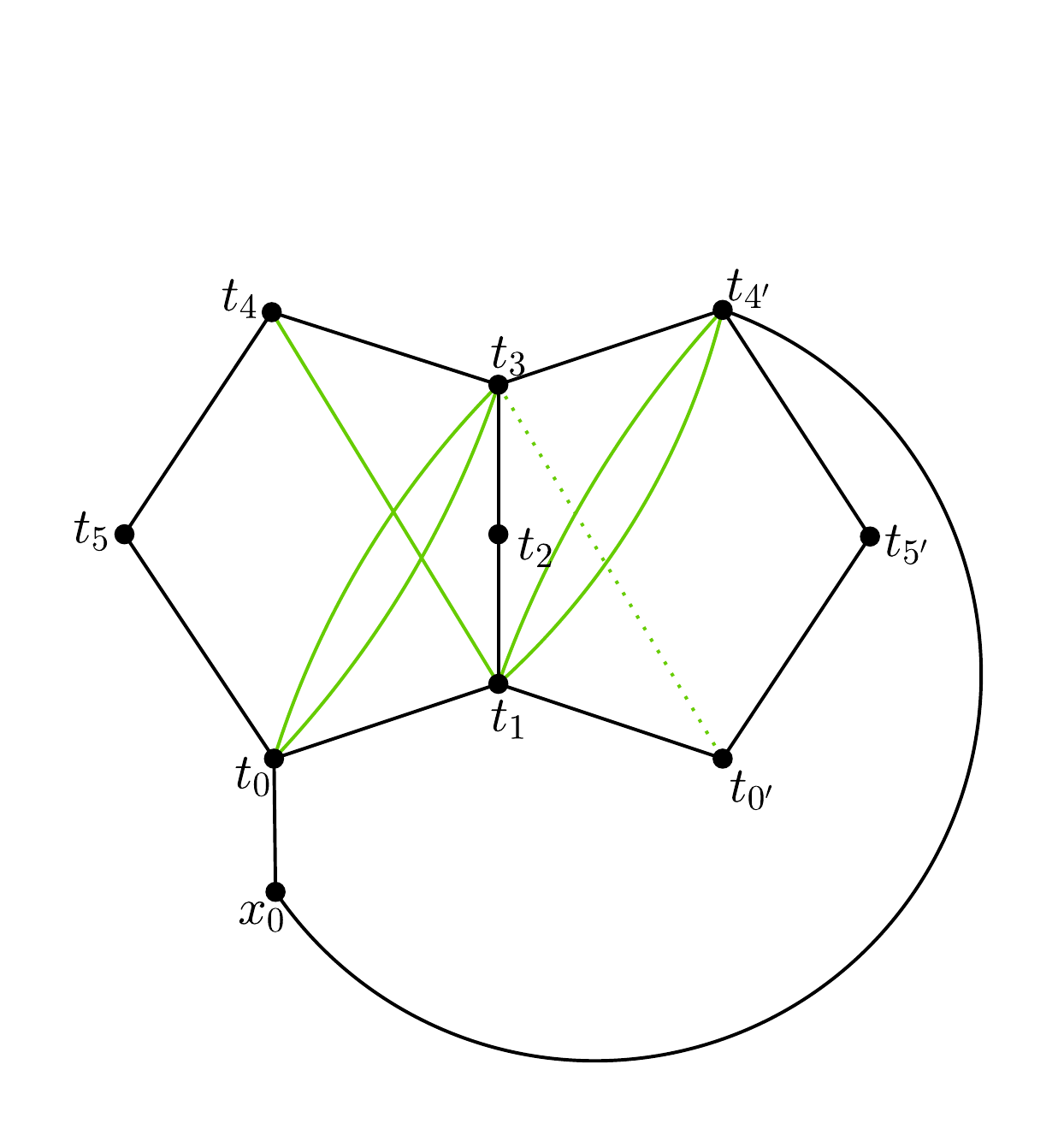}&
\includegraphics[height=33mm]{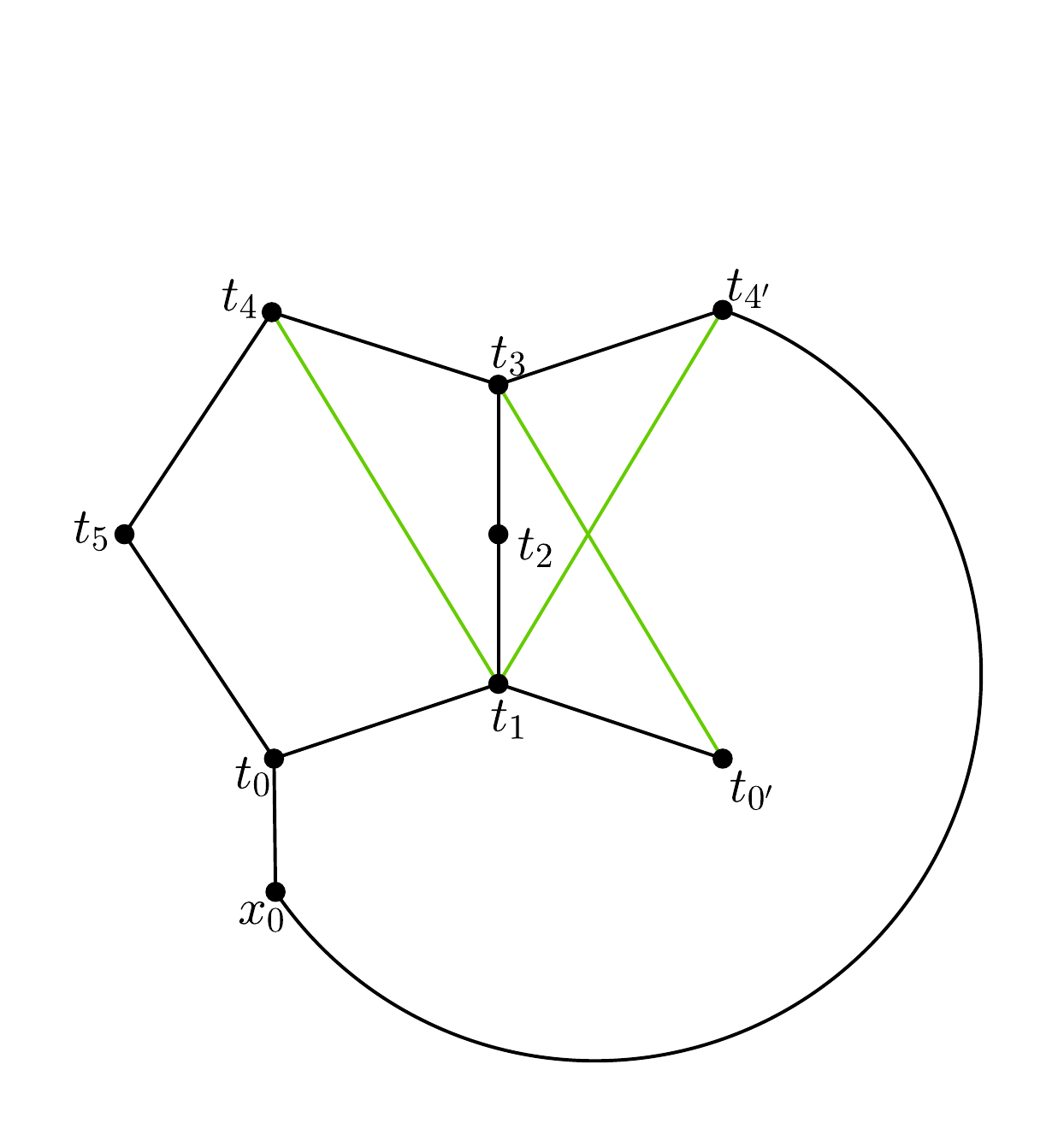} \\
\hline
Claim B2.1 &
Claim B2.1 &
Claim B2.2. \\
\includegraphics[height=33mm]{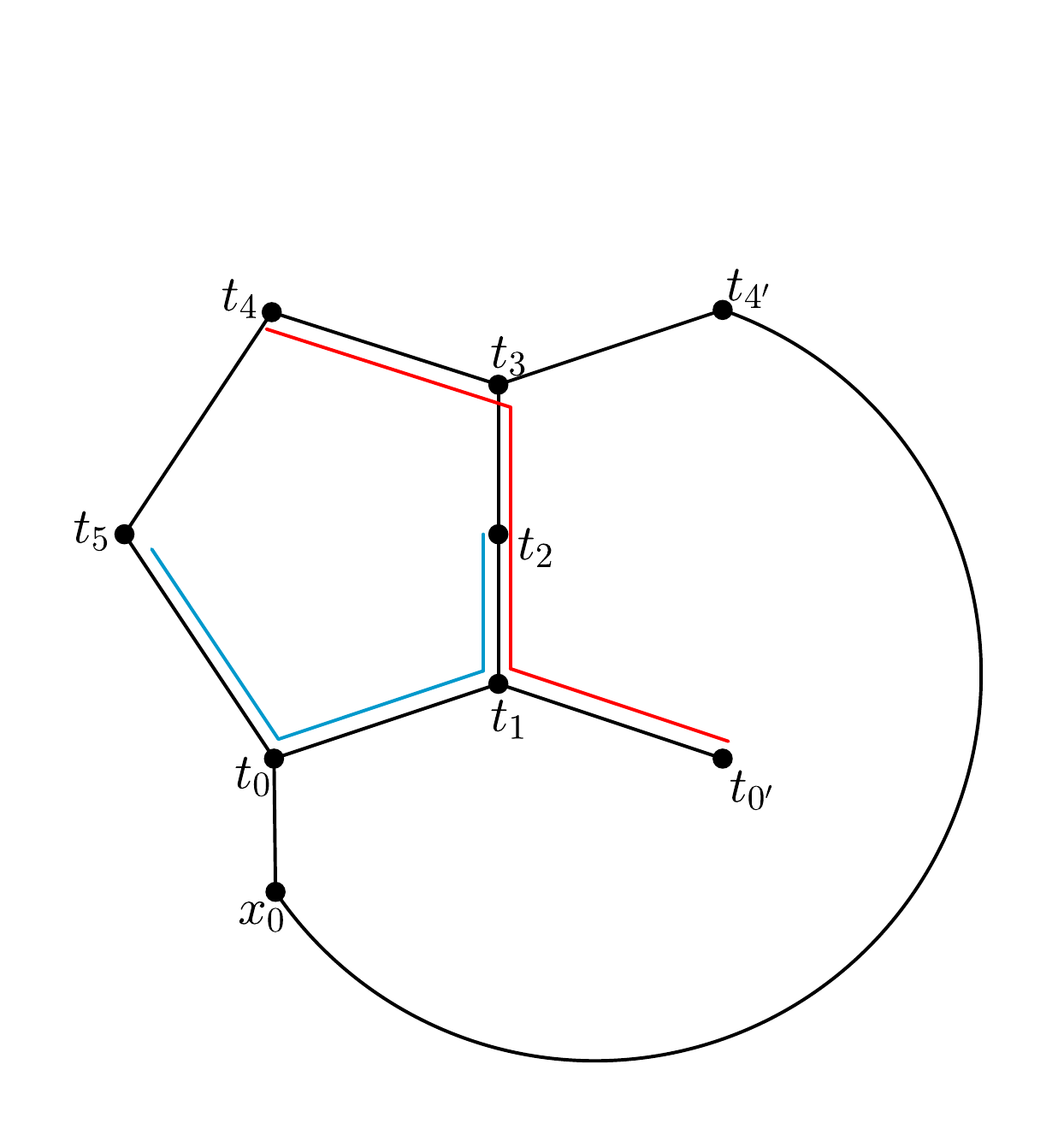} &
\includegraphics[height=33mm]{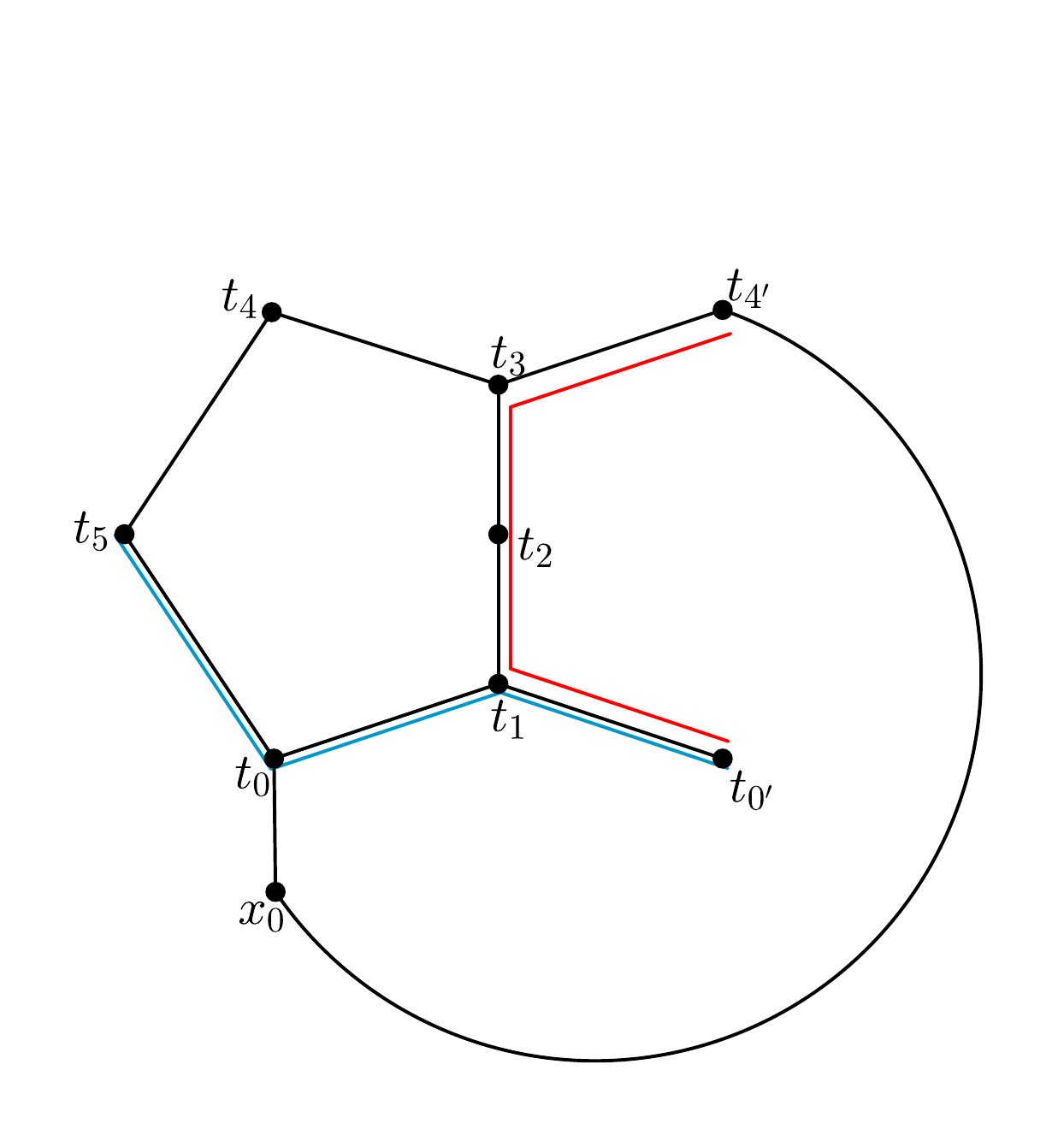} &
\includegraphics[height=33mm]{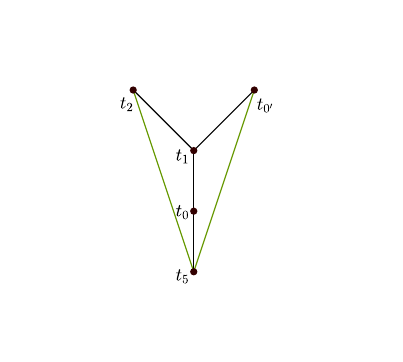} \\
\hline
Case B2.A. &
Case B2.B. &
Case B2.C.\\
\includegraphics[height=33mm]{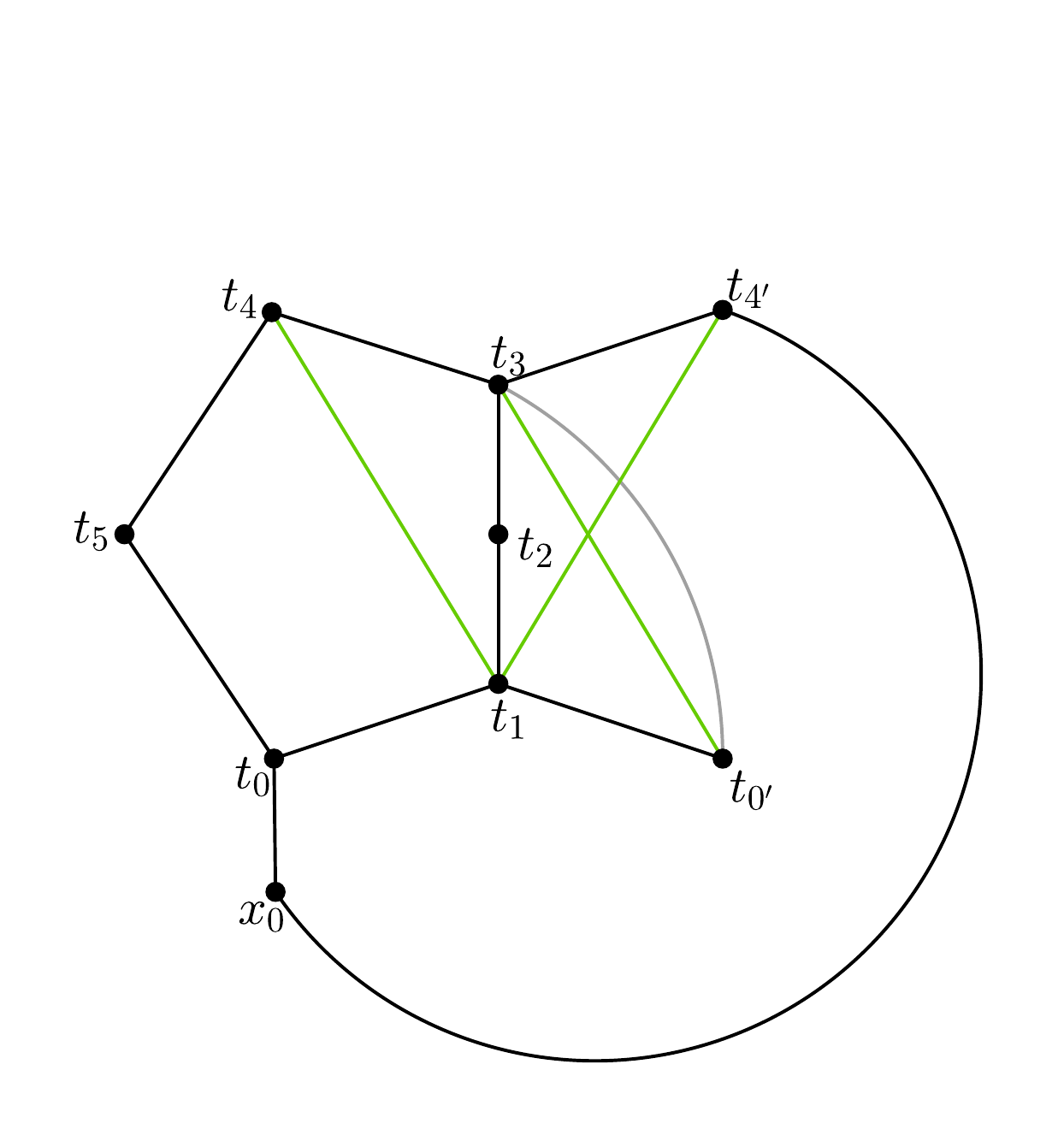} &
\includegraphics[height=33mm]{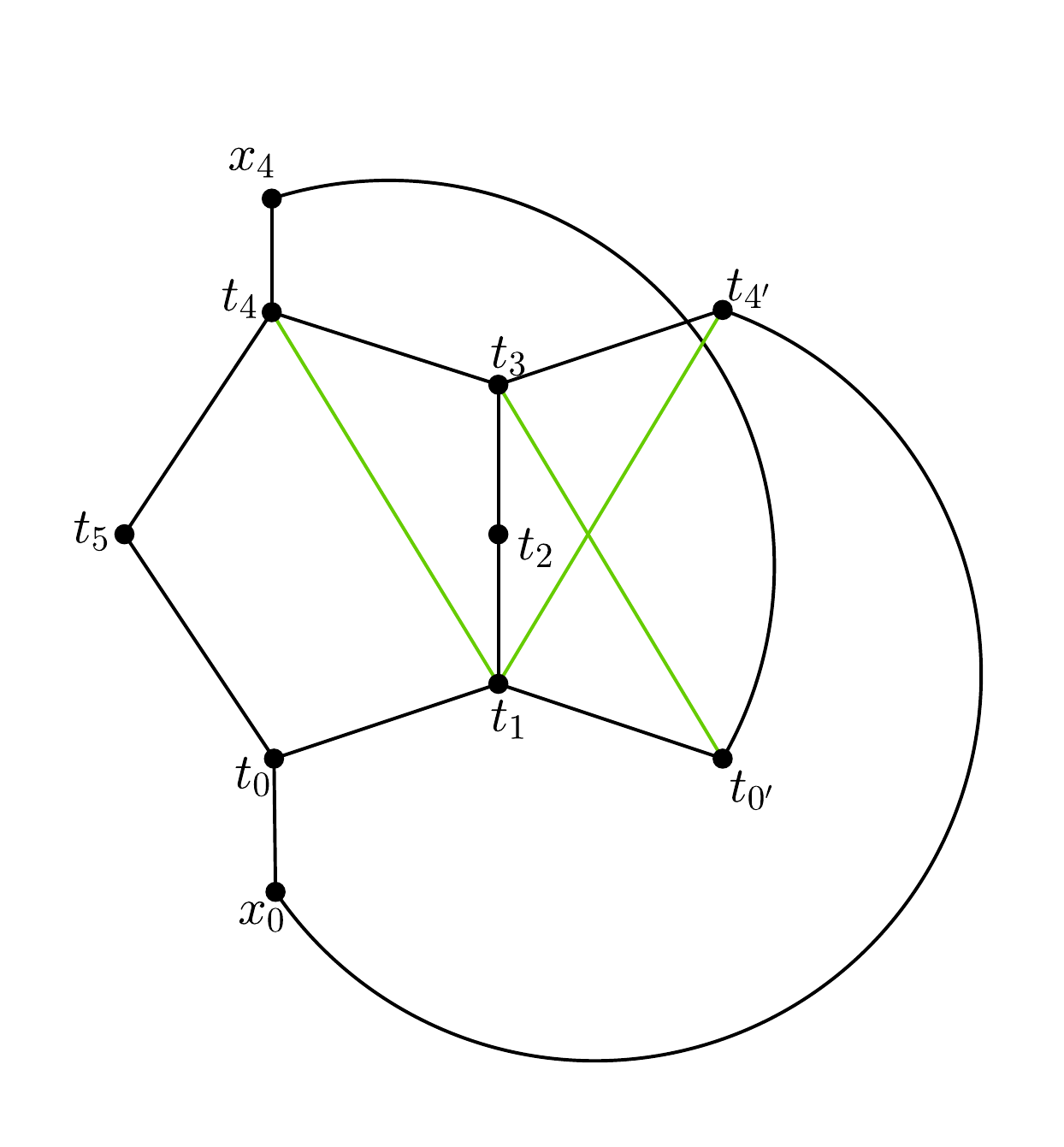} & 
\includegraphics[height=33mm]{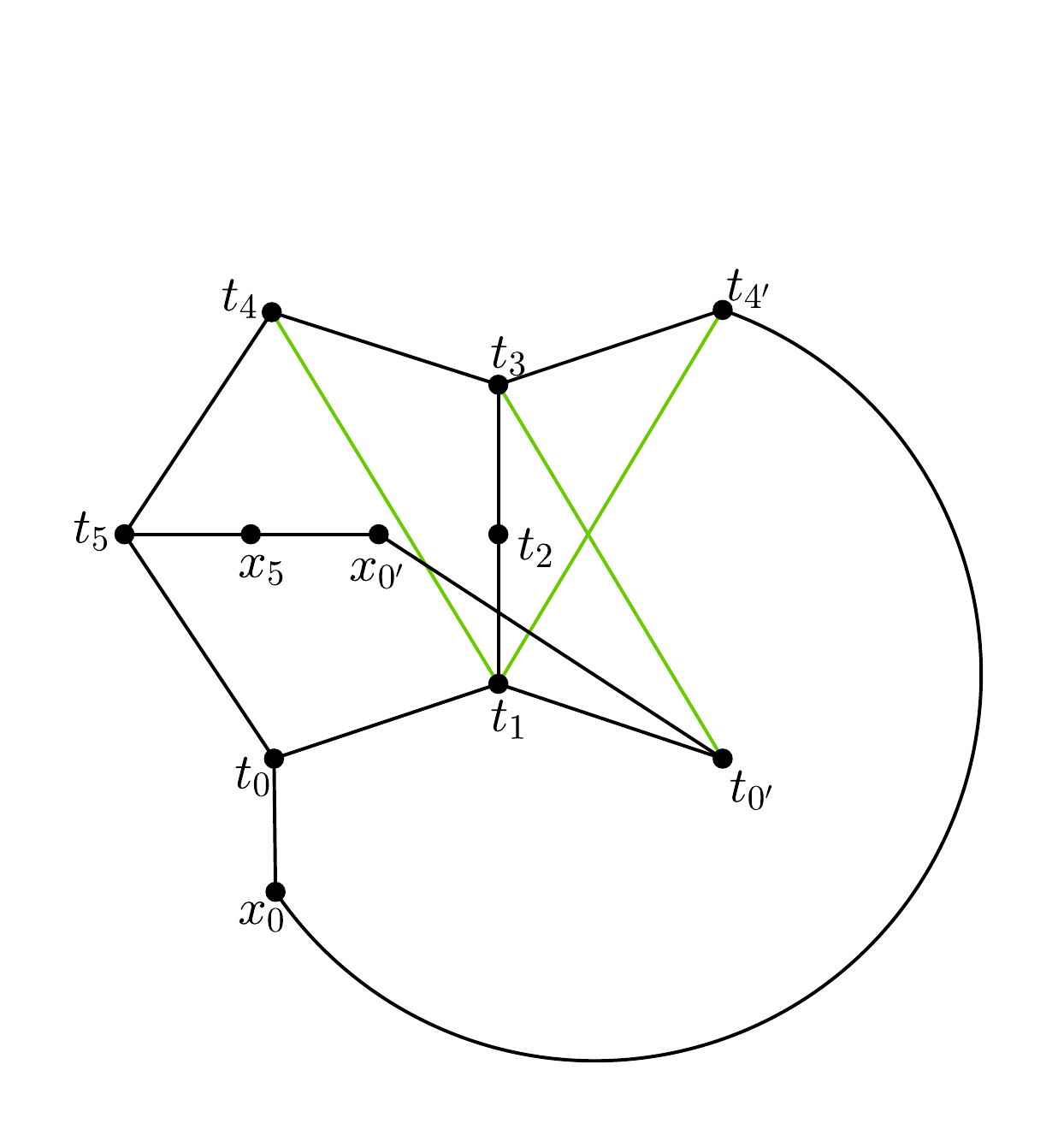} \\
\hline
\end{tabular}
\end{table} 

%%%%%%%%%%%%%%%%%%%%%%%%%%%%%%%%%%
%%%%%%%%%%%%%%%%%%c%%%%%%%%%%%%%%%
\newpage
\section*{Theorem 15. CASE C}
\begin{table}[h!]
\centering
\begin{tabular}{| c | c | c |}
\hline
Claim 1. & 
Claim 2. &
Claim 3. \\
\includegraphics[height=43mm]{claim1.pdf} &
\includegraphics[height=43mm]{claim2.pdf} &
\includegraphics[height=43mm]{claim3.pdf} \\
\hline
CASE C. &
Claim C1. &
Claim C2. \\
\includegraphics[height=43mm]{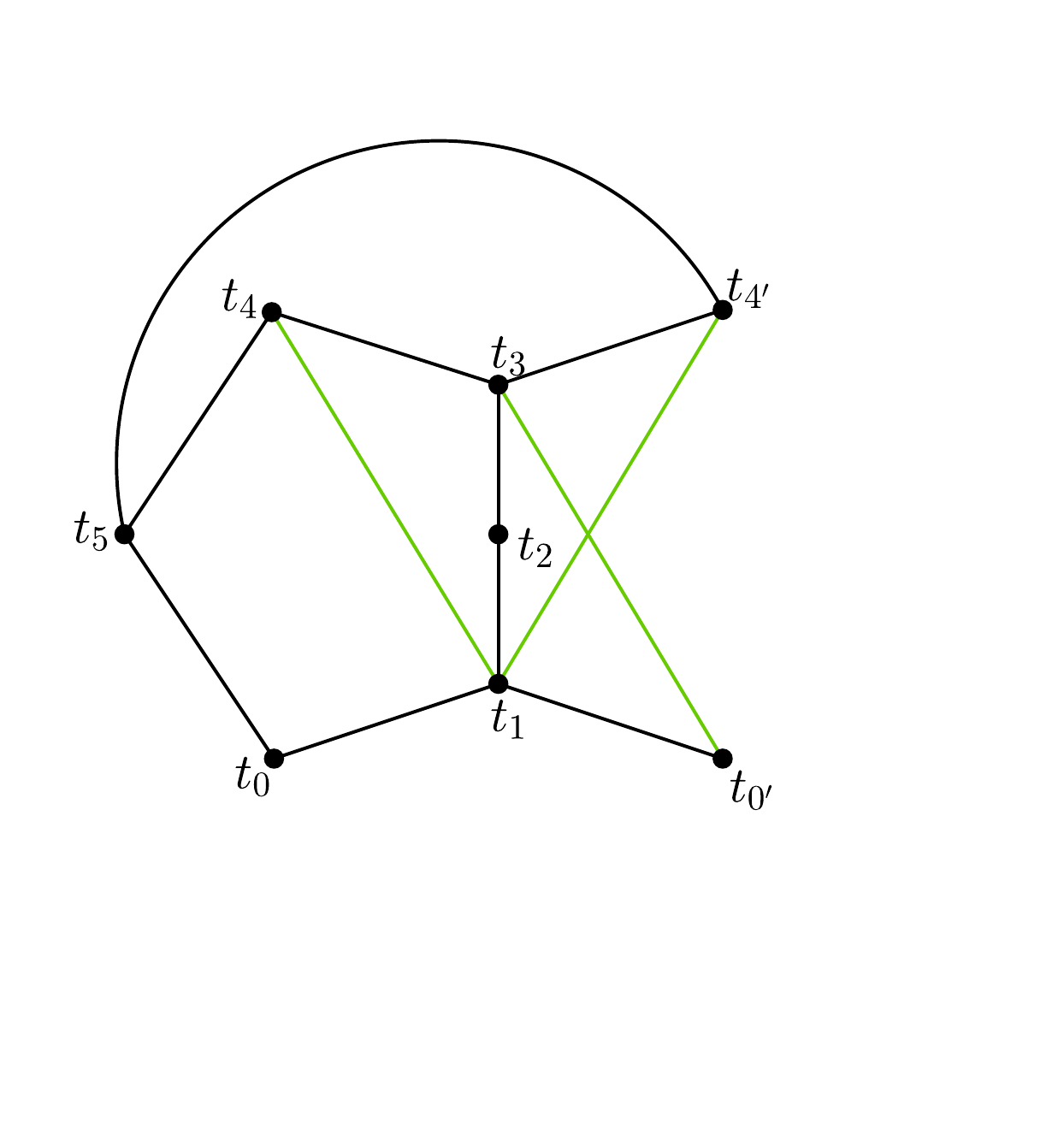} &
\includegraphics[height=43mm]{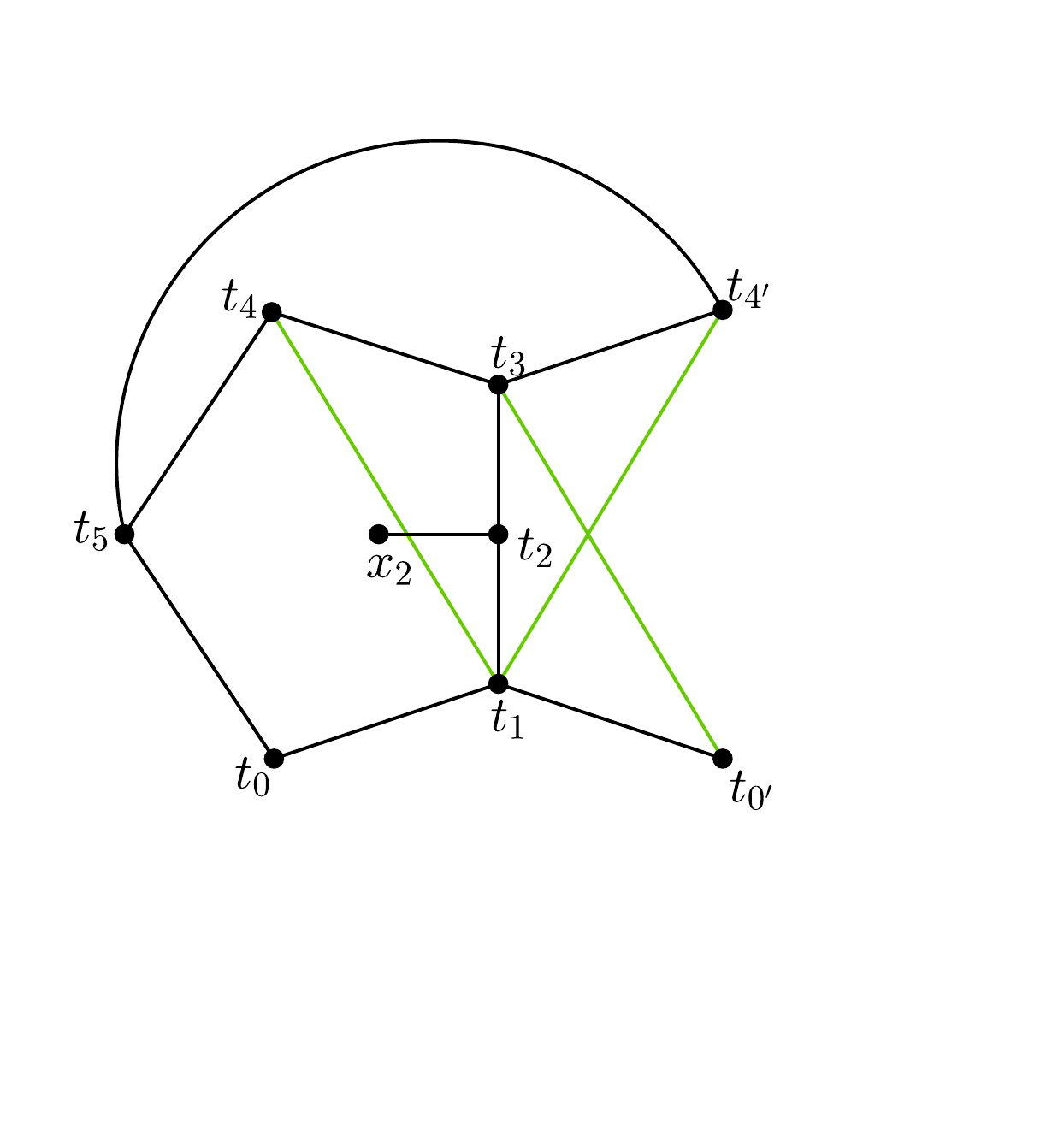} & 
\includegraphics[height=43mm]{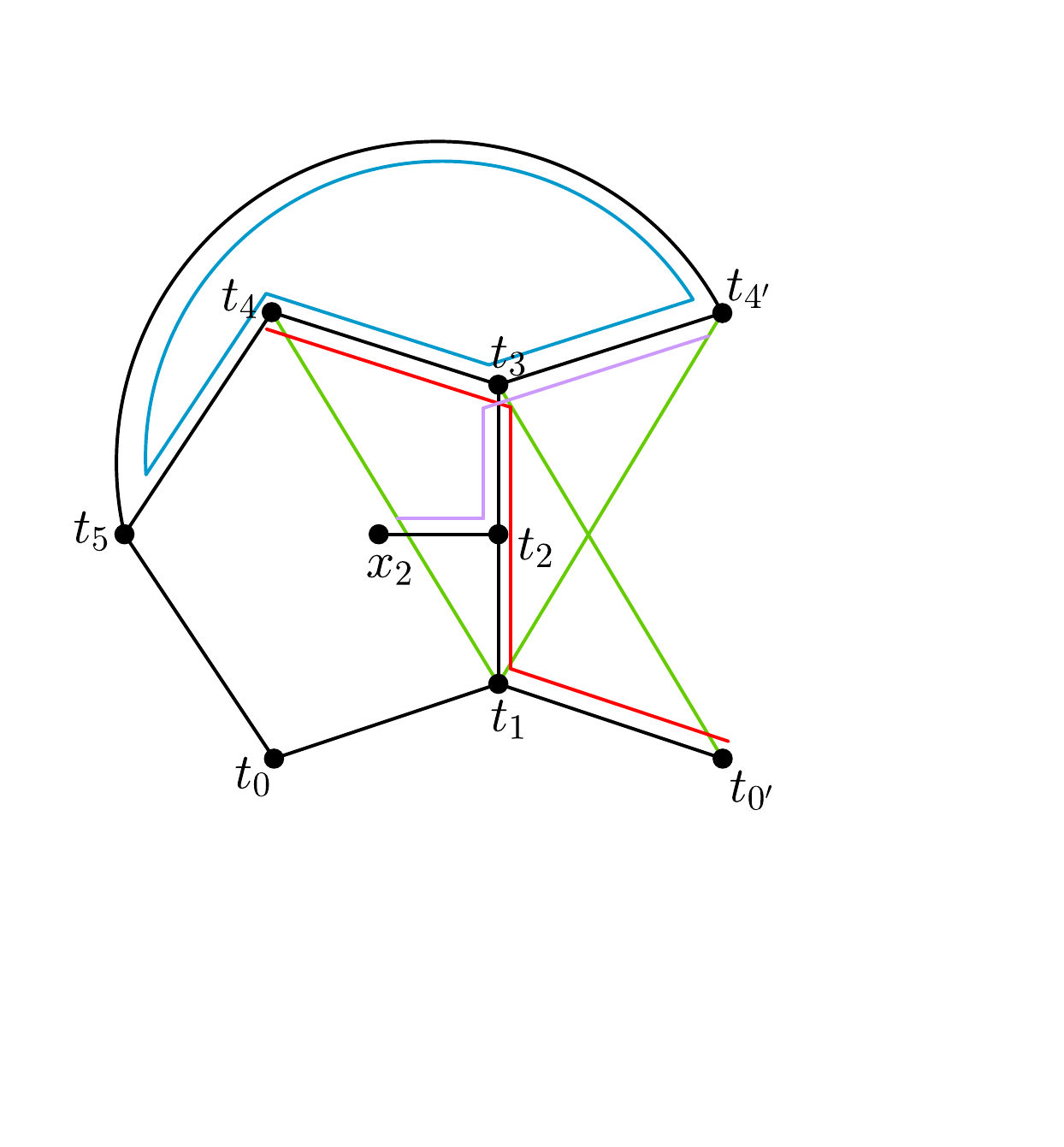} \\
\hline
Claim C3. &
Claim C4.&
Claim C5. \\
\includegraphics[height=43mm]{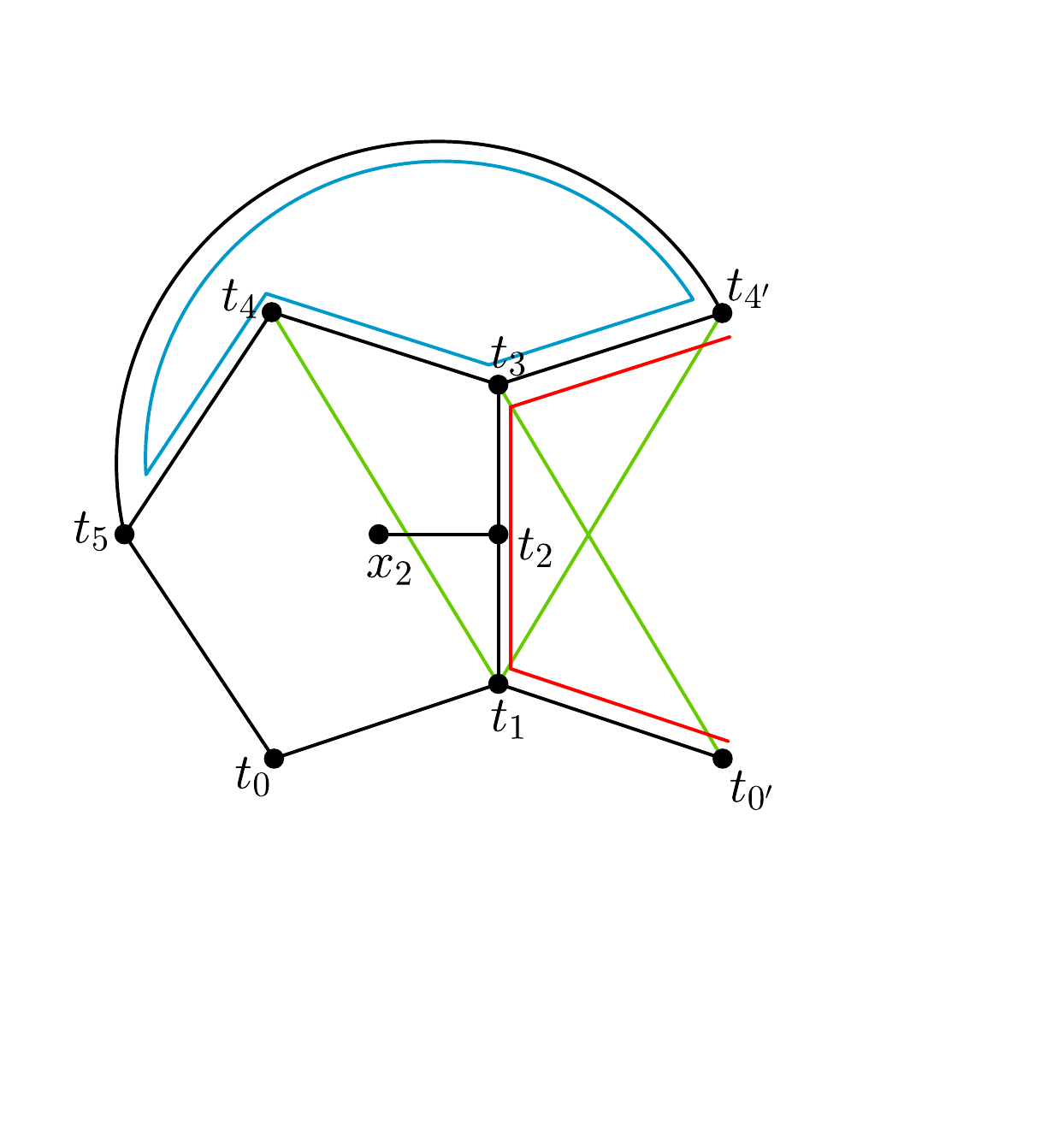} &
\includegraphics[height=43mm]{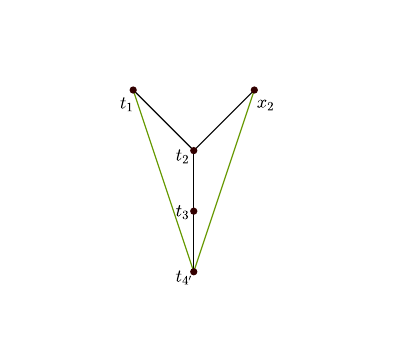} &
\includegraphics[height=43mm]{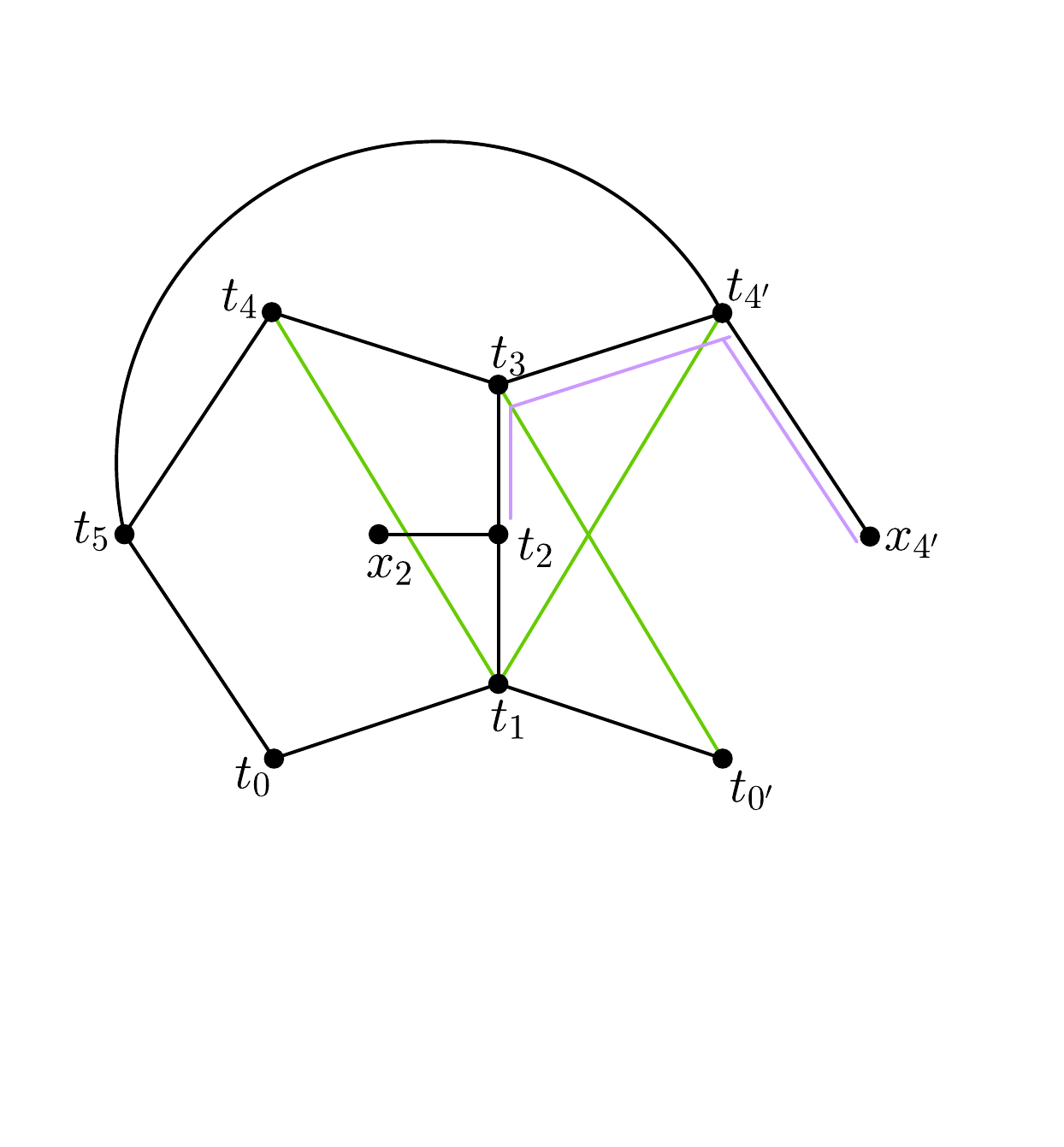} \\
\hline
Claim C6. &
Case C1. &
Claim C1.1.\\
\includegraphics[height=43mm]{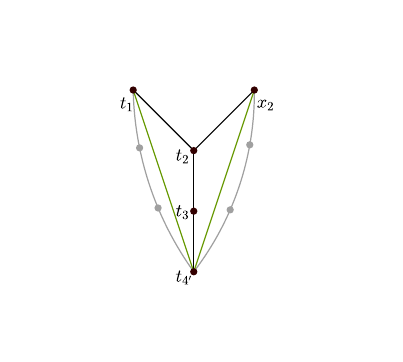} &
\includegraphics[height=43mm]{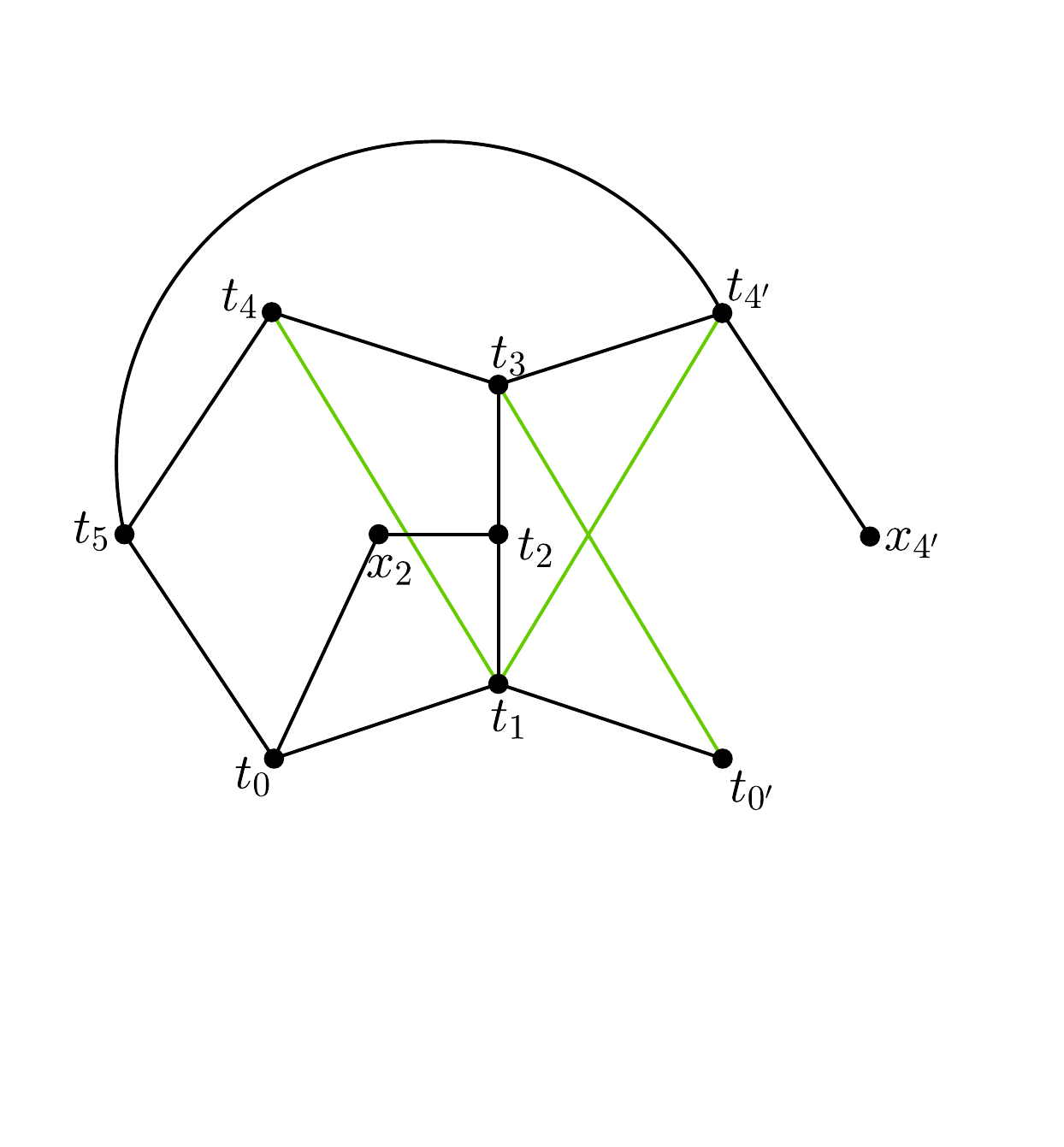} &
\includegraphics[height=43mm]{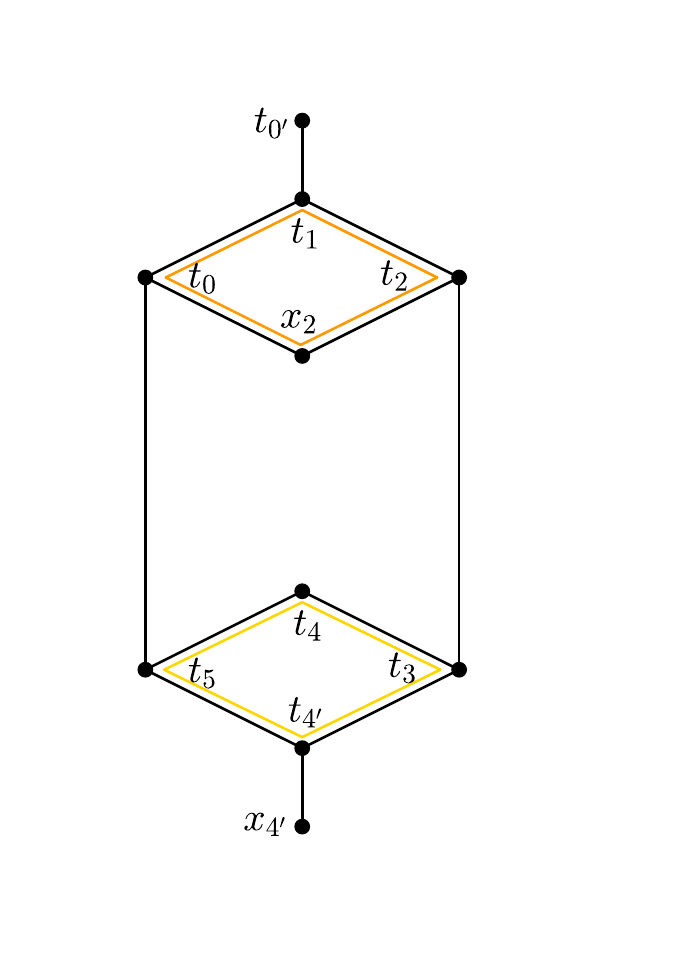}  \\
\hline
\end{tabular}
\end{table}
\clearpage
\begin{table}
\begin{tabular}{|m{4cm}|m{4cm}|m{4cm}|}
\hline
Claim C1.2 & 
Claim C1.3 &
Claim C1.4 \\
\includegraphics[height=42.8mm]{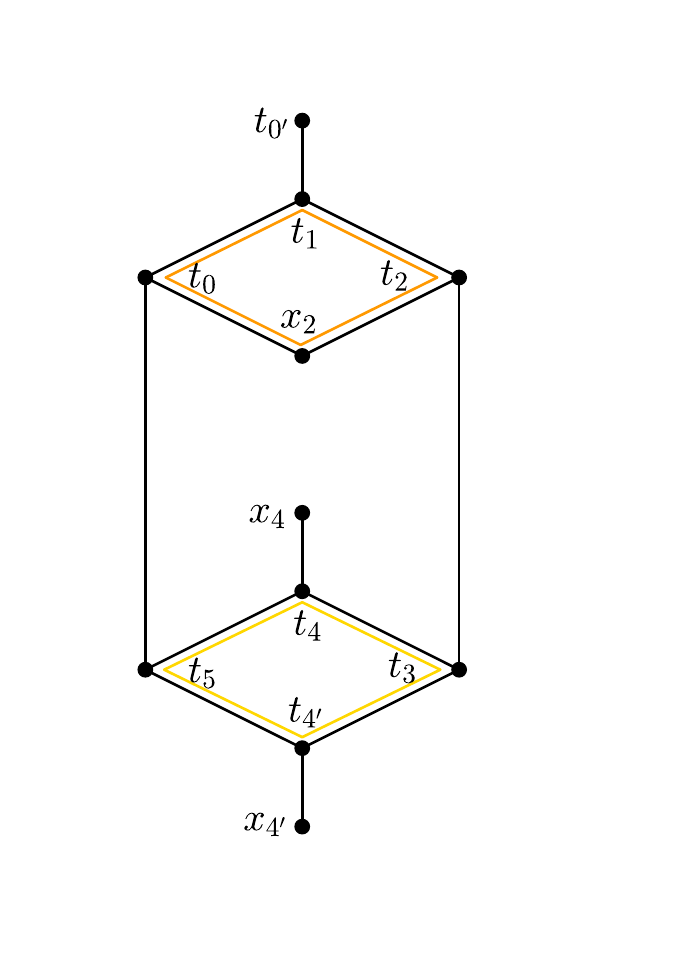} &
\includegraphics[height=42.8mm]{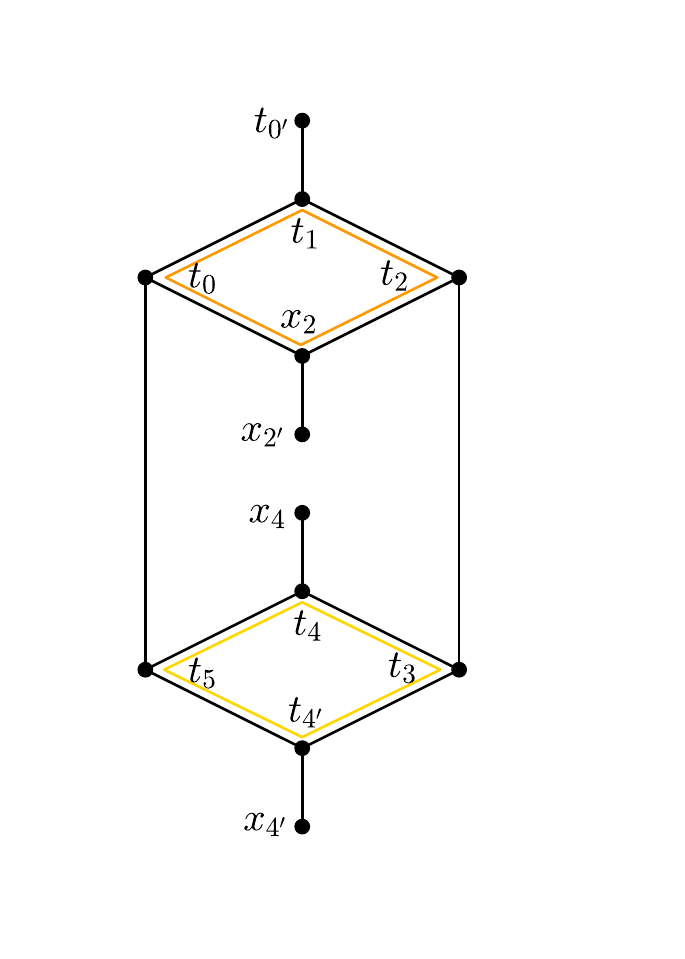} &
\includegraphics[height=42.8mm]{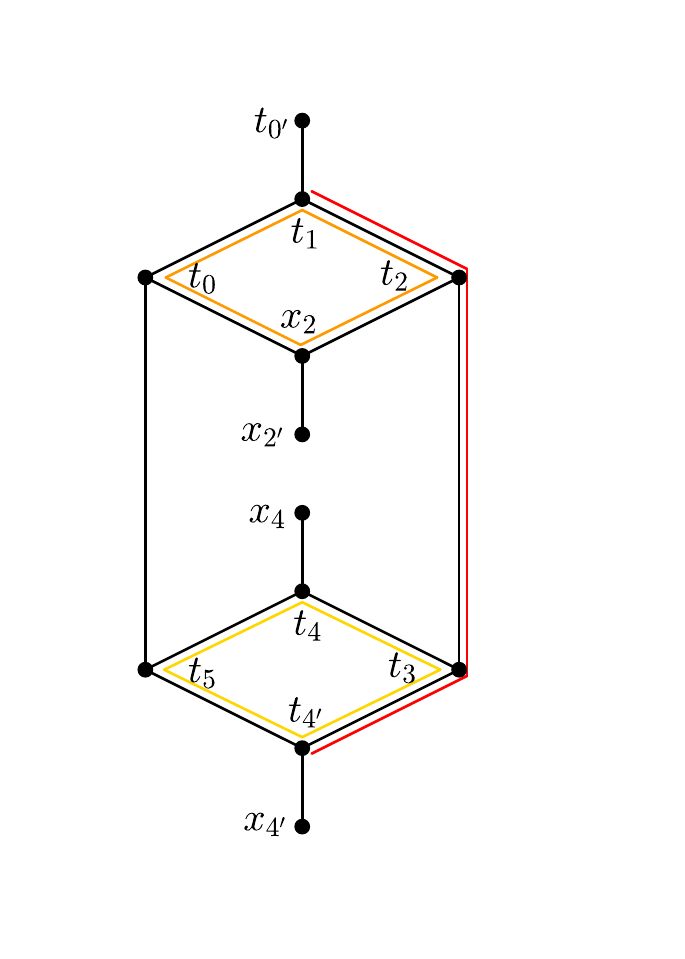} \\
\hline
Claim C1.4 &
Claim C1.5 &
Claim C1.5 \\
\includegraphics[height=42.8mm]{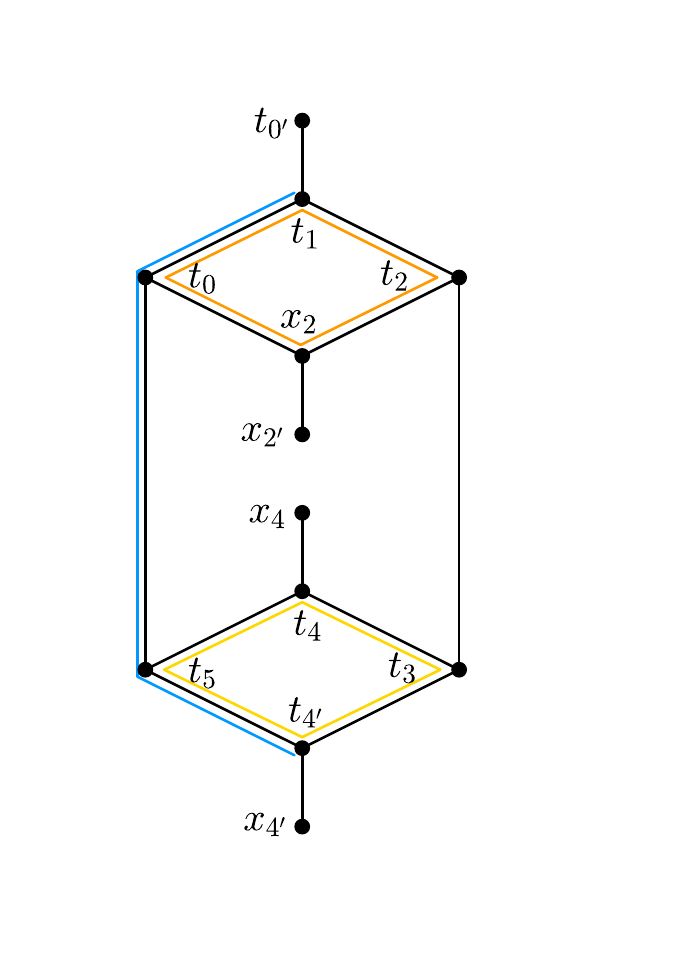} &
\includegraphics[height=42.8mm]{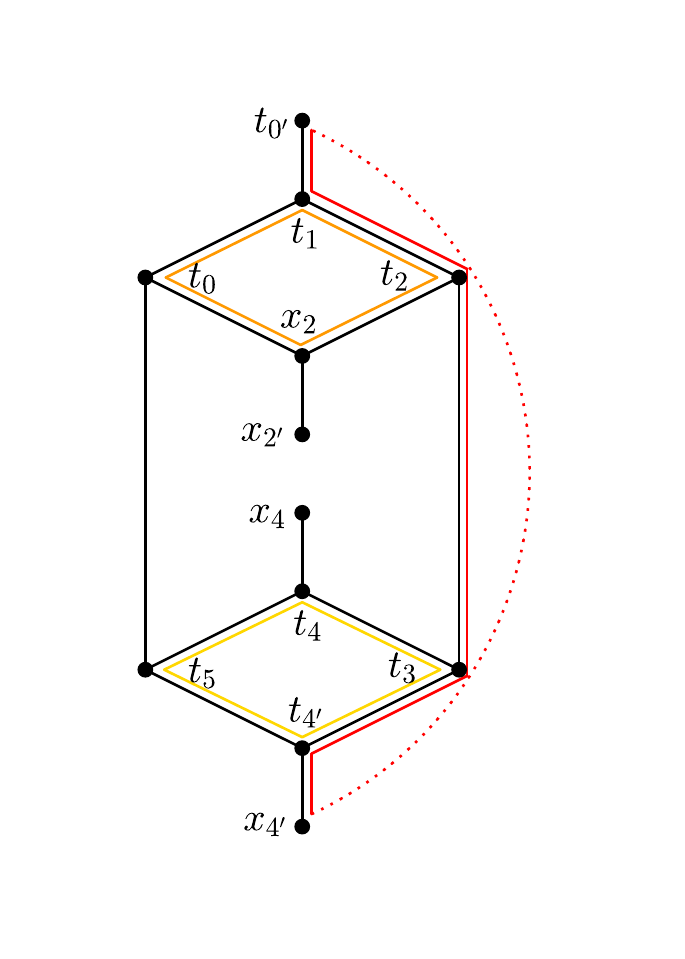} &
\includegraphics[height=42.8mm]{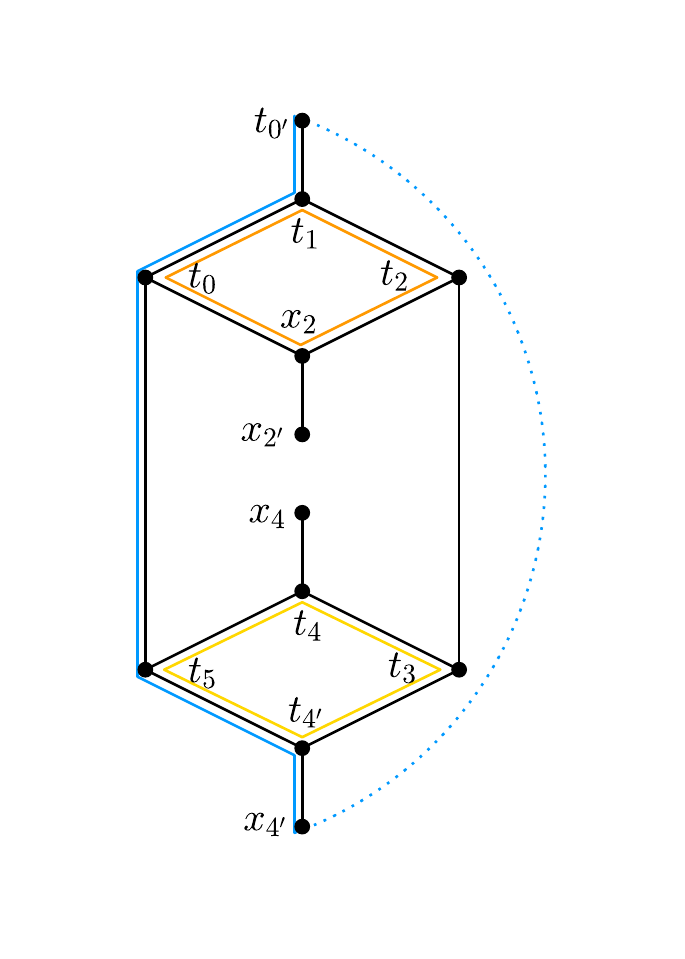} \\
\hline
Claim C1.6 &
Claim C1.6.1 &
Claim C1.6.2 \\
\includegraphics[height=42.8mm]{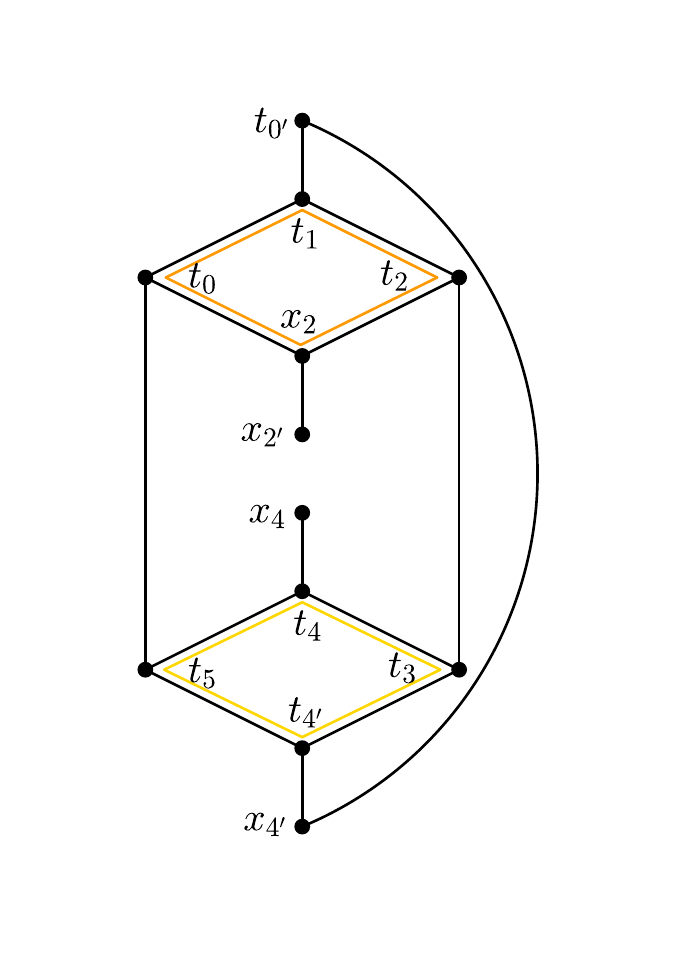} &
\includegraphics[height=42.8mm]{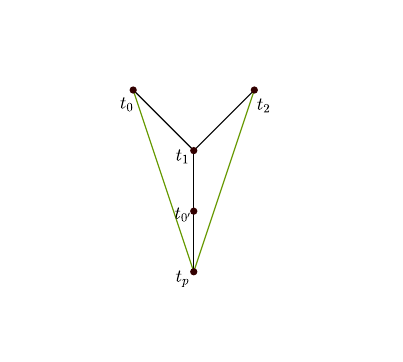} &
\includegraphics[height=42.8mm]{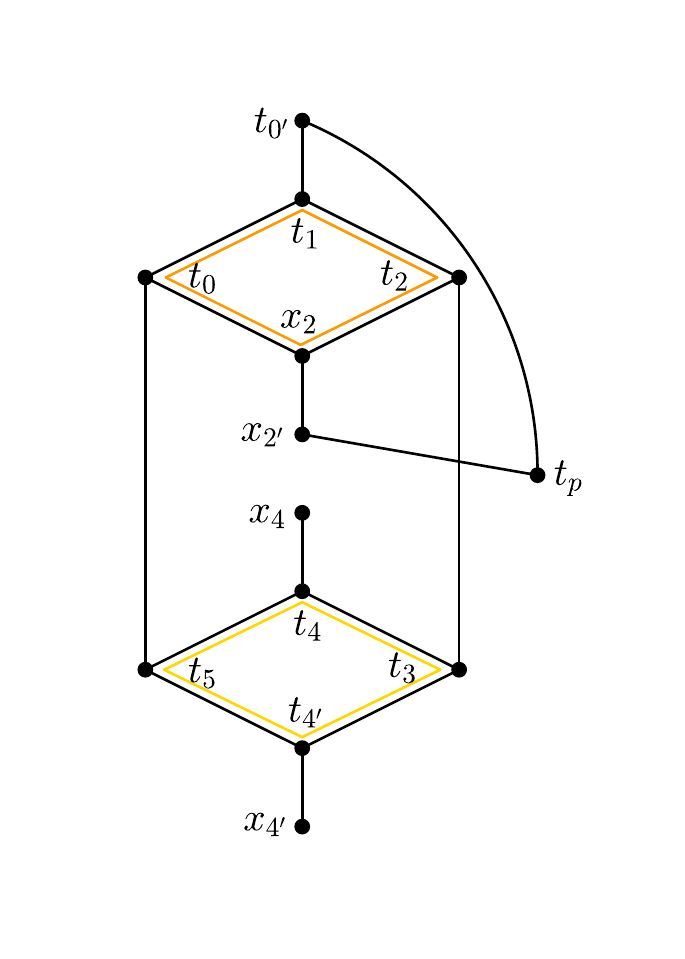} \\
\hline
Claim C1.6.3 &
Case C2 & 
Case C3. \\
\includegraphics[height=42.5mm]{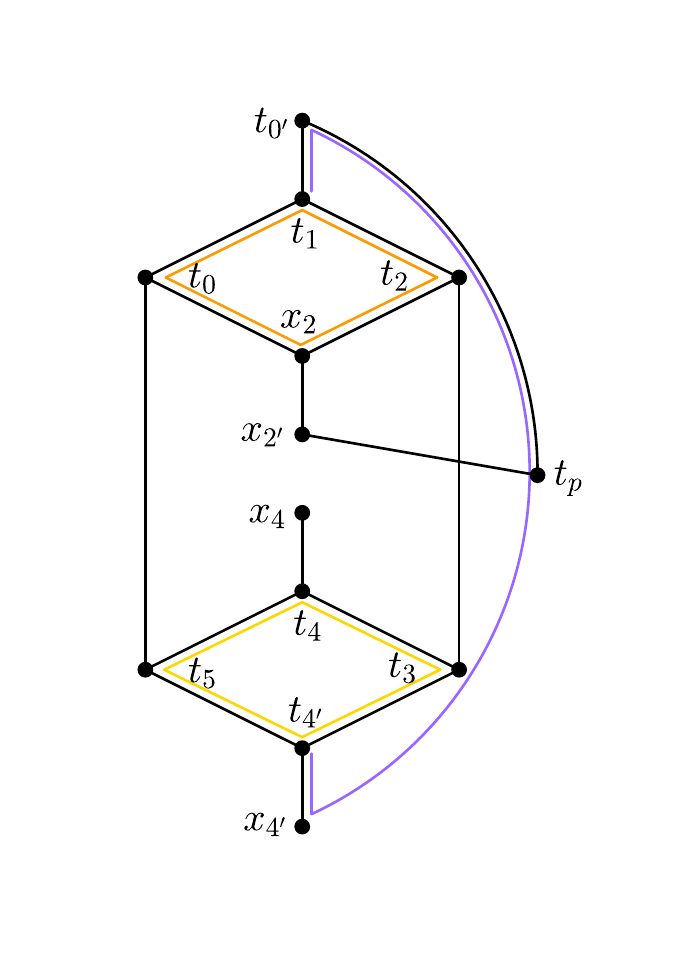} &
\includegraphics[height=42.5mm]{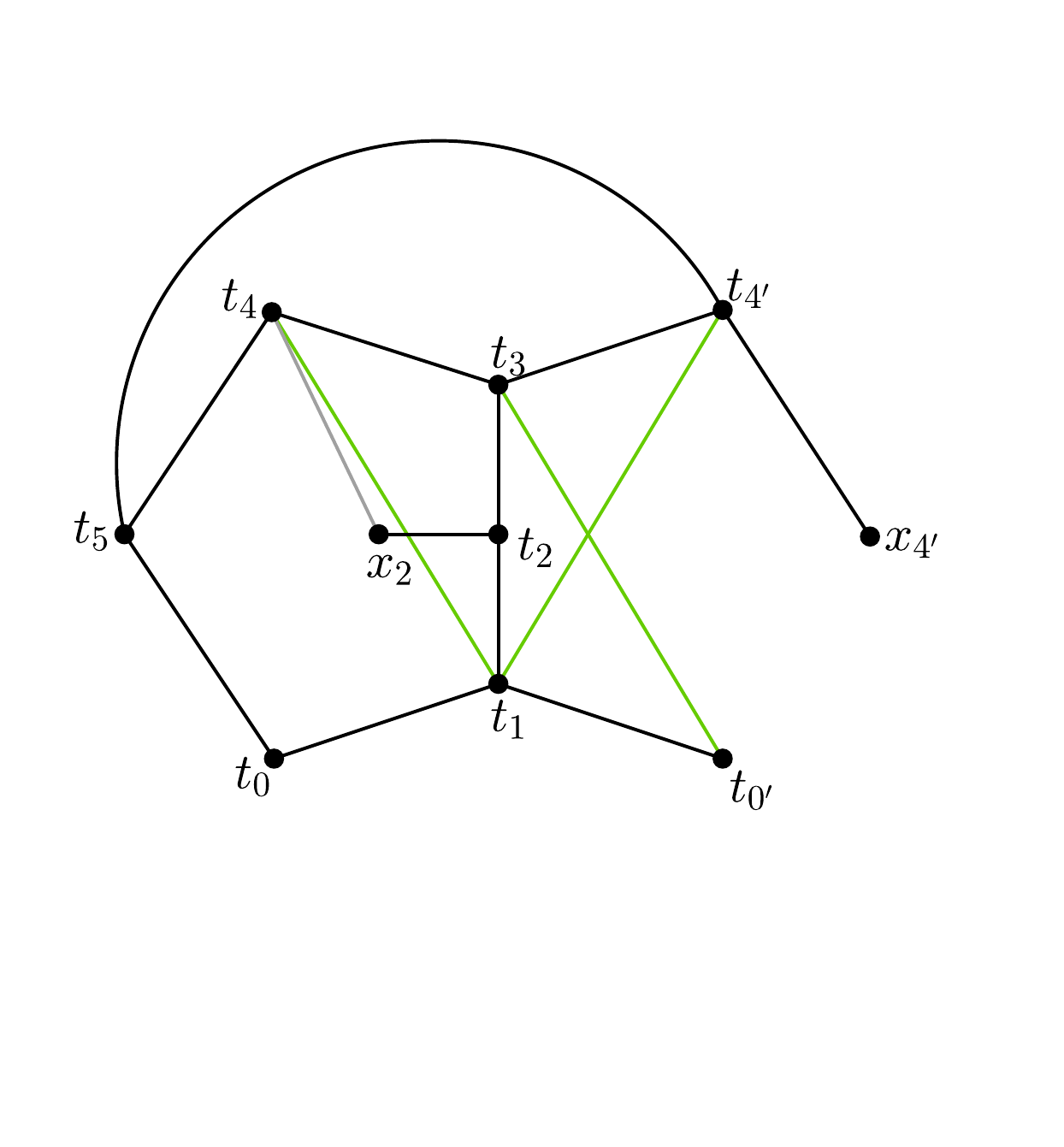} &
\includegraphics[height=42.5mm]{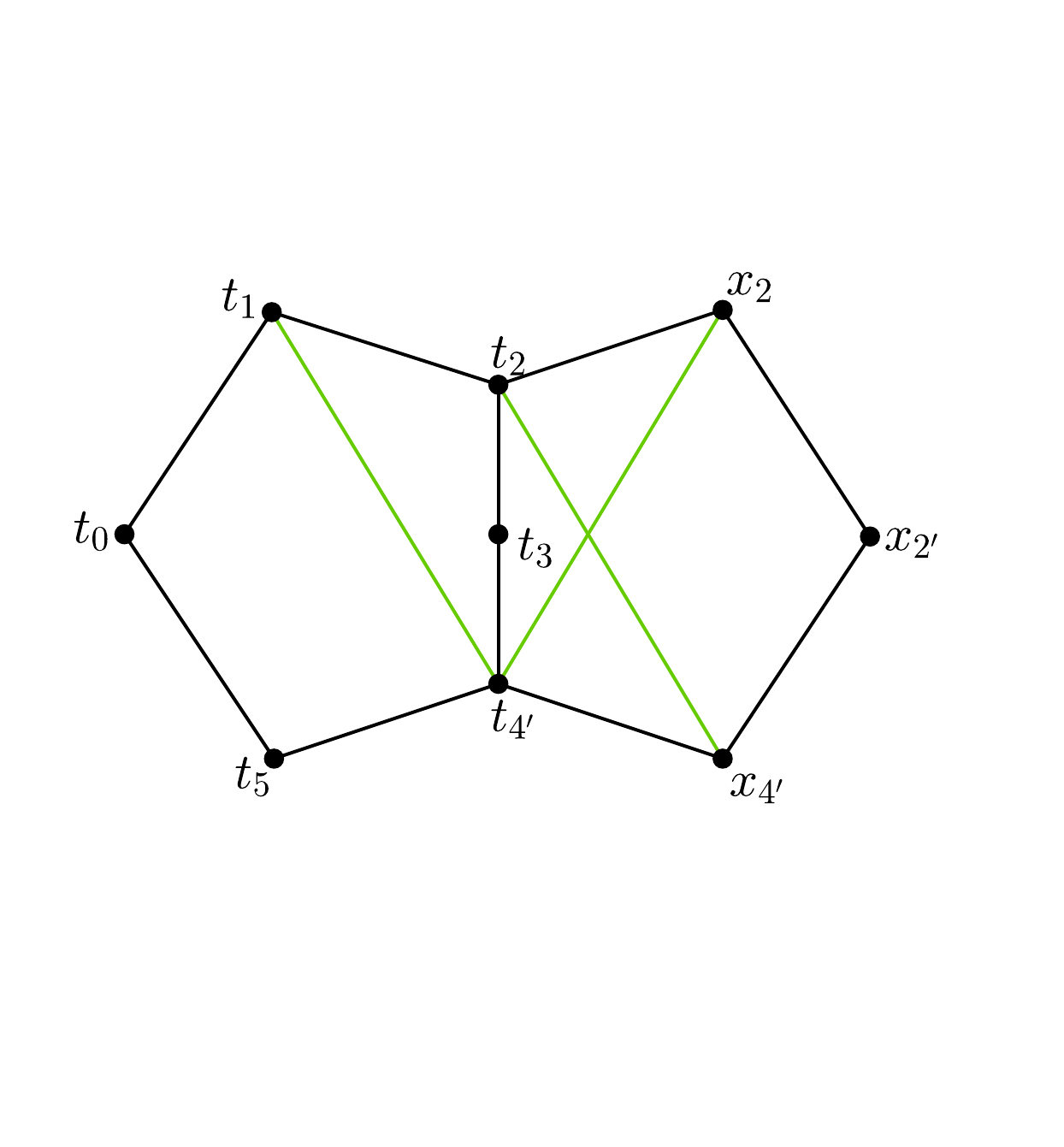}  \\
\hline
\end{tabular}
\end{table}

\clearpage

\section*{Theorem 16}
\vspace{1cm}
\begin{table}[h]
\centering
\begin{tabular}{| m{3.25cm}  m{3.25cm} | m{3.25cm}  m{3.25cm} |}
\hline
Claim 1. 
\hspace{-5mm}\includegraphics[width=43mm]{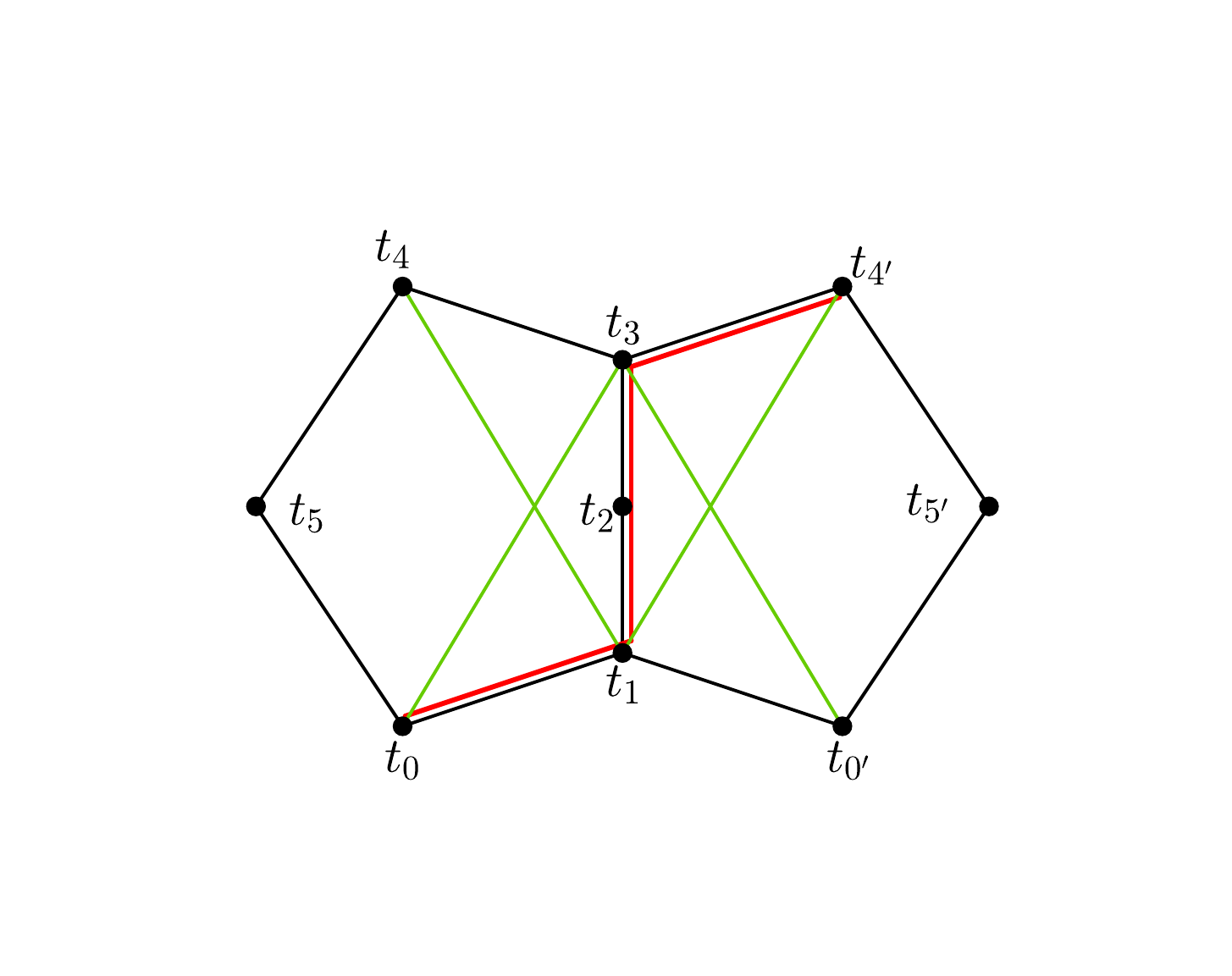} & \hspace{-5mm} \includegraphics[width=40mm]{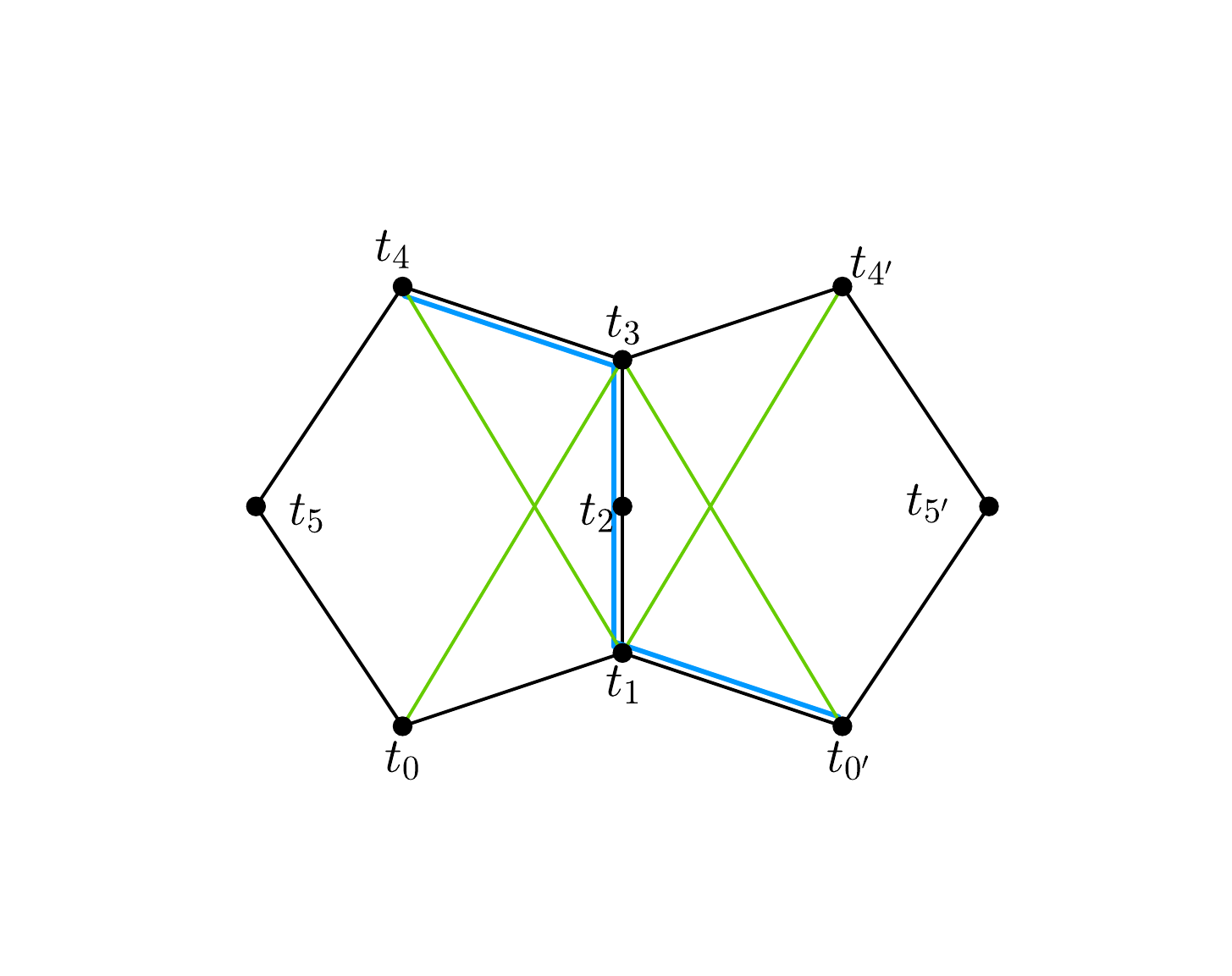} &
Claim 2. 
\hspace{-10mm}\includegraphics[width=40mm]{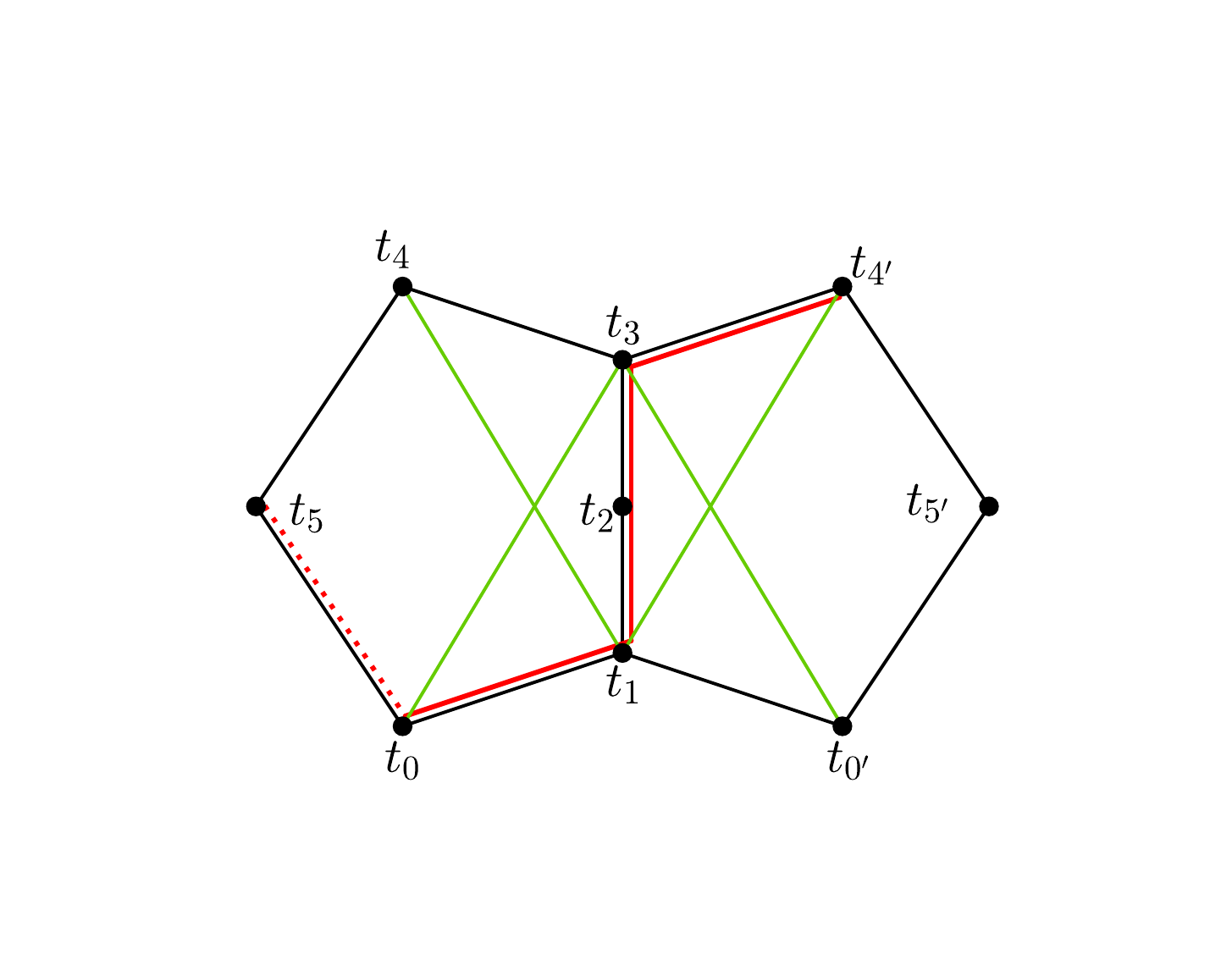} & \hspace{-5mm} \includegraphics[width=43mm]{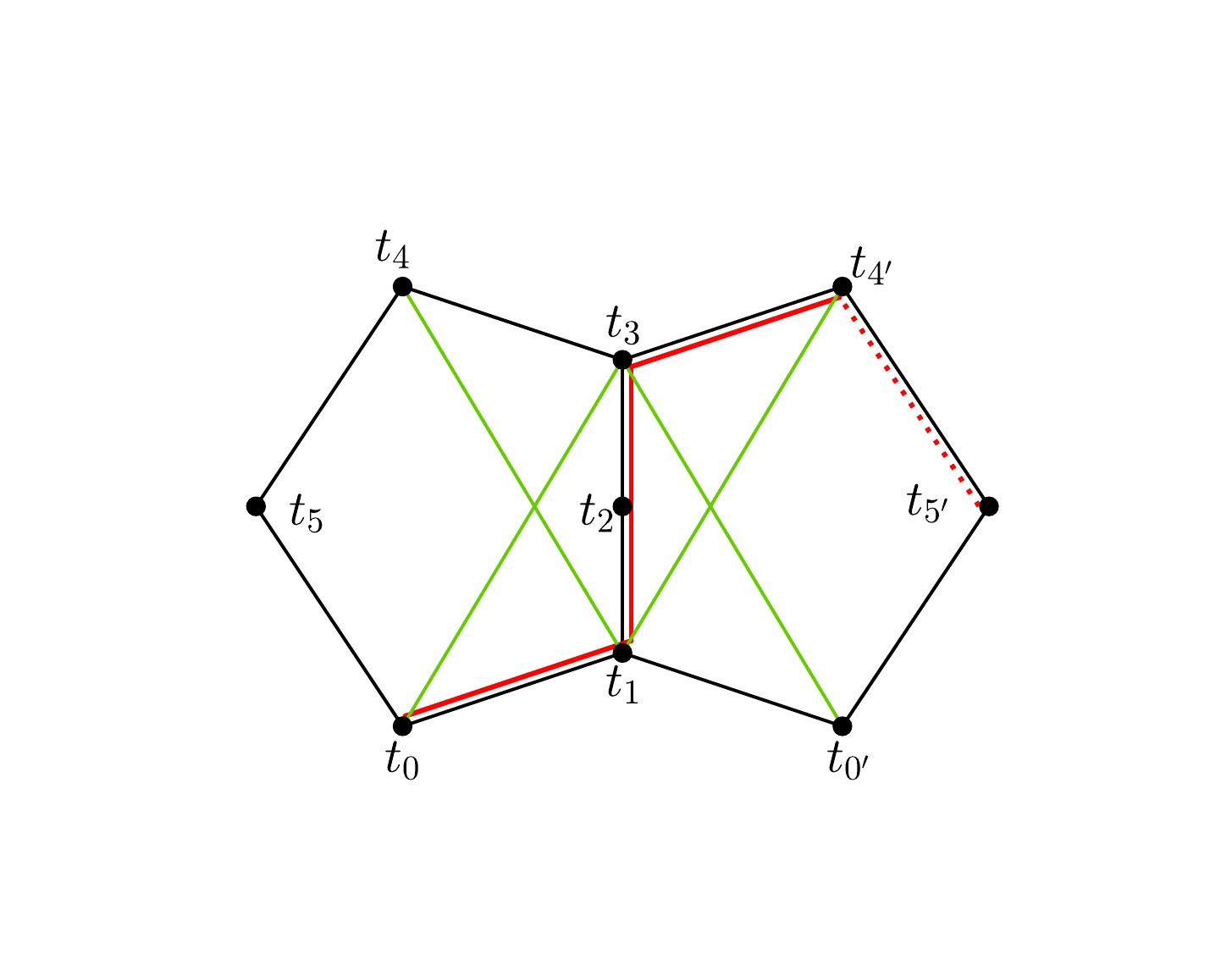} \\
\hline
Claim 3.
\vspace{1.52mm}\includegraphics[width=40mm]{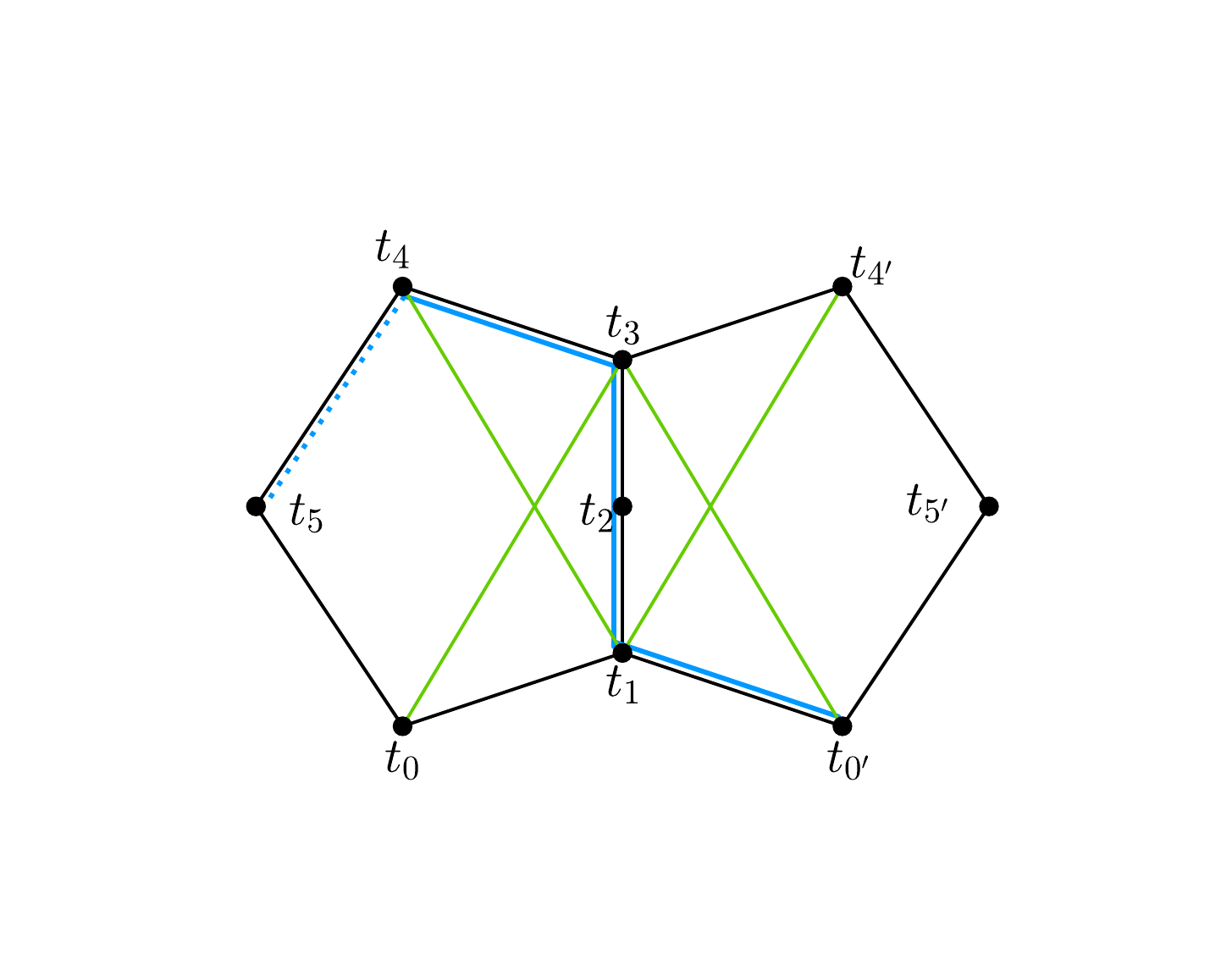} & \vspace{1.52mm} \includegraphics[width=40mm]{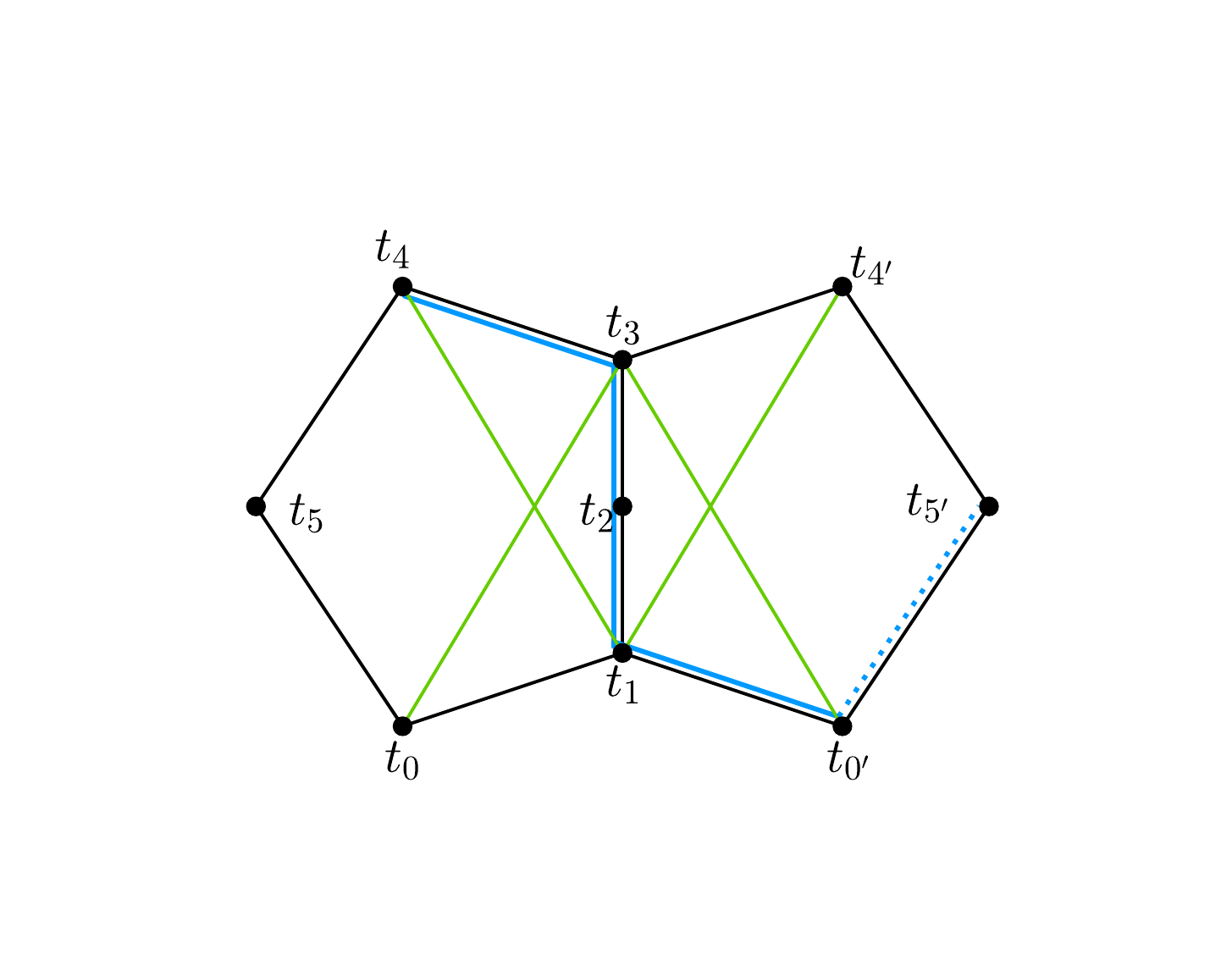} &
Claim 4. 
\vspace{1.52mm}\includegraphics[width=40mm]{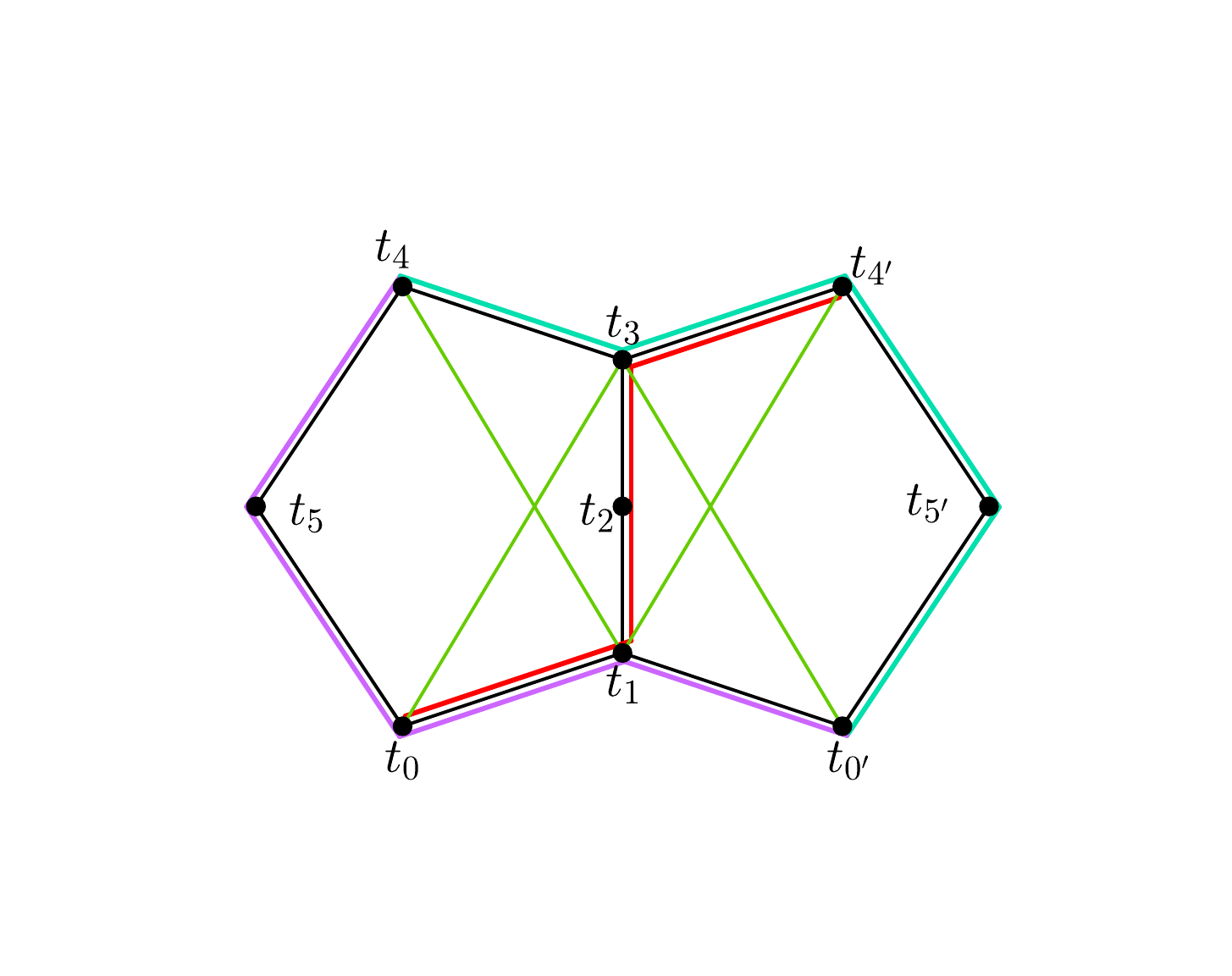} & \vspace{1.52mm} \includegraphics[width=40mm]{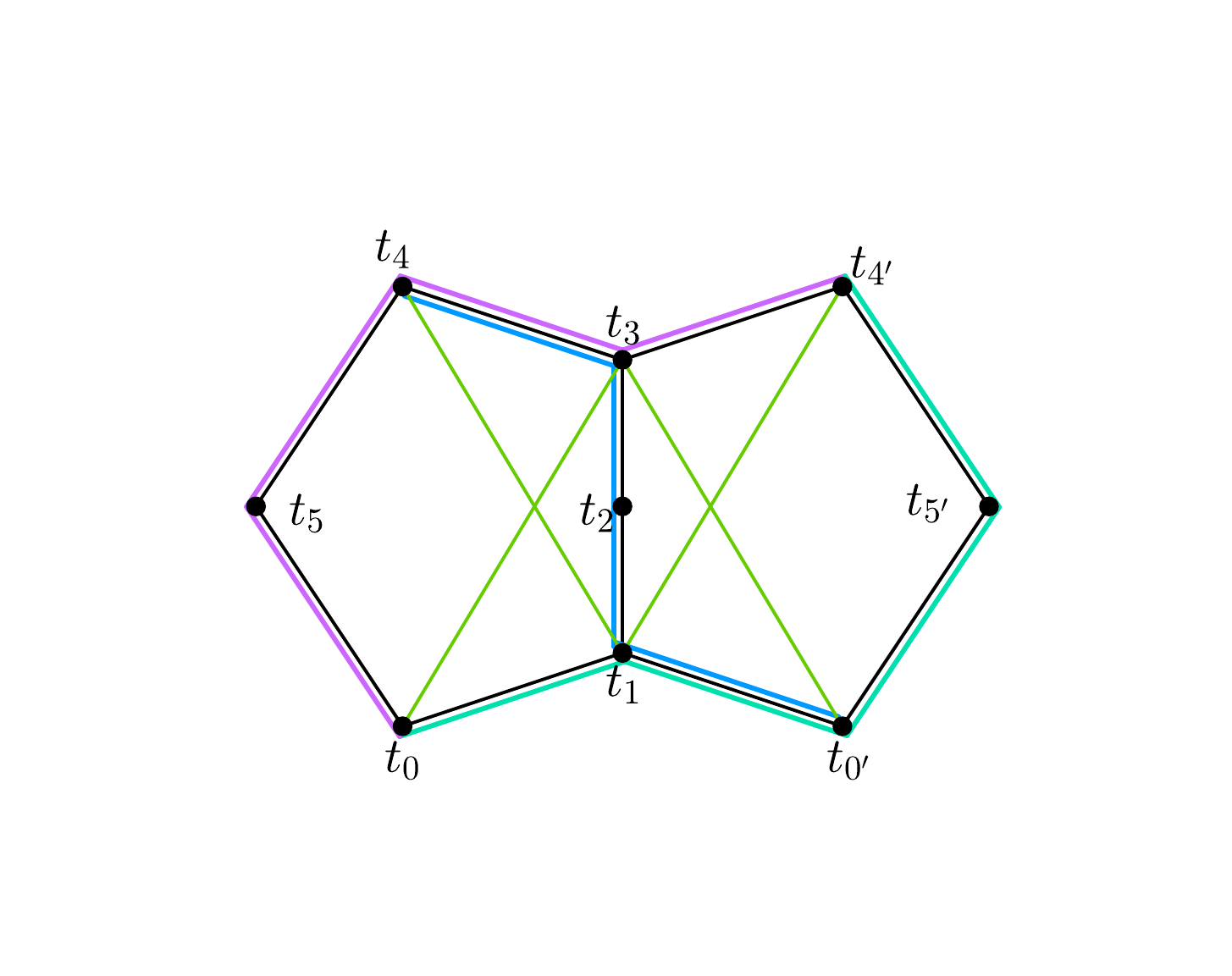} \\
\hline
\multicolumn{2}{| m{6.5cm} |}{Claim 5.
\includegraphics[scale=0.38]{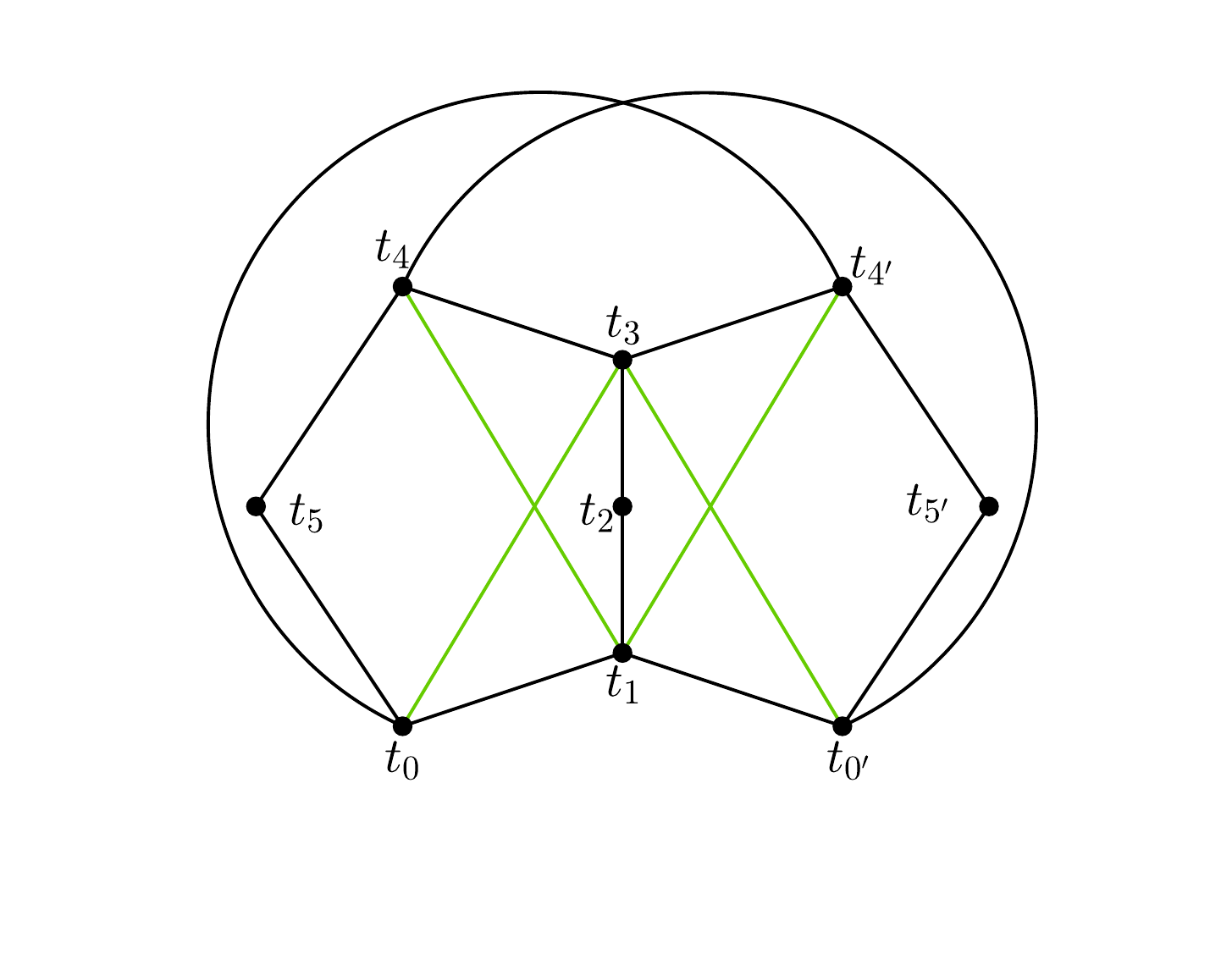}} &
\multicolumn{2}{| m{6.5cm} |}{Claim 6.
\includegraphics[scale=0.38]{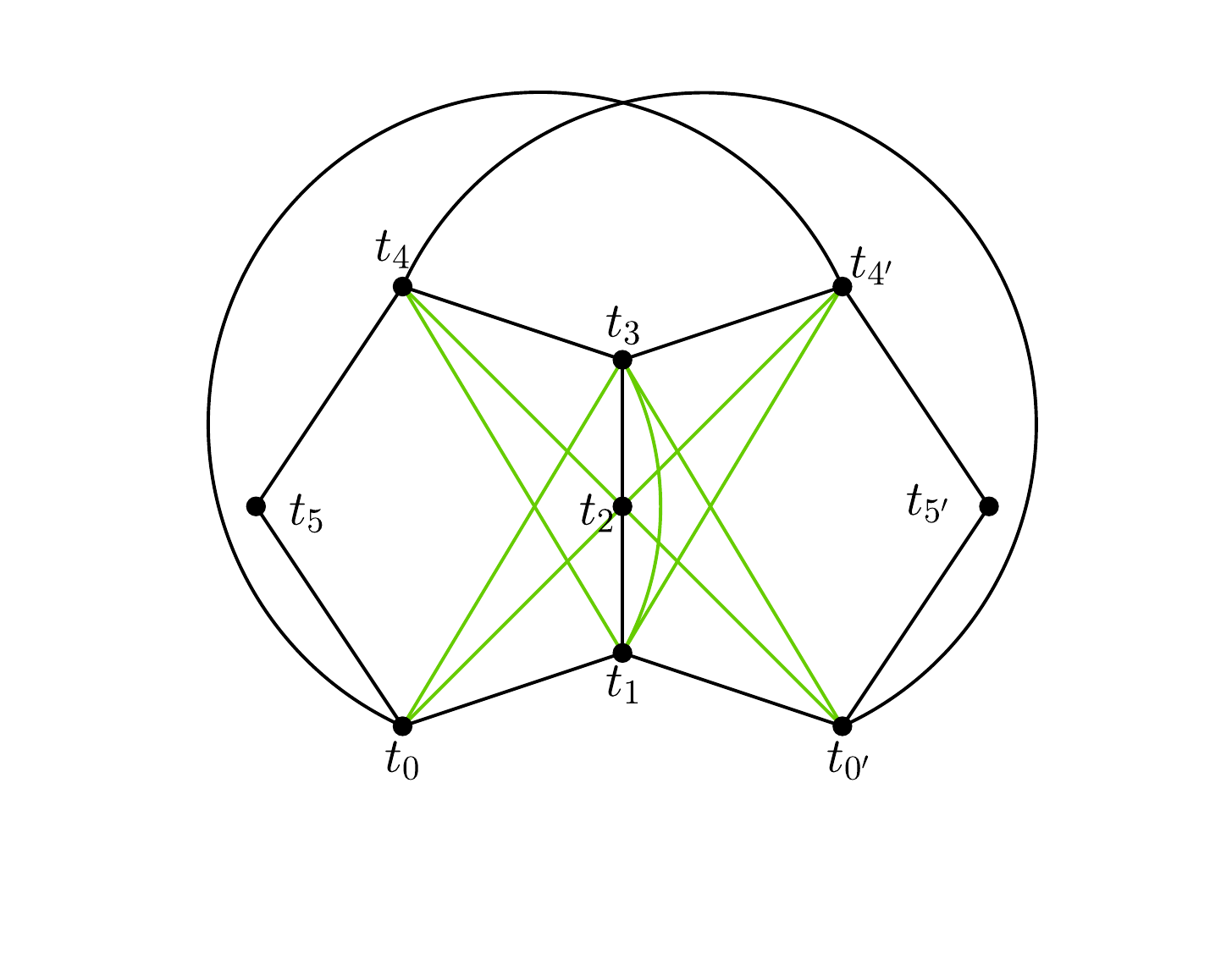}} \\
\hline
\multicolumn{2}{| m{6.5cm} |}{Claim 7.
\vspace{1.52mm}\includegraphics[scale=0.38]{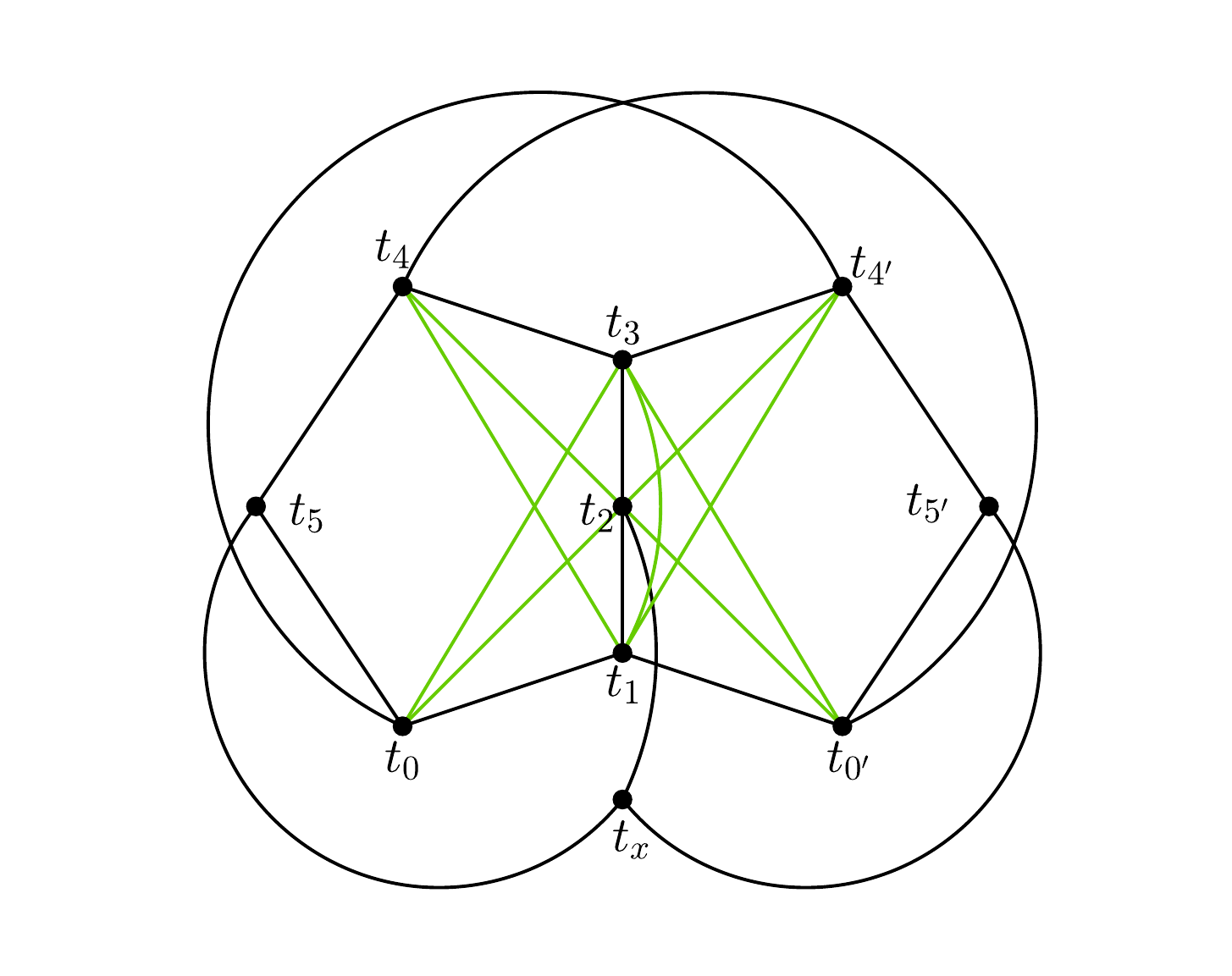}} &
\multicolumn{2}{| m{6.5cm} |}{Claim 8. \vspace{1.52mm}\includegraphics[scale=0.55]{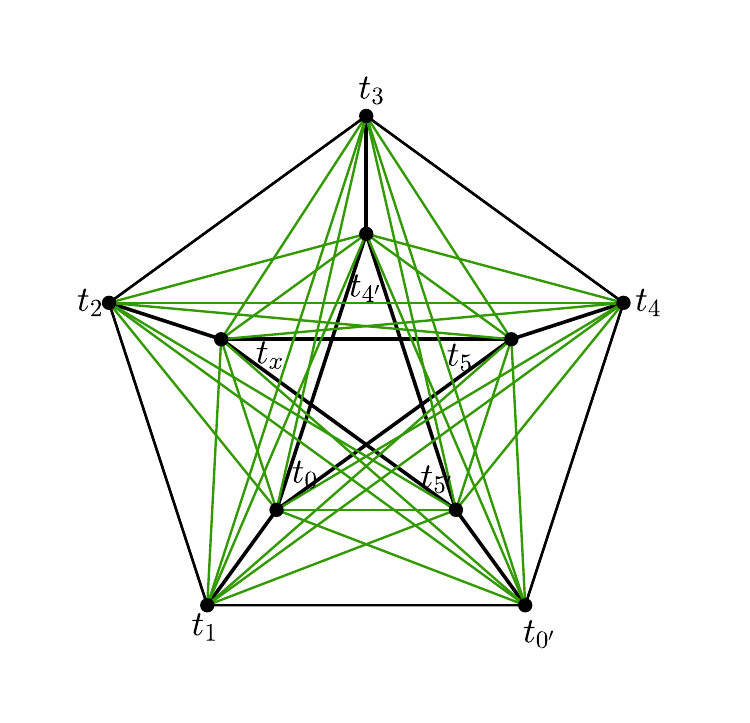}} \\
\hline
\end{tabular}
\end{table}

\clearpage

\section*{Theorem 17}
%\vspace{-0.5cm}
\begin{table}[h!]
\centering
\begin{tabular}{| c | c | c |}
\hline
Claim 1. &
Claim 2. &
Claim 3. \\
\includegraphics[scale=0.35]{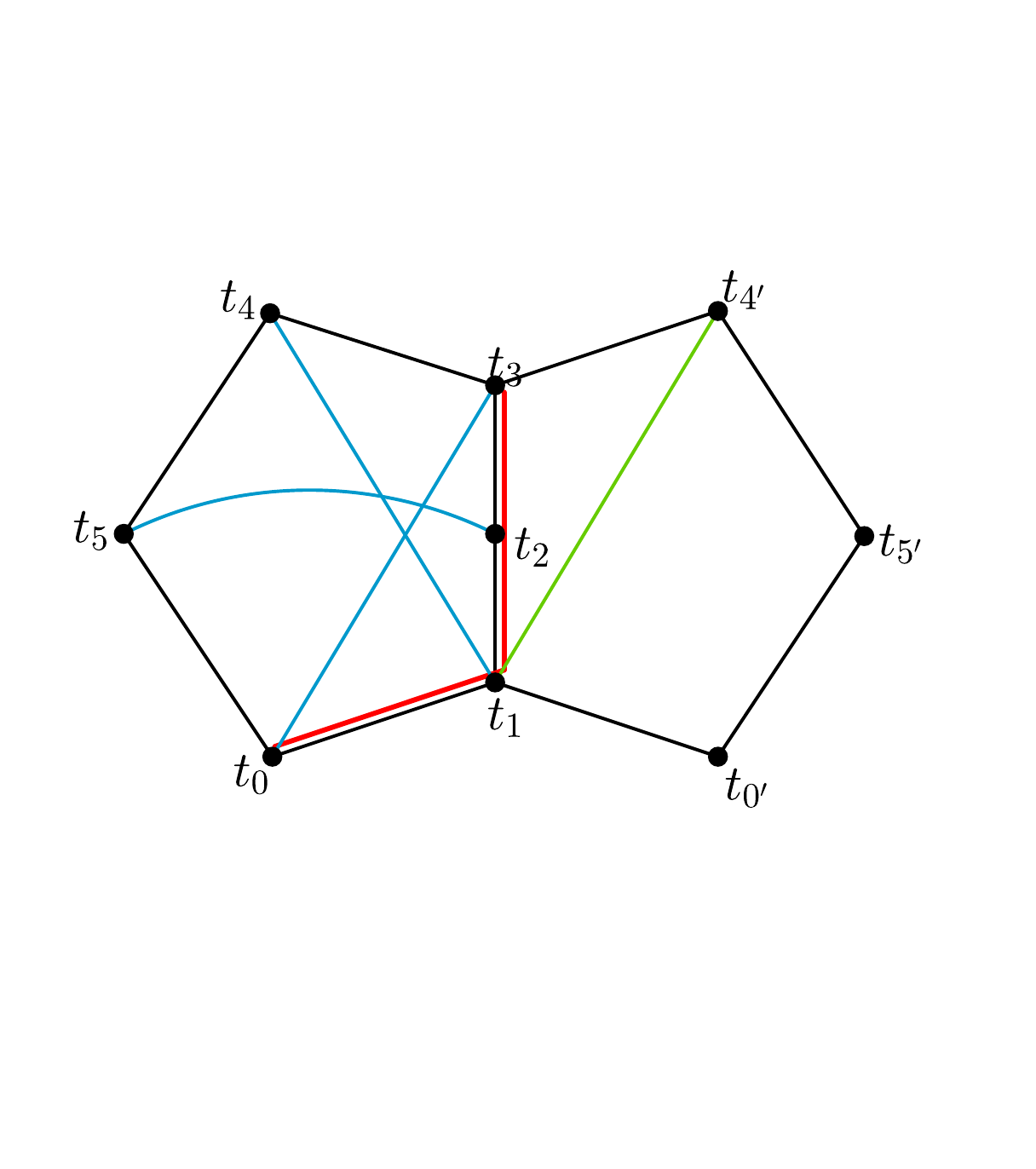} & 
\includegraphics[scale=0.35]{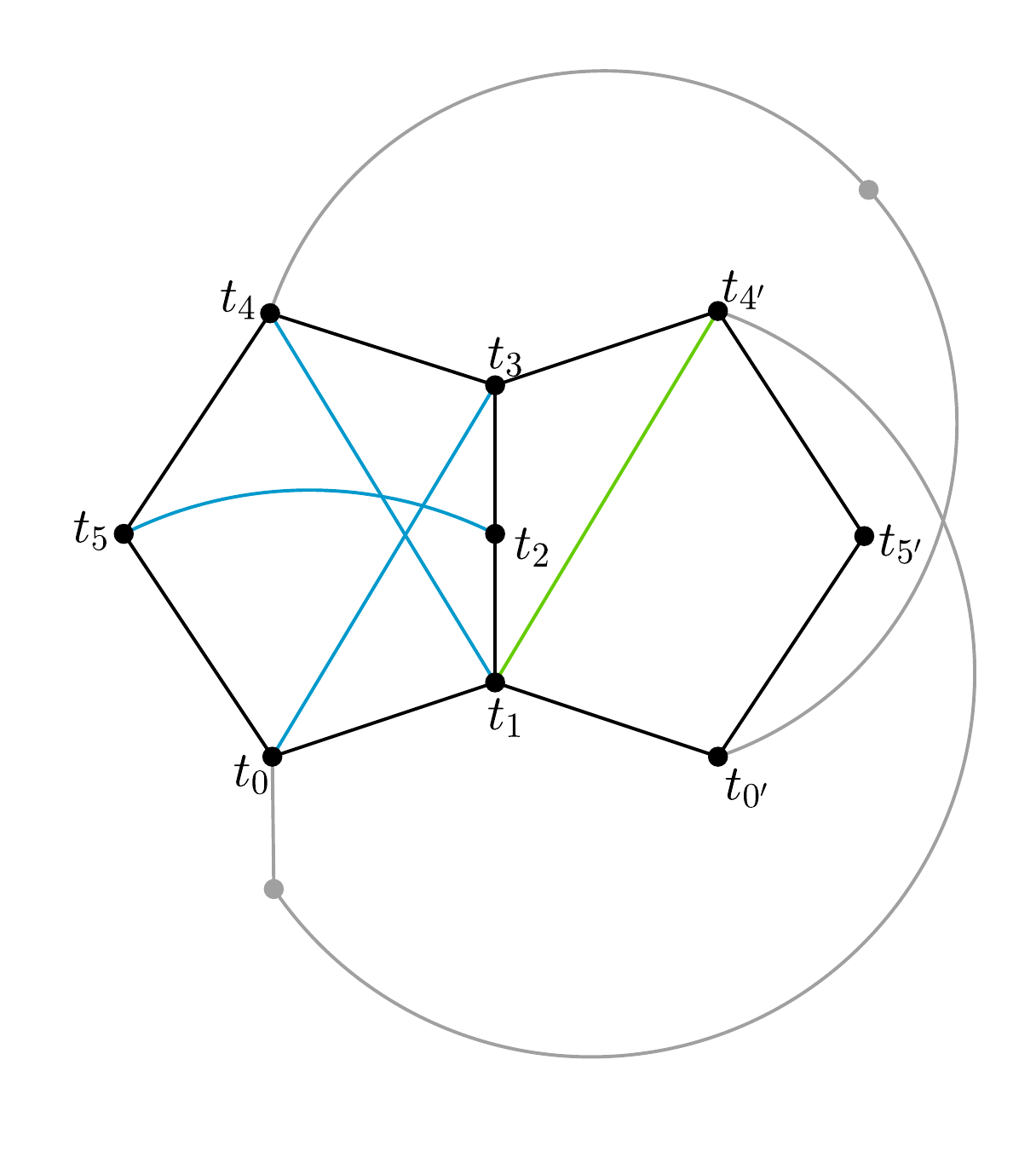} &
\includegraphics[scale=0.35]{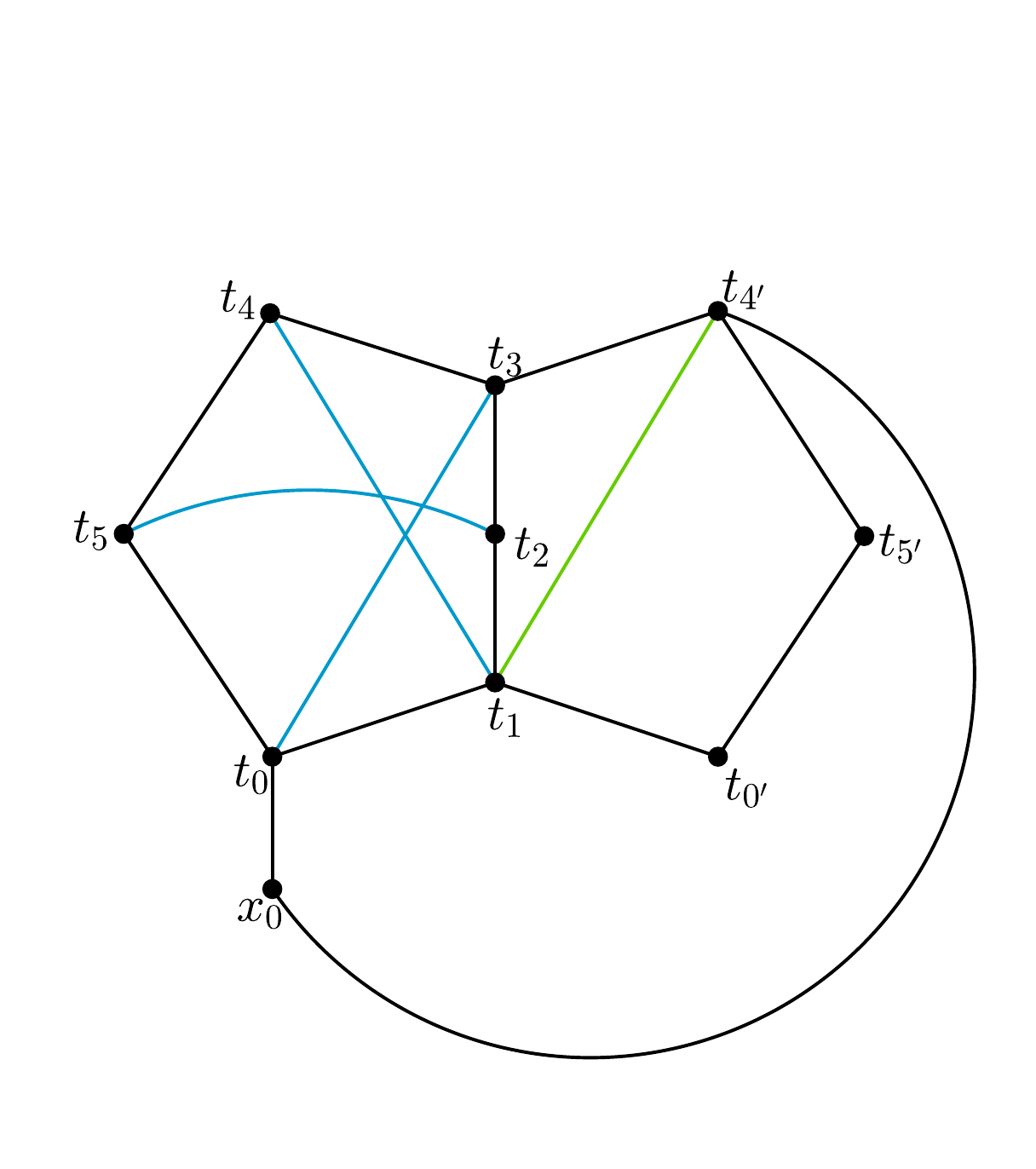} \\
\hline
Claim 4. &
Claim 5. &
Claim 6. \\
\includegraphics[scale=0.35]{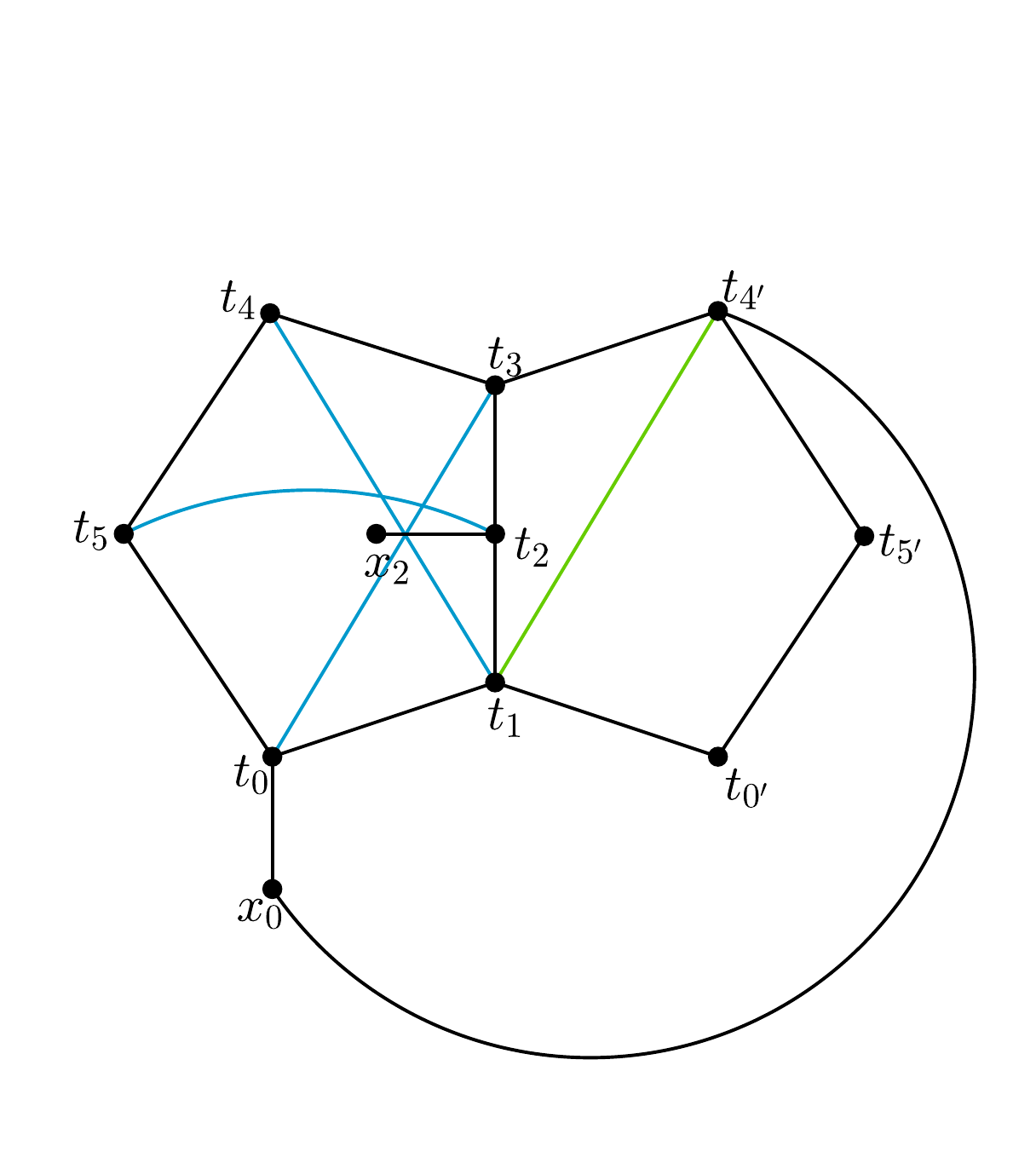} &
\includegraphics[scale=0.35]{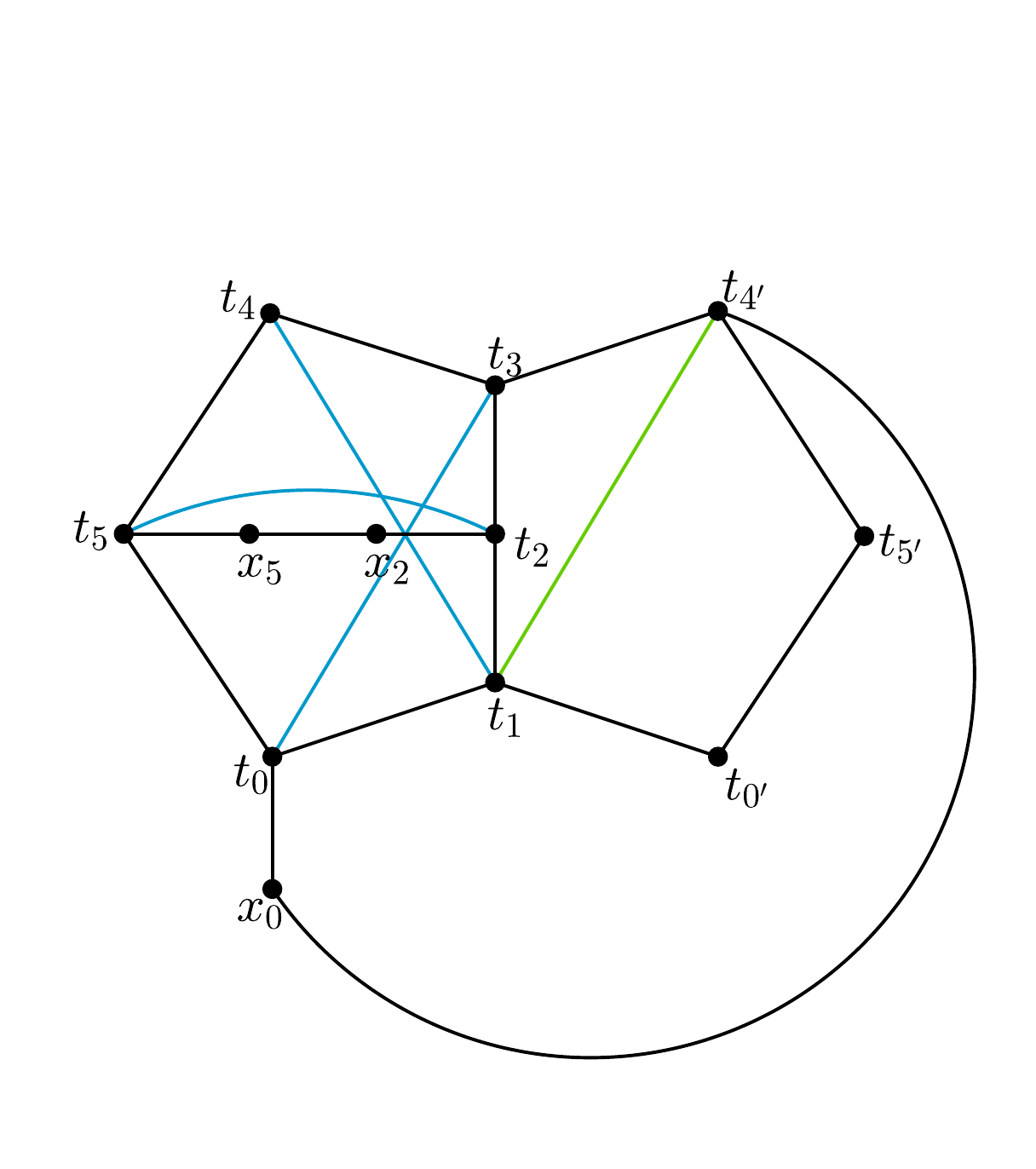} & 
\includegraphics[scale=0.35]{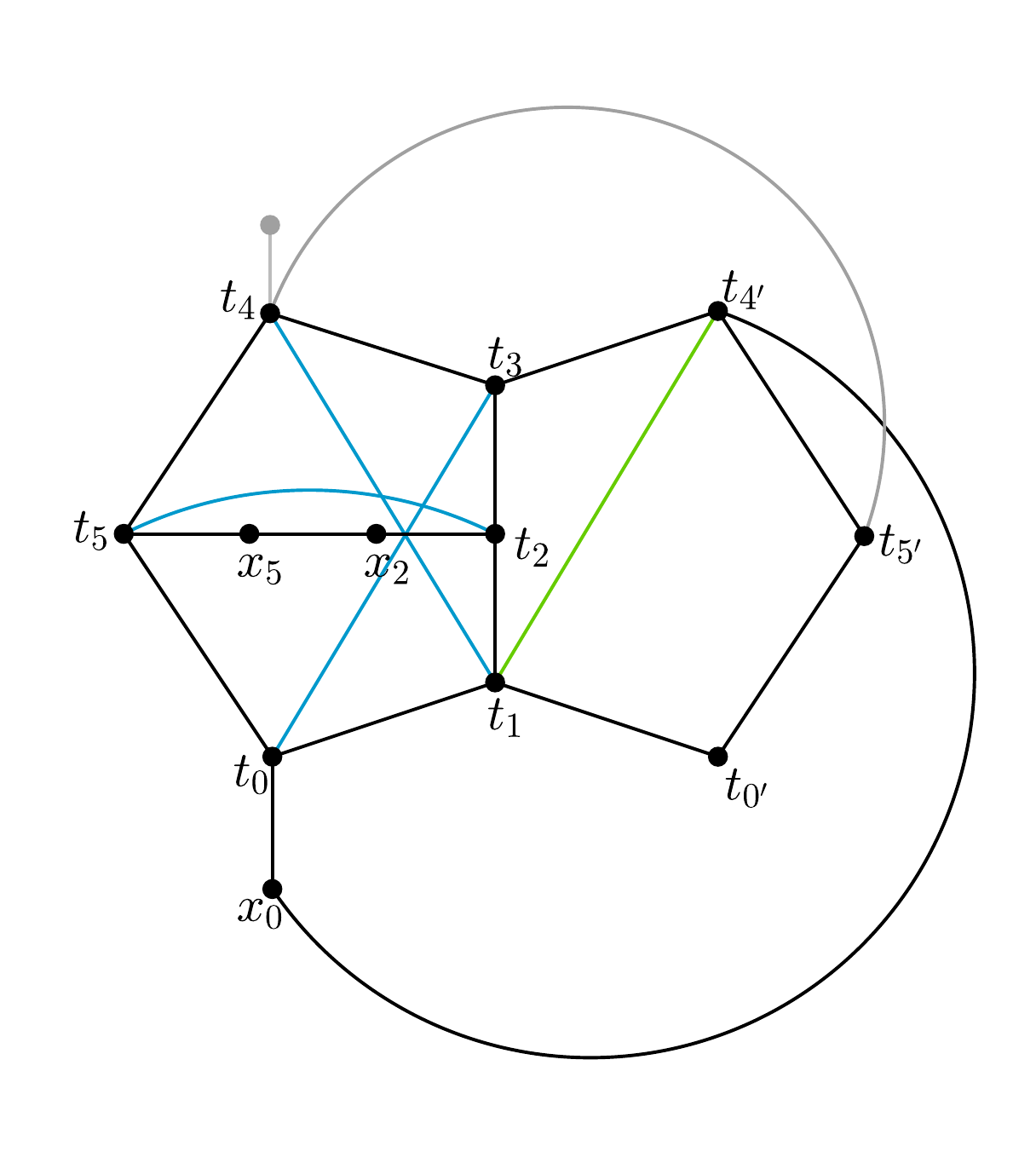} \\
\hline
Claim 7. &
Claim 8. &
Claim 9. \\
\includegraphics[scale=0.7]{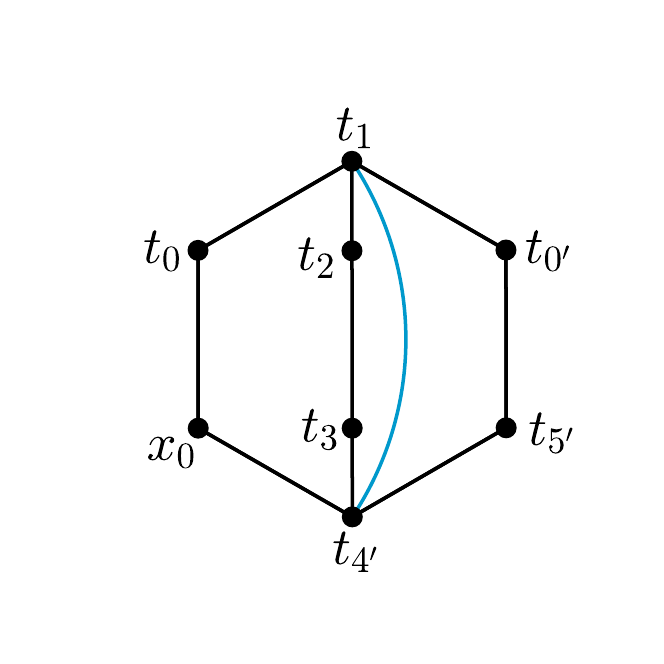} & 
\includegraphics[scale=0.7]{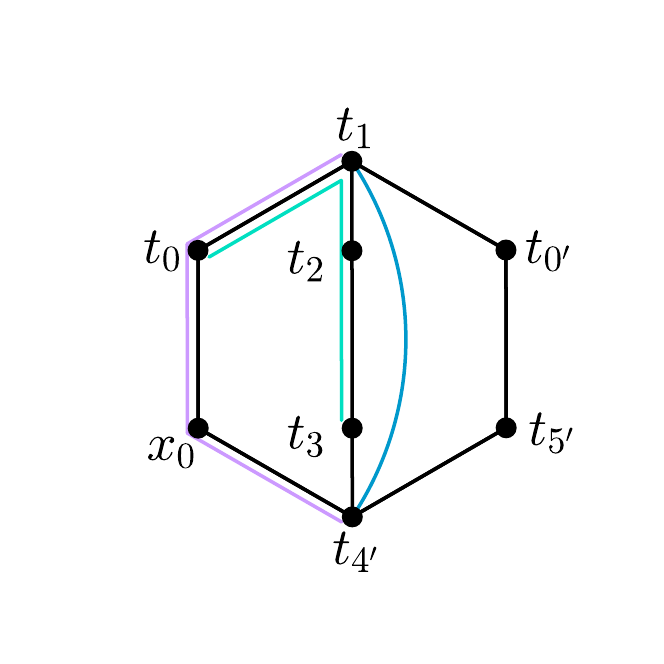} &
\includegraphics[scale=0.55]{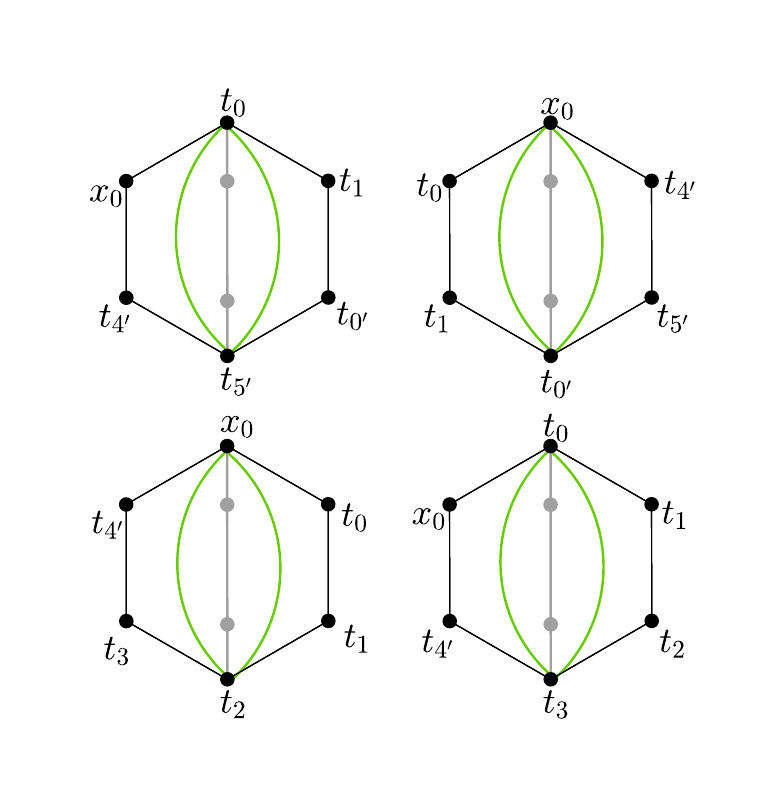} \\
\hline
\end{tabular}
\end{table}

\clearpage

\subsection*{Theorem 17}
\vspace{-0.5cm}
\begin{table}[h!]
\centering
\begin{tabular}{| c | c | c |}
\hline
Claim 10. &
Claim 11. &
Claim 12. \\
\includegraphics[scale=0.7]{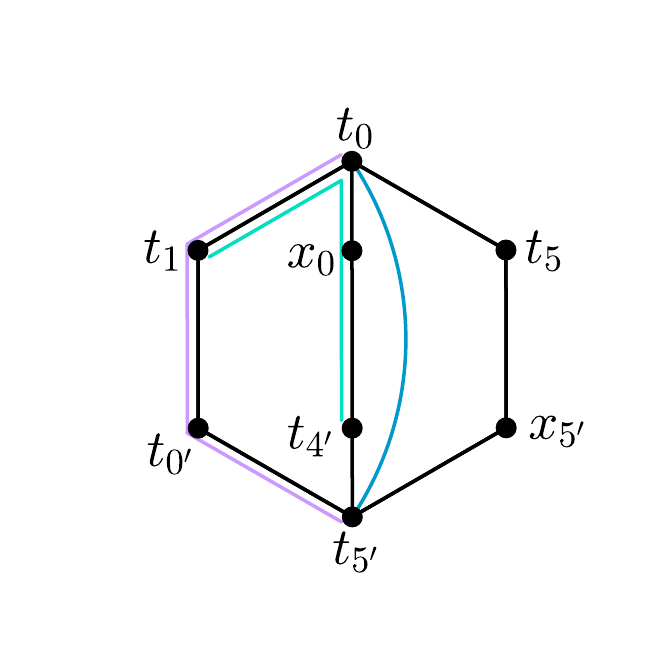} &
\includegraphics[scale=0.7]{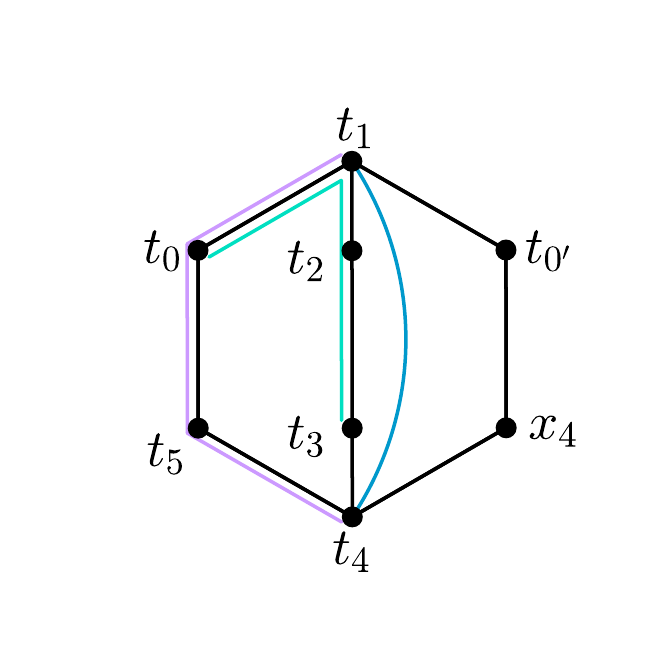} &
\includegraphics[scale=0.35]{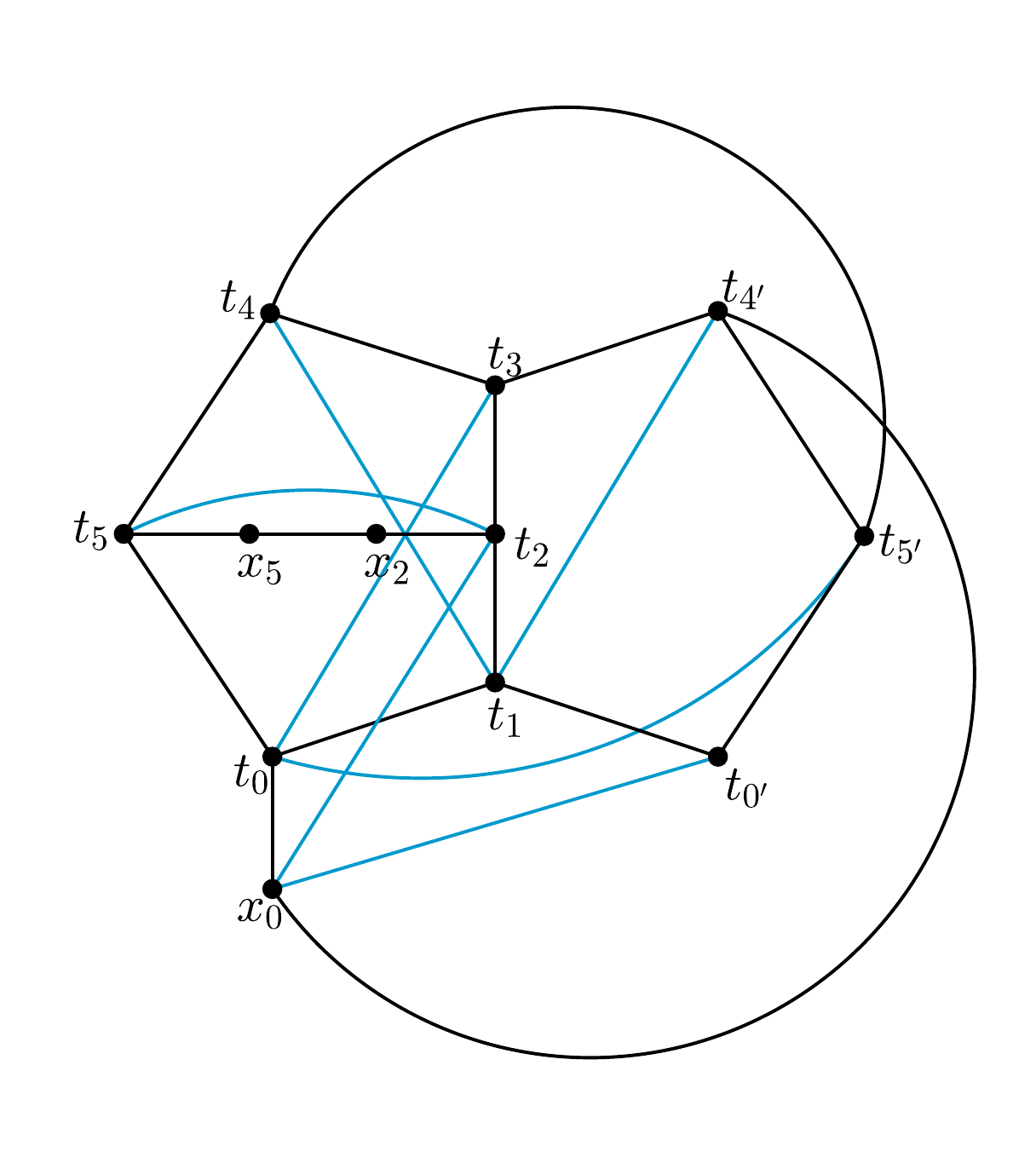} \\
\hline
Claim 13. &
Claim 14. &
Claim 15. \\
\includegraphics[scale=0.7]{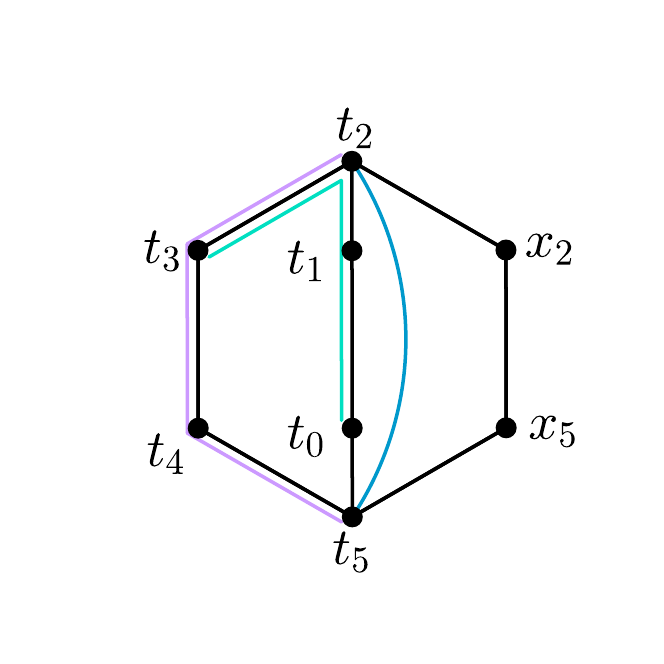} &
\includegraphics[scale=0.55]{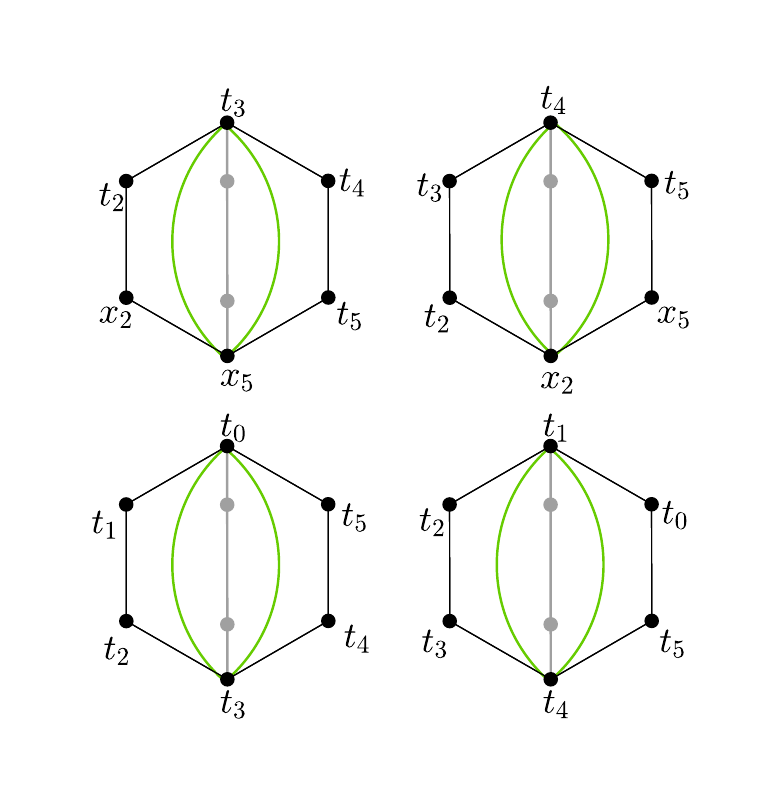} &
\includegraphics[scale=0.35]{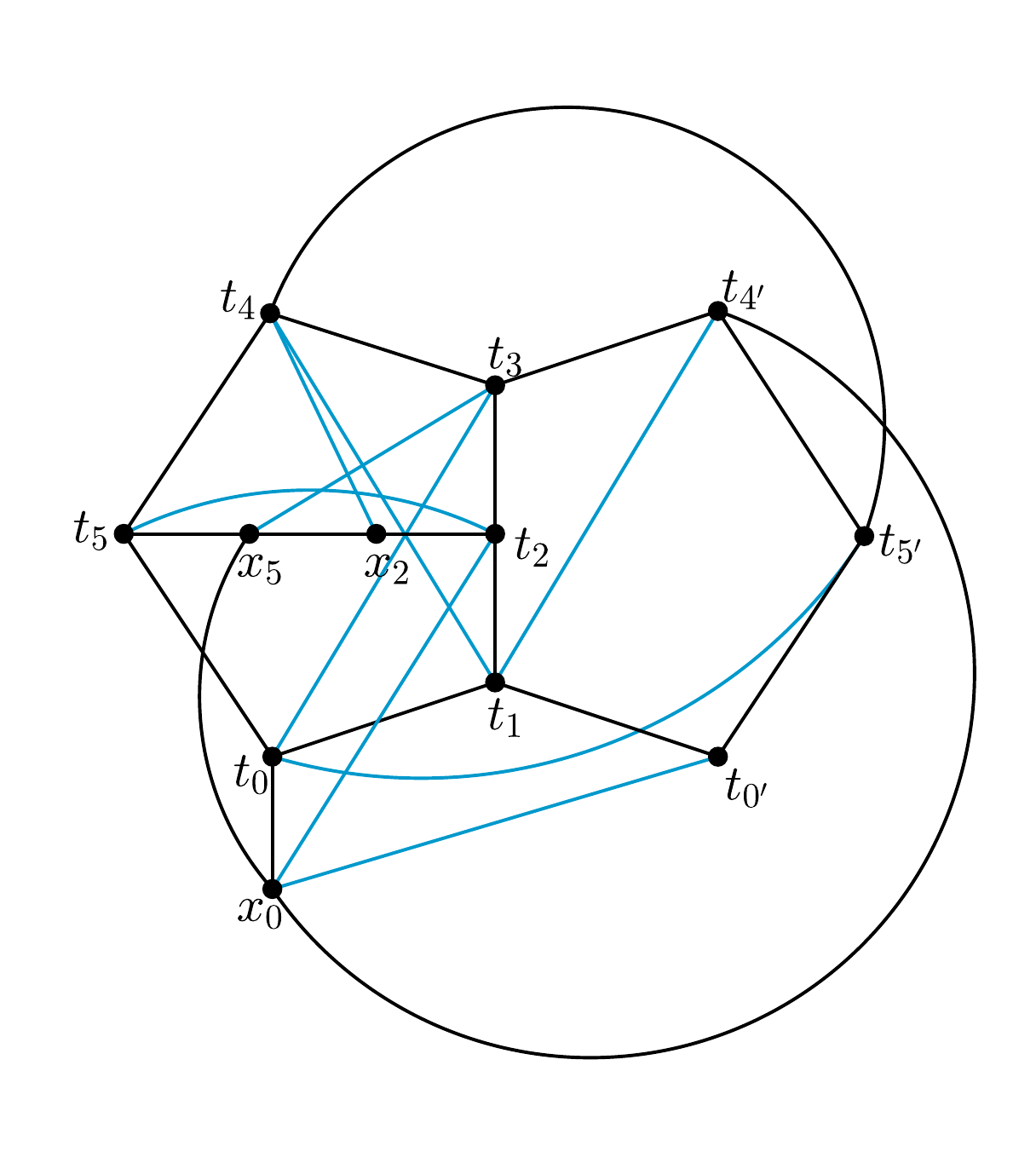} \\
\hline
Claim 16. & 
Claim 17. &
Claim 18. \\
\includegraphics[scale=0.35]{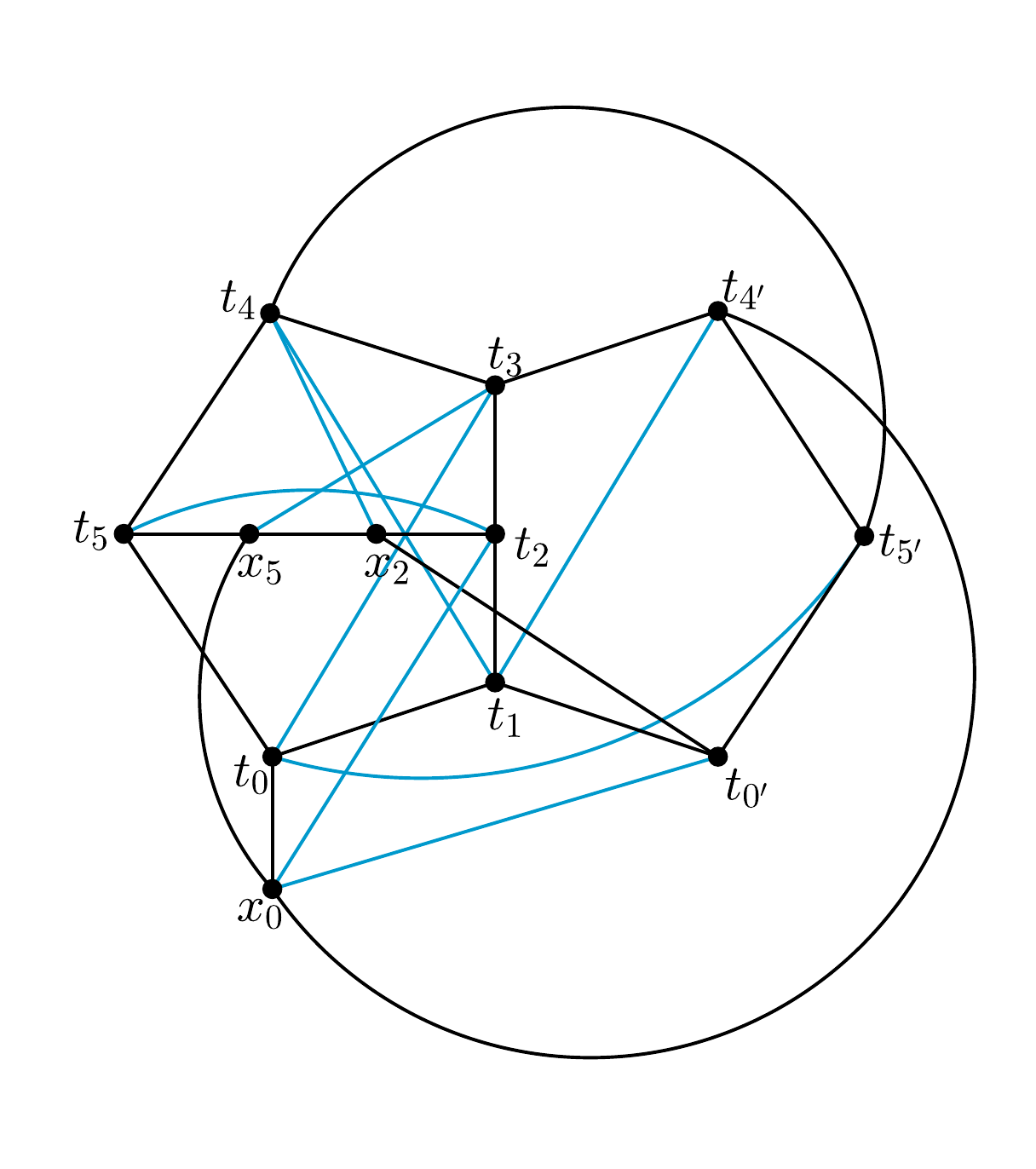}&
\includegraphics[scale=0.35]{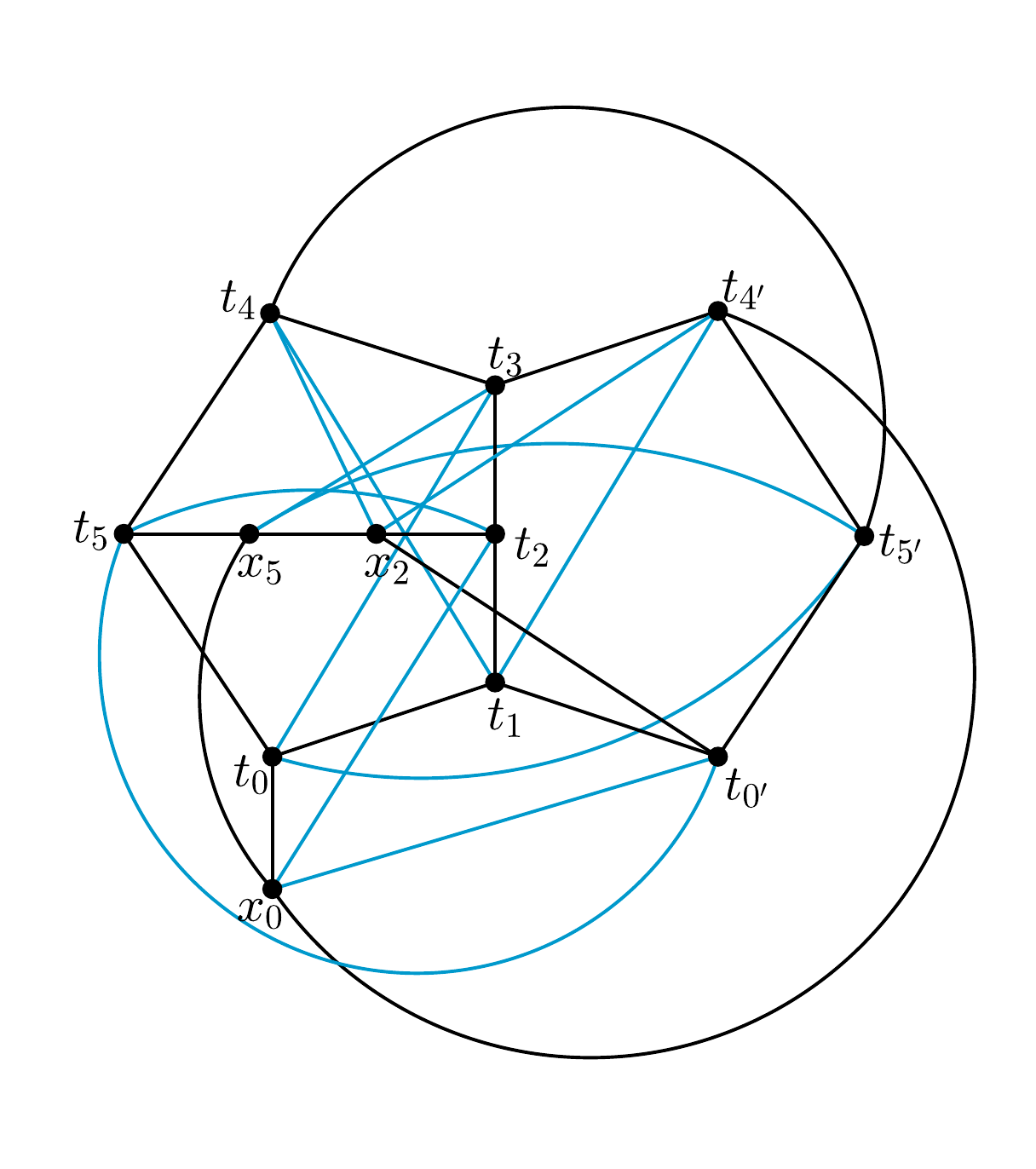} &
\includegraphics[scale=0.7]{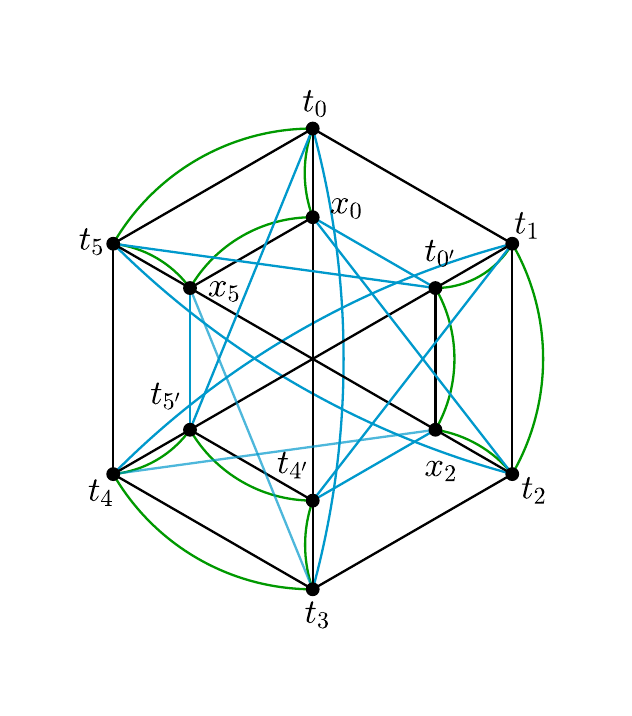}   \\
\hline
\end{tabular}
\end{table}
\vspace{1cm}
\end{appendices}

%%%%%%%%%%%%%%%%%%%%%%%%%%%
\clearpage

\subsection*{Acknowledgements}
Research supported by PAPIIT-M{\' e}xico (Projects IN104915, IN106318), CONACyT-M{\' e}xico (Project 166306) and CIC-UNAM (Academic Exchange Project: “Gráficas y Triangulaciones”. 2019).
.
\end{document}